\documentclass[10pt]{amsart}
\usepackage{amscd}
\usepackage{amsmath}
\usepackage{mathtools}
\usepackage{mathrsfs}
\usepackage{amssymb}
\usepackage[utf8]{inputenc}
\usepackage[T1]{fontenc}

\usepackage{tikz-cd}
\usepackage{hyperref}
\newtheorem{thm}{Theorem}[section]
\newtheorem{prop}[thm]{Proposition}
\newtheorem{lem}[thm]{Lemma}
\newtheorem{cor}[thm]{Corollary}

\theoremstyle{definition}
\newtheorem{definition}[thm]{Definition}
\newtheorem{example}[thm]{Example}

\newtheorem{fact}[thm]{Fact}

\newtheorem{remark}[thm]{Remark}
\newtheorem{conv}[thm]{Convention}
\newtheorem{notation}[thm]{Notation}

\newtheorem{thmintro}{Theorem}
\newtheorem{questintro}{Question}

\makeatletter
\newcommand{\hocolim@}[2]{%
  \vtop{\m@th\ialign{##\cr
    \hfil$#1\operator@font hocolim$\hfil\cr
    \noalign{\nointerlineskip\kern1.5\ex@}#2\cr
    \noalign{\nointerlineskip\kern-\ex@}\cr}}%
}

\newcommand{\dhocolim}{%
  \mathop{\mathpalette\hocolim@{\rightarrowfill@\textstyle}}\nmlimits@
}

\DeclareMathOperator*{\pp}{\mathfrak{p}}
\DeclareMathOperator*{\oo}{\mathfrak{o}}
\DeclareMathOperator*{\qq}{\mathfrak{q}}
\DeclareMathOperator*{\mm}{\mathfrak{m}}
\DeclareMathOperator*{\nn}{\mathfrak{n}}
\DeclareMathOperator*{\Spec}{Spec}
\DeclareMathOperator*{\Prod}{Prod}
\DeclareMathOperator*{\Zgs}{Zg}
\DeclareMathOperator*{\mSpec}{mSpec}

\DeclareMathOperator*{\Cat}{Cat}

\DeclareMathOperator*{\CAT}{CAT}
\DeclareMathOperator*{\Ab}{\mathbf{Ab}}
\DeclareMathOperator*{\Ann}{Ann}

\DeclareMathOperator*{\Soc}{Soc}
\DeclareMathOperator*{\Cogen}{Cogen}

\DeclareMathOperator*{\holim}{holim}
\DeclareMathOperator*{\hocolim}{hocolim}
\DeclareMathOperator*{\colim}{colim}
\DeclareMathOperator*{\Ext}{Ext}
\DeclareMathOperator*{\Tor}{Tor}

\DeclareMathOperator*{\Ker}{Ker}
\DeclareMathOperator*{\Coker}{Coker}

\DeclareMathOperator*{\Hom}{Hom}
\DeclareMathOperator*{\Rop}{\mathit{R}^{op}}
\DeclareMathOperator*{\RHom}{\mathbf{R}Hom}
\DeclareMathOperator*{\ModR}{Mod-\!\mathit{R}}
\DeclareMathOperator*{\ModRop}{Mod-\!\mathit{R}^{op}}

\DeclareMathOperator*{\ModS}{Mod-\!\mathit{S}}
\DeclareMathOperator*{\ModSn}{Mod-\!\mathit{S}_n}
\DeclareMathOperator*{\ModU}{Mod-\!\mathit{U}}
\DeclareMathOperator*{\ModTc}{Mod-\!\mathcal{T}^c}
\DeclareMathOperator*{\ModRqq}{Mod-\!\mathit{R}_{\qq}}
\DeclareMathOperator*{\ModRpp}{Mod-\!\mathit{R}_{\pp}}
\DeclareMathOperator*{\ModRqqprime}{Mod-\!\mathit{R}_{\qq^{\prime}}}
\DeclareMathOperator*{\ModRqqalpha}{Mod-\!\mathit{R}_{\qq_\alpha}}
\DeclareMathOperator*{\ModRqqalphaprime}{Mod-\!\mathit{R}_{\qq^{\prime}_\alpha}}
\DeclareMathOperator*{\ModRqqpp}{Mod-\!(\mathit{R}_{\qq}/\pp)}
\DeclareMathOperator*{\ModRqqprimeppprime}{Mod-\!\mathit{R}_{\qq^{\prime}}/\pp^{\prime}}
\DeclareMathOperator*{\ModRqqppS}{Mod-\!((\mathit{R}_{\qq}/\pp) \otimes_{\mathit{R}} \mathit{S})}
\DeclareMathOperator*{\Mor}{Mor}

\newcommand{\EMP}{\textbf}
\newcommand*{\Perp}[1]{{}^{\perp_{#1}}}
\DeclareMathOperator*{\Der}{\mathbf{D}}

\begin{document}
\title[Definable coaisles over rings of weak dimension one]{Definable coaisles over rings of weak global dimension at most one}
\author{Silvana Bazzoni}
\address{Dipartimento di Matematica, Universit\`{a} di Padova, via Trieste 63, I-35121 Padova - Italy}
\email{bazzoni@math.unipd.it}

\author{Michal Hrbek}
\address{Institute of Mathematics of the Czech Academy of Sciences, \v{Z}itn\'{a} 25, 115 67 Prague, Czech Republic}
\email{hrbek@math.cas.cz}
\thanks{The first named author was partially supported by grants BIRD163492 and DOR1690814 of Padova University. The second named author was partially supported by the Czech Academy of Sciences Programme for research and mobility support of starting researchers, project MSM100191801.}

\begin{abstract}
		In the setting of the unbounded derived category $\Der(R)$ of a ring $R$ of weak global dimension at most one we consider t-structures with a definable coaisle. The t-structures among these which are stable (that is, the t-structures which consist of a pair of triangulated subcategories) are precisely the ones associated to a smashing localization of the derived category. In this way, our present results generalize those of \cite{BS} to the non-stable case. As in the stable case \cite{BS}, we confine for the most part to the commutative setting, and give a full classification of definable coaisles in the local case, that is, over valuation domains. It turns out that unlike in the stable case of smashing subcategories, the definable coaisles do not always arise from homological ring epimorphisms. We also consider a non-stable version of the telescope conjecture for t-structures and give a ring-theoretic characterization of the commutative rings of weak global dimension at most one for which it is satisfied.
\end{abstract}

\maketitle
\setcounter{tocdepth}{1} \tableofcontents
\section*{Introduction}
	An extensive effort has been expended on the study of various subcategories of the unbounded derived category $\Der(R)$ of a ring $R$. Since this category is in most cases too complicated to permit any chance of understanding all of its objects, one can instead attempt to study certain kinds of subcategories with good approximation properties. One source of these is provided by Bousfield localizations, or rather, by taking their kernels. Particularly useful are those localizations which commute with coproducts, these are called the \emph{smashing} localizations because of their origins in algebraic topology. Smashing localizations are abundant as any set of compact objects naturally generates one. Since thick subcategories of compact objects often allow for a full classification (e.g. \cite{Ne}, \cite{Tho}), a particularly desirable situation occurs when any smashing localization is compactly generated. This was formulated by Ravenel \cite{R} as the Telescope conjecture in the case of stable homotopy category of spectra. For derived categories, the Telescope conjecture is known to be false in general, see Keller \cite{Ke}. On the other hand, the Telescope Conjecture was settled in the affirmative for large classes of rings. Here we mention the result of Neeman's \cite{Ne} for commutative noetherian rings and of Krause-\v{S}\v{t}ov\'{i}\v{c}ek \cite{KS} for one-sided hereditary rings. In both works, a classification of the compactly generated localizations is given, where in the first case these are parametrized by the specialization closed subsets of the Zariski spectrum, while in the second case the parametrization is by the universal localizations of the ring in the sense of Schofield. Although the failure of the Telescope Conjecture is usually viewed as a pathological behavior, there are rings for which the Telescope conjecture does not hold in general, but still a full classification of smashing localizations is possible, and a simple ring theoretic criterion is available characterizing when the Telescope conjecture is true. This is a result due to the first author and \v{S}\v{t}ov\'{i}\v{c}ek \cite{BS}:
	
	\begin{thmintro}\label{T:BS}\emph{(\cite[Theorem 3.10, Theorem 6.8, Theorem 7.2]{BS})}
		Let $R$ be a ring of weak global dimension at most one. Then there is a bijection between:
			\begin{enumerate}
					\item[(i)] smashing subcategories of $\Der(R)$,
				\item[(ii)] epiclasses of homological ring epimorphism $R \rightarrow S$.
			\end{enumerate}

			Furthermore, if $R$ is commutative, then the following conditions are equivalent:
			\begin{enumerate}
				\item[(i)] the Telescope Conjecture holds in $\Der(R)$,
				\item[(ii)] any homological ring epimorphism $R \rightarrow S$ is flat,
				\item[(iii)] for any prime ideal $\pp$ of $R$, the prime ideal $\pp R_{\pp}$ is idempotent only if it is zero in $R_{\pp}$.
			\end{enumerate}
	\end{thmintro}

A more general supply of subcategories of triangulated categories inducing nice approximations is provided by the notion of a t-structure, as introduced by Be\u{ı}linson, Bernstein, and Deligne in \cite{BBD}. By definition, a t-structure is an orthogonal pair of full subcategories, called usually the aisle and the coaisle, satisfying some axioms ensuring a behavior similar to that of a torsion pair in an abelian category. In case these subcategories are themselves triangulated, the t-structure is called \EMP{stable}, and the aisles of stable t-structures are precisely the kernels of Bousfield localizations. The smashing property generalizes easily to t-structures - it simply requires the coaisle to be closed under coproducts. However, unlike in the case of stable t-structures, the smashing property for t-structures is too weak to allow for a classification even in basic cases; for example, there is a proper class of smashing t-structures over the ring of integers, cf. \cite{GSh} together with \cite[Example 6.2]{SSV}. Instead, more restrictive conditions were considered for t-structures forming the following hierarchy:
	\begin{center}
	\begin{tabular}{c}
		$\{\text{compactly generated t-structures}\}$ \\
		$\rotatebox{90}{$\supseteq$}$ \\
		$\{\text{t-structures with definable coaisle}\}$ \\
		$\rotatebox{90}{$\supseteq$}$ \\
		$\{\text{homotopically smashing t-structures}\}$ \\
		$\rotatebox{90}{$\supseteq$}$ \\
		$\{\text{smashing t-structures}\}$ \\
	\end{tabular}
	\end{center}
The subtlety of this hierarchy is only appreciated in case of non-stable t-structures, as the latter three conditions collapse in the stable case, as shown by Krause \cite{Kr}. Homotopically smashing t-structures were introduced in \cite{SSV}, where the authors prove that the heart of such t-structures has exact direct limits, and is even a Grothendieck category under mild conditions. An \textit{a priori} stronger condition is to require the coaisle to be a definable subcategory. This condition was recently considered in the setting of silting theory in compactly generated triangulated categories, see e.g. \cite{MV}, \cite{AMV2}, \cite{L}. In particular, Laking \cite{L} proved that under mild assumptions, a left non-degenerate t-structure is induced by a pure-injective cosilting object if and only if its coaisle is definable. Therefore, a classification of t-structures with definable coaisles yields a description of pure-injective cosilting objects in $\Der(R)$ up to equivalence. Furthermore, it is also proved in \cite{L} that for left non-degenerate t-structures, the homotopically smashing property is actually equivalent to the coaisle being definable. Not before this paper was submitted, Saor\'{i}n and \v{S}\v{t}ov\'{i}\v{c}ek \cite[Remark 8.9]{SS2} employed a result on cotorsion pairs of \v{S}aroch \cite[Theorem 6.1]{Sar} to show that the coaisle of an arbitrary homotopically smashing t-structure in any algebraic compactly generated triangulated category is definable. In particular, the two middle classes of the hierarchy above collapse in the case of $\Der(R)$, and as a consequence, our results could be formulated equally for homotopically smashing t-structures.

	The following question comes naturally as a strengthening of the Telescope Conjecture from the stable case to general t-structures.
	\begin{questintro}\label{Q:UTC}
			For which rings is it true that every t-structure in $\Der(R)$ with a definable coaisle is compactly generated?
	\end{questintro}
	In particular, an affirmative answer to Question~\ref{Q:UTC} implies, in light of \cite{MV}, that all t-structures induced by bounded cosilting complexes over $R$ are compactly generated, a sort of cofinite type result in silting theory. Very recently, it was shown that commutative noetherian rings \cite{HN} and one-sided hereditary rings \cite[Theorem 3.11]{AH} are among answers to Question~\ref{Q:UTC}, generalizing the two results about Telescope Conjecture cited above. One of the goals of this paper is to consider Question~\ref{Q:UTC} for rings of weak global dimension at most one and give an analog of Theorem~\ref{T:BS} for t-structures which are not necessarily stable. In particular, we prove the following:

\begin{thmintro}{(Theorem~\ref{T:gtc})}
	Let $R$ be a commutative ring of weak global dimension at most one. Then Question~\ref{Q:UTC} has a positive answer for $R$ if and only if the Telescope conjecture holds in $\Der(R)$.
\end{thmintro}

The first step in this direction will be to prove that any definable subcategory of the derived category of a (not necessarily commutative) ring of weak global dimension at most one is determined on the cohomology (Theorem~\ref{T:cohomology}). This generalizes the results for hereditary and von Neumann regular rings of Garkusha-Prest \cite{GP}. Similarly to the case of smashing subcategories in \cite{BS}, the K\"{u}nneth formula plays an essential role in the proof Theorem~\ref{T:cohomology}. This reduction to cohomology classes will readily allow us to answer Question~\ref{Q:UTC} in the affirmative for not necessarily commutative von Neumann regular rings (Corollary~\ref{C:VNR}; also see the references preceding it).

	As in \cite{BS}, we then switch our focus to the commutative case. The basis for our findings is the structure of compactly generated t-structures which were described in terms of certain filtrations of the Zariski spectrum in \cite{AJS} and this was further generalized to not necessarily noetherian rings in \cite{HCG}; also see \cite{Stl} for a different but related kind of result. The property of being compactly generated localizes well, and this allows us to consider the commutative rings of weak global dimension at most one locally - that is, to confine to valuation domains. As one of the main results of this paper, we give a full classification of t-structures with definable coaisles over valuation domains (Theorem~\ref{T02}) by establishing a bijective correspondence between them and certain invariants defined on the Zariski spectrum which we call ``admissible filtrations''. 

Given a valuation domain $R$, an admissible filtration is an integer-indexed sequence of systems of formal intervals in $\Spec(R)$ satisfying certain axioms. Such systems were already used in the classification of the smashing subcategories in \cite{BS}, as well as in the study of cotilting modules in \cite{B} and \cite{B2}. In the stable case in \cite{BS}, the bijective correspondence was established between smashing localizations and admissible systems satisfying a condition of being ``nowhere dense''. However, as shown in \cite[Example 5.1]{B2}, there are cotilting modules which correspond to an admissible system which is not nowhere dense. Cotilting modules naturally give raise to Happel-Reiten-Smal\o~ t-structures with definable coaisles --- this suggests that the classification for general t-structure should be in terms of sequences of admissible systems satisfying the non-density condition only locally in some sense, with respect to cohomological degrees. This is indeed the case, see Definition~\ref{D:admissiblefiltration}. 

	The possibility of having dense intervals in the members of the admissible filtration is connected to a new phenomenon which was not visible in case of stable t-structures. Given a ring $R$ of weak global dimension at most one, not all t-structures with definable coaisle can be described in terms of homological ring epimorphisms. More precisely, in \cite[\S 5]{AH}, Angeleri H\"{u}gel and the second author establish an \textit{injective} map of the following form (the statement will be made precise in the body of this paper, see \ref{SS:homological}):
		\begin{equation}\label{Easgn} \tag{$\star$}\left \{ \begin{tabular}{ccc} \text{ $\mathbb{Z}$-indexed chains of } \\ \text{homological ring epimorphisms} \\ \text{of $R$ up to equivalence } \end{tabular}\right \}  \xhookrightarrow{}  \left \{ \begin{tabular}{ccc} \text{ t-structures with definable } \\ \text{ coaisles in $\Der(R)$ } \end{tabular}\right \}.\end{equation}

				By Theorem~\ref{T:BS}, the image of the assignment (\ref{Easgn}) always contains all of the t-structures which are in addition stable. However, already over the Kronecker algebra over a field, the assignment is not \textit{surjective}, as it misses precisely all the shifts of the Happel-Reiten-Smal\o~ t-structure associated to the dual of the Lukas tilting module \cite[\S 6.4]{AH}. In this paper, we will demonstrate that over valuation domains, the map (\ref{Easgn}) can potentially miss a lot of t-structures, provided that the Zariski spectrum of the domain is ``topologically rich enough'' to allow density to occur in an admissible filtration, see Remark~\ref{R:dense}. We compile the relations between the various kinds of t-structures and invariants over a valuation domain $R$ established in Sections 8 and 9 in the following commutative diagram of 1-1 correspondences and inclusions.
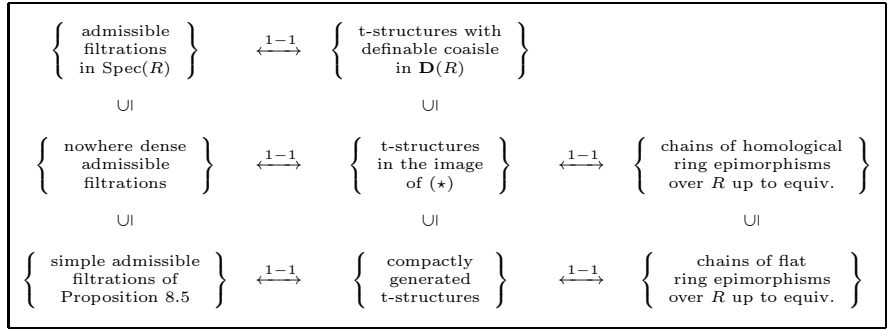
\begin{center}\tiny
	\begin{tabular}{|ccccc|}
		\hline
		& & & & \\
		$\left \{  \begin{tabular}{c} \text{admissible} \\ \text{filtrations} \\ \text{in $\Spec(R)$} \end{tabular} \right \}$ & $\xleftrightarrow{1-1}$ & $\left \{  \begin{tabular}{c} \text{t-structures with} \\ \text{definable coaisle} \\ \text{in $\Der(R)$} \end{tabular} \right \}$ & & \\
		& & & & \\
		$\rotatebox{90}{$\subseteq$}$	& & $\rotatebox{90}{$\subseteq$}$ & & \\
		& & & & \\
		$\left \{  \begin{tabular}{c} \text{nowhere dense} \\ \text{admissible} \\ \text{filtrations} \end{tabular} \right \}$ & $\xleftrightarrow{1-1}$ & $\left \{  \begin{tabular}{c} \text{t-structures} \\ \text{in the image} \\ \text{of (\ref{Easgn})} \end{tabular} \right \}$ & $\xleftrightarrow{1-1}$ & $\left \{  \begin{tabular}{c} \text{chains of homological} \\ \text{ring epimorphisms} \\ \text{over $R$ up to equiv.} \end{tabular} \right \}$ \\
		& & & & \\
		$\rotatebox{90}{$\subseteq$}$	& & $\rotatebox{90}{$\subseteq$}$ & & $\rotatebox{90}{$\subseteq$}$ \\
		& & & & \\
		$\left \{  \begin{tabular}{c} \text{simple admissible} \\ \text{filtrations of} \\ \text{Proposition~\ref{P:cg}} \end{tabular} \right \}$ & $\xleftrightarrow{1-1}$ & $\left \{  \begin{tabular}{c} \text{compactly} \\ \text{generated} \\ \text{t-structures} \end{tabular} \right \}$ & $\xleftrightarrow{1-1}$ & $\left \{  \begin{tabular}{c} \text{chains of flat} \\ \text{ring epimorphisms} \\ \text{over $R$ up to equiv.} \end{tabular} \right \}$ \\
		& & & & \\
		\hline
	\end{tabular}
\end{center}

	The paper is concluded by discussing the non-degeneracy condition of the classified t-structures. In general, there is the following chain of conditions that one can impose on a definable coaisle in the derived category of any ring $R$:

	\begin{center}
	\begin{tabular}{c}
		$\{\text{co-intermediate definable coaisles}\}$ \\
		$\rotatebox{90}{$\supseteq$}$ \\
		$\{\text{non-degenerate definable coaisle}\}$ \\
		$\rotatebox{90}{$=$}$ \\
		$\{\text{non-degenerate coaisles of homotopically smashing t-structures}\}$ \\
	\end{tabular}
	\end{center}
 
By results of \cite{MV} and \cite{L}, respectively, the smallest class above corresponds to equivalence classes of (bounded) cosilting complexes over $R$, while the two larger classes coincide and correspond to equivalence classes of pure-injective cosilting objects. The usual definition (\cite{WZ}, \cite{MV}) demands the cosilting complex to be a bounded complex of injective $R$-modules, ensuring that the induced t-structure is co-intermediate, while one makes no such assumption when defining a general cosilting object (\cite{NSZ}, \cite{PV}) in a triangulated category. We use our classification over a valuation domain to show that while any pure-injective cosilting object is in this setting cohomologically bounded below (Corollary~\ref{C:bbelow}), the induced t-structure may not be co-intermediate in general (Example~\ref{EX1}), obtaining that the inclusion of the two classes above is strict.

	The paper is organized as follows. Sections 1 and 2 compile the necessary facts and recent results about t-structures, definability in triangulated categories, homotopy colimits, and the cosilting theory. This is done in the generality of triangulated categories which underlie a compactly generated Grothendieck derivator. In Section 3 we study the definable subcategories of the unbounded derived category of a ring of weak global dimension at most one, showing in particular that these are determined on cohomology (Theorem~\ref{T:cohomology}). The definable coaisles are parametrized by certain increasing sequences of definable subcategories of the module category (Proposition~\ref{P:moduletheoretic}). As a consequence, we answer Question~\ref{Q:UTC} in the affirmative for von Neumann regular rings (Corollary~\ref{C:VNR}). After that, we confine to the case of a valuation domain, and give a full classification of the module-theoretic cosilting classes via ``admissible systems'' of intervals in Section 4, mainly Theorem~\ref{T01}. Section 5 introduces the topological notion of (non-)density, computes the cosilting classes by homological formulas (Proposition~\ref{P:tor}), and provides a construction of dense-everywhere cotilting modules (Proposition~\ref{P:cotiltingdense}), which is needed for the sequel. In the next two Sections 6 and 7, the assignments between definable coaisles in the derived category and the ``admissible filtrations'' on the Zariski spectrum are established (Proposition~\ref{P:coaislestointervals} and Proposition~\ref{P:intervalstocoaisles}). In Section 8 we prove that these assignments are mutually inverse, and thus induce the promised bijective correspondence (Theorem~\ref{T02}). Finally, in the last Section 9 we show that the condition of being ``nowhere dense'' of the admissible filtrations corresponds precisely to the t-structure being induced by a chain of homological ring epimorphisms via (\ref{Easgn}) (Theorem~\ref{T:homepi}) and conclude with several examples. 

	\subsection*{Acknowledgement} The authors would like to thank the anonymous referee for many valuable suggestions, one of which helped us discover a mistake in an earlier version of the manuscript. This led us to the correct notion of the degreewise non-density condition (Definition~\ref{D:admissiblefiltration}) which in turn allowed us to describe the definable coaisles as tensor-orthogonal classes (Proposition~\ref{P:givenbytensor}).

	\subsection*{Conventions}Throughout the paper, all subcategories are strict, full and additive, and all functors are additive.

	Unless specified, by a module we always mean a right module over a ring $R$, and the category of right $R$-modules will be denoted as $\ModR$, while the category of abelian groups is denoted as $\Ab$. The chain complexes of $R$-modules are written in the cohomological notation, that is, the degree increases along the differential.
\section{t-structures with definable coaisles}
Let $\mathcal{T}$ be a triangulated category with all small coproducts. We will always denote the suspension functor of $\mathcal{T}$ by $[1]$, and the cosuspension functor by $[-1]$. A \EMP{t-structure} (\cite{BBD}) in $\mathcal{T}$ is a pair $\mathbf{t} = (\mathcal{U},\mathcal{V})$ of subcategories satisfying the following three conditions:
\begin{enumerate}
	\item[(i)] $\Hom_\mathcal{T}(U,V) = 0$ for all $U \in \mathcal{U}$ and $V \in \mathcal{V}$,
	\item[(ii)] for each object $X \in \mathcal{T}$ there is a triangle
			$$U \rightarrow X \rightarrow V \rightarrow U[1]$$
				with $U \in \mathcal{U}$ and $V \in \mathcal{V}$, and
		\item[(iii)] $\mathcal{U}[1] \subseteq \mathcal{U}$, or equivalently, $\mathcal{V}[-1] \subseteq \mathcal{V}$.
\end{enumerate}

The subcategory $\mathcal{U}$ is called the \EMP{aisle} of the t-structure $\mathbf{t}$, and the subcategory $\mathcal{V}$ is the \EMP{coaisle} of $\mathbf{t}$. We will call a subcategory of $\mathcal{T}$ an aisle if it fits as an aisle into a t-structure, and the same custom will be used for coaisles. Given a subcategory $\mathcal{C}$ of $\mathcal{T}$ we adopt the notation
$$\mathcal{C}\Perp{0} = \{X \in \mathcal{T} \mid \Hom{}_\mathcal{T}(C,X) = 0 ~\forall C \in \mathcal{C}\},$$
and
$$\Perp{0}\mathcal{C} = \{X \in \mathcal{T} \mid \Hom{}_\mathcal{T}(X,C) = 0 ~\forall C \in \mathcal{C}\}.$$
It is not hard to see that conditions $(i)$ and $(ii)$ imply that $\mathcal{U} = \Perp{0}\mathcal{V}$ and $\mathcal{V} = \mathcal{U}\Perp{0}$. As a consequence, all aisles and all coaisles are closed under extensions and direct summands in $\mathcal{T}$. Moreover, any aisle is closed under all coproducts in $\mathcal{T}$, and any coaisle is closed under all products existing in $\mathcal{T}$.

The triangle from condition $(ii)$ is unique up to a unique isomorphism of triangles. Indeed, it is always isomorphic to the triangle
			$$\tau_\mathcal{U}(X) \rightarrow X \rightarrow \tau_\mathcal{V}(X) \rightarrow \tau_\mathcal{U}(X)[1],$$
where $\tau_\mathcal{U}$ and $\tau_\mathcal{V}$ are the right and left adjoint to the inclusions $\mathcal{U} \subseteq \mathcal{T}$ and $\mathcal{V} \subseteq \mathcal{T}$, respectively. Moreover, the existence of (any of) these adjoints under $(i)$ and $(iii)$ is equivalent to condition $(ii)$ (see \cite[Proposition 1.1]{KV}). We will call the triangle from $(ii)$ the \EMP{approximation triangle} of $X$ with respect to the t-structure $\mathbf{t}$. We will be especially interested in the reflection functor $\tau_\mathcal{V}$, which will be called the \EMP{coaisle approximation functor}.
\subsection{Definability in compactly generated triangulated categories} Recall that an object $C \in \mathcal{T}$ is \EMP{compact} if the functor $\Hom_\mathcal{T}(C,-)$ sends coproducts in $\mathcal{T}$ to coproducts in $\mathbf{Ab}$, and let $\mathcal{T}^c$ denote the triangulated subcategory of all compact objects of $\mathcal{T}$. From now on we will assume that $\mathcal{T}$ is a \EMP{compactly generated} triangulated category, that is, that $\mathcal{T}$ has small coproducts, $\mathcal{T}^c$ is skeletally small, and that $\Hom_\mathcal{T}(\mathcal{T}^c,X) = 0$ implies $X = 0$ for any $X \in \mathcal{T}$. For any category $\mathcal{C}$, let $\Mor(\mathcal{C})$ denote the morphism category of $\mathcal{C}$.
\begin{definition}
We say that a subcategory $\mathcal{C}$ of $\mathcal{T}$ is \EMP{definable} if there is a subset $\Phi$ of $\Mor(\mathcal{T}^c)$ such that 
	$$\mathcal{C} = \{X \in \mathcal{T} \mid \Hom{}_\mathcal{T}(f,X) \text{ is surjective for all $f \in \Phi$}\}.$$
\end{definition}
	A recent result from \cite{L} shows that, under mild assumptions, definable subcategories can be characterized by their closure properties, in a way analogous to definable subcategories of module categories. Before stating this result, we need to recall the notions of purity in compactly generated triangulated categories, and of derivators and homotopy (co)limits.
\subsection{Purity in compactly generated triangulated categories}
Consider the category $\ModTc$ of all $\mathcal{T}^c$-modules, that is, of all contravariant functors $\mathcal{T}^c \rightarrow \mathbf{Ab}$. We let $\mathbf{y}: \mathcal{T} \rightarrow \ModTc$ be the \EMP{restricted Yoneda functor}, by which we mean the functor defined by restricting the standard Yoneda functor on $\mathcal{T}$ to $\mathcal{T}^c$. Explicitly,
$$\mathbf{y}(+) = \Hom{}_{\mathcal{T}}(-,+)_{\restriction \mathcal{T}^c}.$$
This functor can be used to build a useful theory of purity in $\mathcal{T}$.
\begin{definition}
	A triangle $X \xrightarrow{f} Y \xrightarrow{g} Z \rightarrow X[1]$ in $\mathcal{T}$ is a \EMP{pure triangle} if the induced sequence
		$$0 \rightarrow \mathbf{y}(X) \xrightarrow{\mathbf{y}(f)} \mathbf{y}(Y) \xrightarrow{\mathbf{y}(g)} \mathbf{y}(Z) \rightarrow 0$$
		is exact in $\ModTc$. If this is the case, we call $f$ a \EMP{pure monomorphism} and $g$ a \EMP{pure epimorphism} in $\mathcal{T}$. We remark that, of course, pure monomorphisms in $\mathcal{T}$ will usually not be monomorphisms in the categorical sense, and the same is the case with pure epimorphisms.

		Moreover, we call an object $E \in \mathcal{T}$ \EMP{pure-injective} if any pure monomorphism $E \rightarrow X$ in $\mathcal{T}$ splits.	
\end{definition}
	The purity in $\mathcal{T}$ is closely tied to the definable subcategories in $\mathcal{T}$ via the notion of the Ziegler spectrum. Here, we follow \cite[\S 17]{P}. The \EMP{Ziegler spectrum} $\Zgs(\mathcal{T})$ of $\mathcal{T}$ is the collection of isomorphism classes of all indecomposable pure-injective objects of $\mathcal{T}$. Then $\Zgs(\mathcal{T})$ is always a set, and it is equipped with a topology given as follows: A subset $U$ of $\Zgs(\mathcal{T})$ is closed if and only if there is a definable subcategory $\mathcal{C}$ of $\mathcal{T}$ such that $U = \Zgs(\mathcal{T}) \cap \mathcal{C}$. The following result due to Krause says in particular that every definable subcategory of $\mathcal{T}$ is fully determined by the indecomposable pure-injective objects it contains.
	\begin{thm}\label{T:Ziegler}\emph{(\cite{Kr2})}
		Let $\mathcal{T}$ be a compactly generated triangulated category. Then there is a bijective correspondence:
			$$\left \{ \begin{tabular}{ccc} \text{ closed subsets $U$ of $\Zgs(\mathcal{T})$} \end{tabular}\right \}  \leftrightarrow  \left \{ \begin{tabular}{ccc} \text{ definable subcategories $\mathcal{C}$ of $\mathcal{T}$ }  \end{tabular}\right \}$$
		Both correspondences are given by the mutually inverse assignments
			$$U \mapsto \{X \in \mathcal{T} \mid \text{there is a pure monomorphism $X \rightarrow \prod_{i \in I}P_i$, where $P_i \in U$}\},$$
			$$\mathcal{C} \mapsto \mathcal{C} \cap \Zgs(\mathcal{T}).$$
	\end{thm}

\subsection{Derivators and homotopy (co)limits}
Triangulated categories usually do not have many useful limits and colimits apart from products and coproducts. A way to remedy this is to introduce an additional structure on them and compute the homotopy (co)limits instead. In our case, this extra structure comes from assuming that $\mathcal{T}$ is the underlying category of a strong and stable derivator. Since we will very soon restrict ourselves to the case of derived categories, we omit most of the details on derivators, and refer the reader to \cite{L} and references therein for an exposition of the theory well-suited for our application.

A derivator is a contravariant 2-functor $\mathbb{D}: \Cat^{op} \rightarrow \CAT$ from the category of small categories to the category of all categories, satisfying certain conditions. We denote by $\star$ the category consisting of a single object and a single map. The category $\mathbb{D}(\star)$ is called the \EMP{underlying category} of the derivator $\mathbb{D}$. For every small category $I$, we consider the unique functor $\pi_I: I \rightarrow \star$. The definition of a derivator implies that the functor $\mathbb{D}(\pi_I): \mathbb{D}(\star) \rightarrow \mathbb{D}(I)$ admits both the right and the left adjoint functor. We denote the right adjoint by $\holim: \mathbb{D}(I) \rightarrow \mathbb{D}(\star)$ and the left adjoint by $\hocolim: \mathbb{D}(I) \rightarrow \mathbb{D}(\star)$. We omit the definition of a \EMP{strong and stable} derivator, but we remark that amongst the consequences of these properties is that the category $\mathbb{D}(I)$ is triangulated for all $I \in \Cat$.

Given a small category $I$ and an object $i \in I$, let $i$ also denote the functor $i:\star \rightarrow I$ sending the unique object of $\star$ onto $i$. Then we have the induced functor $\mathbb{D}(i): \mathbb{D}(I) \rightarrow \mathbb{D}(\star)$. For any $\mathscr{X} \in \mathbb{D}(I)$ we denote $\mathscr{X}_i = \mathbb{D}(i)(\mathscr{X}) \in \mathbb{D}(\star)$ and call it the \EMP{$i$-th component of $\mathscr{X}$}. Together, the component functors induce the \EMP{diagram functor} $d_I: \mathbb{D}(I) \rightarrow \mathbb{D}(\star)^I$. The objects of $\mathbb{D}(I)$ are called the \EMP{coherent diagrams} in the underlying category of shape $I$. Via the diagram functor, any coherent diagram can be interpreted as a usual (or \EMP{incoherent}) diagram in the underlying category.

\subsection{Standard derivator of a module category} Here we follow \cite[\S 5]{St}. Let $R$ be a ring, and let $\ModR$ be the abelian category of all right $R$-modules. For any small category $I \in \Cat$, we let $(\ModR)^I$ be the category of all $I$-shaped diagrams in $\ModR$, that is, the abelian category of all functors $I \rightarrow \ModR$. Let $\Der((\ModR)^I)$ denote the unbounded derived category of $(\ModR)^I$. Recall that there is a natural equivalence between the category of chain complexes of objects in $(\ModR)^I$, and the $I$-shaped diagrams of chain complexes of $R$-modules. Therefore, $\Der((\ModR)^I)$ can be considered as the Verdier localization of the category of $I$-shaped diagrams of chain complexes. There is the \EMP{standard derivator} associated to $\ModR$ which assigns to any small category $I \in \Cat$ the triangulated category $\Der((\ModR)^I)$. The underlying category $\Der(\ModR)$ will be denoted simply by $\Der(R)$. This assignment defines a strong and stable derivator, and the homotopy limit and colimit functors can be in this case described by derived functors in the following way. Let $I \in \Cat$, then we define the \EMP{homotopy colimit} functor $\hocolim_{i \in I}: \Der((\ModR)^I) \rightarrow \Der(R)$ to be the left derived functor $\mathbf{L}\colim_{i \in I}$ of the usual colimit functor $\colim_{i \in I}: (\ModR)^I \rightarrow \ModR$. Dually, we define the \EMP{homotopy limit} functor as $\holim_{i \in I} := \mathbf{R} \lim_{i \in I}: \Der((\ModR)^I) \rightarrow \Der(R)$.

 The objects of $\Der((\ModR)^I)$, that is, the coherent diagrams of shape $I$, are all represented by diagrams of chain complexes of $R$-modules. Let $\mathscr{X} \in \Der((\ModR)^I)$ be represented by a diagram $(X_i \mid i \in I)$ of chain complexes. Then clearly, $\mathscr{X}_i \simeq X_i$ as objects of $\Der(R)$ for any $i \in I$.

We will be especially interested in the homotopy colimit construction in the case when the small category $I$ is directed. In this situation, we call $\hocolim_{i \in I}$ a \EMP{directed homotopy colimit}. Because the direct limit functor $\varinjlim_{i \in I} = \colim_{i \in I}$ on the category of chain complexes of $R$-modules is exact, we have for each object $\mathscr{X} \in \Der((\ModR)^I)$, represented by a diagram $(X_i \mid i \in I)$ of chain complexes, the isomorphism $\hocolim_{i \in I} \mathscr{X} \simeq \varinjlim_{i \in I} X_i$ in $\Der(R)$. In particular, we have for any $n \in \mathbb{Z}$ the following isomorphism on cohomologies
$$H^n(\hocolim_{i \in I} \mathscr{X}) \simeq \varinjlim_{i \in I}H^n(X_i).$$ 
For more details, we refer to \cite[Proposition 6.6]{St}.
\subsection{Definable coaisles}
We now assume that $\mathcal{T}$ is an underlying subcategory of a \EMP{compactly generated derivator}, that is, a strong and stable derivator $\mathbb{D}$ such that the underlying category $\mathbb{D}(\star)$ (and by \cite[Lemma 3.2]{L}, consequently also any of the categories $\mathbb{D}(I)$ for any small category $I$) is compactly generated. This implies that $\mathcal{T}$ is a compactly generated triangulated category, in which we can compute homotopy colimits and limits. \begin{definition}
We say that a subcategory $\mathcal{C}$ of $\mathcal{T}$ is 
		\begin{itemize}
				\item \EMP{closed under directed homotopy colimits} if for any directed small category $I$ and any coherent diagram $\mathscr{X} \in \mathbb{D}(I)$ such that $\mathscr{X}_i \in \mathcal{C}$ for all $i \in I$ we have $\hocolim_{i \in I} \mathscr{X} \in \mathcal{C}$,
				\item \EMP{closed under pure monomorphisms} if for any pure monomorphism $Y \rightarrow X$ such that $X \in \mathcal{C}$ we have $Y \in \mathcal{C}$.
		\end{itemize}
		Following \cite{SSV}, we call a t-structure $(\mathcal{U},\mathcal{V})$ \EMP{homotopically smashing} if the coaisle $\mathcal{V}$ is closed under directed homotopy colimits. We point out here that any aisle is closed under arbitrary homotopy colimits, and any coaisle is closed under arbitrary homotopy limits, this is \cite[Proposition 4.2]{SSV}.
\end{definition}

We are now ready to state the result from \cite{L} characterizing definable subcategories of $\mathcal{T}$ by their closure properties. 
\begin{thm}\emph{(\cite[Theorem 3.11]{L})}\label{T:rosie1}
		A subcategory $\mathcal{C}$ of $\mathcal{T}$ is definable if and only if $\mathcal{C}$ is closed under products, directed homotopy colimits, and pure monomorphisms.
	\end{thm}
	We will be especially interested in the situation when a coaisle of a t-structure is a definable subcategory. The following result shows that in this case the existence of the triangles from condition (ii) is automatic. 	
	\begin{thm}\emph{(\cite[Proposition 4.5]{AMV2})}
			Let $\mathcal{V}$ be a definable subcategory of $\mathcal{T}$ closed under extensions and cosuspensions. Then the pair $(\Perp{0}\mathcal{V},\mathcal{V})$ is a t-structure.
	\end{thm}
	Putting the last two results together, we have a nice intrinsic characterization of the notion of a definable coaisle.
	\begin{cor}\label{C:closure}
			A subcategory $\mathcal{V}$ of $\mathcal{T}$ is a definable coaisle if and only if $\mathcal{V}$ is closed in $\mathcal{T}$ under extensions, cosuspensions, products, directed homotopy colimits, and pure monomorphisms.
	\end{cor}
\section{Cosilting objects and t-structures induced by them}\label{S:cosilting}
In this section we recall the results of \cite{MV} and \cite{L}, which show that definability of coaisles is closely related to cosilting theory. For any object $C \in \mathcal{T}$, we define the following two subcategories of $\mathcal{T}$:
$$\Perp{\leq 0}C= \{X \in \mathcal{T} \mid \Hom{}_\mathcal{T}(X,C[i]) = 0 ~\forall i \leq 0 \}\text{, and}$$
$$\Perp{>0}C=\{X \in \mathcal{T} \mid \Hom{}_\mathcal{T}(X,C[i]) = 0 ~\forall i > 0 \}.$$
An object $C$ of a triangulated category $\mathcal{T}$ is \EMP{cosilting} provided that the pair $\mathbf{t}=(\Perp{\leq 0}C,\Perp{>0}C)$ forms a t-structure in $\mathcal{T}$. In this situation, we say that the t-structure $\mathbf{t}$ is a \EMP{cosilting t-structure}, and it is \EMP{induced} by the cosilting object $C$. Among the consequences of the definition (see \cite[Proposition 4.3]{PV}) is that any cosilting object $C$ is a (weak) cogenerator in $\mathcal{T}$, that $C \in \Perp{>0}C$, and that any cosilting t-structure is non-degenerate in the following sense: .

\begin{definition}
		A $(\mathcal{U},\mathcal{V})$ t-structure in a triangulated category $\mathcal{T}$ is called \EMP{non-degenerate} provided that $\bigcap_{n \in \mathbb{Z}}\mathcal{U}[n] = 0$ and $\bigcap_{n \in \mathbb{Z}}\mathcal{V}[n] = 0$.
\end{definition}

Now we are ready to state the following result due to Laking.
\begin{thm}\emph{(\cite[Theorem 4.6]{L})}\label{T:rosie2}
		Let $\mathcal{T}$ be an underlying triangulated category of a compactly generated derivator, and consider a non-degenerate t-structure $\mathbf{t} = (\mathcal{U},\mathcal{V})$ in $\mathcal{T}$. Then the following conditions are equivalent:
		\begin{enumerate}
			\item[(i)] $\mathbf{t}$ is induced by a pure-injective cosilting object $C$,
			\item[(ii)] the subcategory $\mathcal{V}$ is definable,
			\item[(iii)] the t-structure $\mathbf{t}$ is homotopically smashing,
			\item[(iv)] the subcategory $\mathcal{V}$ is closed under coproducts, i.e. $\mathbf{t}$ is a smashing t-structure, and the \EMP{heart} $\mathcal{H} := \mathcal{U} \cap \mathcal{V}[1]$ is a Grothendieck category.
		\end{enumerate}
\end{thm}
Since any t-structure induced by a cosilting object is non-degenerate, we have also the following reformulation:
\begin{cor}\label{C:rosie3}
		Let $\mathcal{T}$ be an underlying triangulated category of a compactly generated derivator, and let $\mathbf{t} = (\mathcal{U},\mathcal{V})$ in $\mathcal{T}$. Then the following conditions are equivalent:
		\begin{enumerate}
			\item[(i)] $\mathbf{t}$ is induced by a pure-injective cosilting object $C$,
			\item[(iii)] the t-structure $\mathbf{t}$ is non-degenerate and $\mathcal{V}$ is definable.
		\end{enumerate}
\end{cor}
Now we confine to the case of $\mathcal{T} = \Der(R)$ for a ring $R$. We say that a subcategory $\mathcal{V}$ of $\Der(R)$  is \EMP{co-intermediate} if there are integers $m \leq n$ such that $\Der^{\geq n} \subseteq \mathcal{V} \subseteq \Der^{\geq m}$, where 
$$\Der{}^{\geq k} = \{X \in \Der(R) \mid H^l(X) = 0 ~\forall l<k\}.$$
We say that a cosilting object $C \in \Der(R)$ is a \EMP{bounded cosilting complex} if $C$ is isomorphic to a bounded complex of injective $R$-modules in $\Der(R)$. As an example, any large $n$-cotilting $R$-module (in the sense of \cite[\S 15]{GT}) is a bounded cosilting complex when considered as an object of $\Der(R)$. Then \cite{MV} shows that any bounded cosilting complex of $\Der(R)$ is pure-injective, and we have the two following characterization of the t-structures induced by bounded cosilting complexes:
\begin{thm}\emph{(\cite[Theorem 3.14]{MV})}\label{T:MV}
		Let $R$ be a ring, and $\mathcal{V}$ a subcategory of $\Der(R)$. Then the following conditions are equivalent:
		\begin{enumerate}
				\item[(i)] 	$\mathcal{V}$ is the coaisle of a t-structure induced by a bounded cosilting complex,
			\item[(ii)] $\mathcal{V}$ is definable, co-intermediate, and closed under extensions and cosuspensions.
		\end{enumerate}
\end{thm}
\subsection{Module-theoretic cosilting torsion-free classes} 
We start by a well-known construction due to Happel, Reiten, and Smal\o. To do this, we must first recall the notion of a torsion pair in a module category. Let $R$ be a ring. A \EMP{torsion pair} in $\ModR$ is a pair $(\mathcal{T},\mathcal{F})$ of subcategories of $\ModR$ such that $\Hom_R(\mathcal{T},\mathcal{F})=0$ and both the subcategories are maximal with respect to this property. We call $\mathcal{T}$ a \EMP{torsion class} and $\mathcal{F}$ a \EMP{torsion-free class}. It is well-known that a subcategory $\mathcal{T}$ of $\ModR$ is a torsion class (belonging to some torsion pair) if and only if $\mathcal{T}$ is closed under extensions, coproducts, and epimorphic images. Dually, torsion-free classes are characterized by the closure under extensions, products, and submodules. Finally, a torsion pair is \EMP{hereditary} if $\mathcal{T}$ is closed under submodules, or equivalently, $\mathcal{F}$ is closed under taking injective envelopes. We call a subcategory of $\ModR$ \EMP{definable} if it is closed under products, pure submodules, and direct limits. In particular, a torsion-free class is definable if and only if it is closed under direct limits. We refer the reader to \cite{P} as a main reference for the theory of definable subcategories in the setting of a module category. 

Then we define the \EMP{Happel-Reiten-Smal\o ~t-structure} $(\mathcal{U},\mathcal{V})$ induced by the torsion pair $(\mathcal{T},\mathcal{V})$ to be the pair of subcategories of $\Der(R)$ given as
$$\mathcal{U} = \{X \in \Der(R) \mid H^n(X) = 0 ~\forall n>0 \text{ and } H^0(X) \in \mathcal{T}\},$$
and 
$$\mathcal{V} = \{X \in \Der(R) \mid H^n(X) = 0 ~\forall n<0 \text{ and } H^0(X) \in \mathcal{F}\}.$$
By \cite{HRS}, this construction induces an injective assignment from the class of torsion pairs in $\ModR$ to the class of t-structures in $\Der(R)$. Clearly, the coaisle $\mathcal{V}$ of any Happel-Reiten-Smal\o ~t-structure satisfies $\Der^{\geq 1} \subseteq \mathcal{V} \subseteq \Der^{\geq 0}$. Conversely, any t-structure with coaisle satisfying the latter property is Happel-Reiten-Smal\o, see \cite[Lemma 1.1.2]{Po}. It is an easy task to characterize the Happel-Reiten-Smal\o ~t-structures which are induced by a cosilting object. 
\begin{lem}\label{L:hrs}
		Let $R$ be a ring and $(\mathcal{U},\mathcal{V})$ a t-structure in $\Der(R)$. Then the following conditions are equivalent:
		\begin{enumerate}
			\item[(i)] $(\mathcal{U},\mathcal{V})$ is a cosilting t-structure such that $\Der^{\geq 1} \subseteq \mathcal{V} \subseteq \Der^{\geq 0}$,
			\item[(ii)] $(\mathcal{U},\mathcal{V})$ is a Happel-Reiten-Smal\o ~t-structure induced by a torsion pair $(\mathcal{T},\mathcal{F})$ in $\ModR$ such that $\mathcal{F}$ is closed under direct limits.
		\end{enumerate}
\end{lem}
		\begin{proof}
				In view of Theorem~\ref{T:MV}, the only thing we need to check is that if $(\mathcal{U},\mathcal{V})$ is a Happel-Reiten-Smal\o ~t-structure induced by a torsion pair $(\mathcal{T},\mathcal{F})$, then $\mathcal{V}$ is definable in $\Der(R)$ if and only if $\mathcal{F}$ is closed under direct limits in $\ModR$. By the definition of $\mathcal{V}$ we have that $X \in \mathcal{V}$ if and only if $H^0(X) \in \mathcal{F}$ for any $X \in \Der^{\geq 0}$. If $\mathscr{X} \in \Der((\ModR)^I)$ for some directed diagram $I$ with $\mathscr{X}_i \in \mathcal{V}$ for all $i \in I$, then we have $H^0(\hocolim_{i \in I} \mathscr{X}_i) \simeq \varinjlim_{i \in I}H^0(\mathscr{X}_i)$. Therefore, $\hocolim_{i \in I} \mathscr{X}_i \in \mathcal{V}$ provided that $\mathcal{F}$ is closed under direct limits. On the other hand, any $I$-shaped directed system of modules in $\mathcal{F}$ can be regarded as a coherent diagram in $\Der((\ModR)^I)$ with coordinates being stalk complexes from $\mathcal{V}$. Therefore, if $\mathcal{V}$ is closed under directed homotopy colimits then $\mathcal{F}$ is closed under direct limits. Finally, since $(\mathcal{U},\mathcal{V})$ is non-degenerate, $\mathcal{V}$ is definable if and only if $\mathcal{V}$ is closed under directed homotopy colimits by Theorem~\ref{T:rosie2}.
		\end{proof}
		Finally, we discuss the connection to the cosilting and cotilting modules. Following \cite{AMV} and \cite{BP}, an $R$-module $T$ is \EMP{cosilting} if there is an injective copresentation 
$$0 \rightarrow T \rightarrow Q_0 \xrightarrow{\sigma} Q_1,$$
such that $\mathcal{C}_\sigma = \Cogen(T),$
where $\mathcal{C}_\sigma = \{M \in \ModR \mid \Hom_R(M,\sigma) \text{ is surjective}\}$. A class $\mathcal{C} \subseteq \ModR$ is called \EMP{cosilting} if there is a cosilting module $T$ such that $\mathcal{C} = \Cogen(T)$. It is easy to infer from a result due to Breaz-\v{Z}emli\v{c}ka and Wei-Zhang that cosilting classes are precisely the torsion-free classes closed under direct limits, in other words, the definable torsion-free classes in $\ModR$.
\begin{thm}\EMP{(\cite{BZ},\cite{WZ})}\label{T:BZ}
	A class $\mathcal{C} \subseteq \ModR$ is cosilting if and only if $\mathcal{C}$ is a definable torsion-free class.
\end{thm}
		\begin{proof}
				In \cite{WZ} it is proved that a torsion-free class is cosilting if and only if it is a covering class. Any definable subcategory is covering, and \cite[Corollary 4.8]{BP} shows that any cosilting class is definable.
		\end{proof}
	Cosilting modules are precisely the module-theoretic shadows of 2-term cosilting complexes. A cosilting complex is \EMP{2-term} if it can be represented by a complex of injective $R$-modules concentrated in degrees 0 and 1. We say that two cosilting objects are \EMP{equivalent} if they induce the same t-structure, and that two cosilting modules are \EMP{equivalent} if they cogenerate the same cosilting class. 
	\begin{thm}\emph{(\cite[Theorem 4.19]{WZ})}\label{T:WZ}
	Let $R$ be a ring. Then there are bijections between the following sets:
		\begin{enumerate}
			\item[(i)] equivalence classes of 2-term cosilting complexes $C$,
			\item[(iii)] equivalence classes of cosilting $R$-modules $T = \Ker(Q_0 \rightarrow Q_1)$, where $Q_0, Q_1$ are injective $R$-modules.
		\end{enumerate}
		The bijection composes of two mutually inverse assignments
		$$C \mapsto H^0(C)\text{, and}$$
		$$T \mapsto (\cdots\rightarrow 0 \rightarrow Q_0 \xrightarrow{\sigma} Q_1 \rightarrow 0 \rightarrow  \cdots).$$
\end{thm}
\begin{prop}
		Let $R$ be a ring and $\mathbf{t}=(\mathcal{U},\mathcal{V})$ a t-structure. Then the following conditions are equivalent:
		\begin{enumerate}
				\item[(i)] $\mathbf{t}$ is induced by a 2-term cosilting complex $C$,
				\item[(ii)] $\mathbf{t}$ is a Happel-Reiten-Smal\o ~t-structure induced by a torsion pair $(\mathcal{T},\mathcal{F})$, where $\mathcal{F}$ is a cosilting class.
		\end{enumerate}
		Furthermore, the cosilting class $\mathcal{F}$ is cogenerated by the cosilting module $H^0(C)$ for any choice of the equivalence representative $C$.
\end{prop}
\begin{proof}
		If $(i)$ holds, then the coaisle of the cosilting t-structure clearly squeezes between $\Der^{\geq 1}$ and $\Der^{\geq 0}$, and therefore is Happel-Reiten-Smal\o ~by Lemma~\ref{L:hrs}. Furthermore, the torsion pair inducing this t-structure necessary has the torsion-free class cogenerated by the cosilting module $H^0(C)$ by \cite[Proposition 2.16]{Pop}. Conversely, if $\mathbf{t}$ is Happel-Reiten-Smal\o ~induced by a torsion pair $(\mathcal{T},\mathcal{F})$ with $\mathcal{F}$ cosilting, then $\mathbf{t}$ is a cosilting t-structure by Theorem~\ref{T:BZ} and Lemma~\ref{L:hrs}. Let $T$ be a cosilting module cogenerating $\mathcal{F}$, and let $\sigma: Q_0 \rightarrow Q_1$ be a map witnessing that $T$ is a cosilting module. Then $\sigma$ is a 2-term cosilting complex by Theorem~\ref{T:WZ}, and $\sigma$ induces $\mathbf{t}$ by \cite[Proposition 2.16]{Pop}.
\end{proof}

	As a summary, studying the Happel-Reiten-Smal\o ~t-structures induced by a cosilting object in $\Der(R)$ boils down to studying definable torsion-free classes in $\ModR$. 

	\subsection{Cotilting modules}\label{SS:cotilting} We also need to recall the basics on (large) 1-cotilting modules. Let $\mathcal{C}$ be a subcategory of $\ModR$. We will use the notation
	$$\Perp{}\mathcal{C} = \{M \in \ModR \mid \Ext{}_R^1(M,C) = 0 ~\forall C \in \mathcal{C}\}\text{, and}$$
	$$\mathcal{C}\Perp{} = \{M \in \ModR \mid \Ext{}_R^1(C,M) = 0 ~\forall C \in \mathcal{C}\},$$
	and if $\mathcal{C} = \{C\}$ is a singleton, we will drop the curly brackets. An $R$-module $C$ is called \EMP{(1-)cotilting} provided $C$ has injective dimension at most one and $\Perp{}C = \Cogen(C)$. It is easily seen that any 1-cotilting module is a cosilting module, and this is witnessed by any injective coresolution $0 \rightarrow C \rightarrow Q_0 \xrightarrow{\sigma} Q_1 \rightarrow 0$. The cosilting class $\Cogen(C)$ is in this case called a \EMP{(1-)cotilting class} induced by $C$. Clearly, any 1-cotilting class contains all projective $R$-modules. Conversely, any cosilting class containing $R$ is a cotilting class by \cite[Proposition 3.14]{APST}. 
\section{Definable subcategories in the derived category of rings of weak global dimension at most one}
	Recall that a ring $R$ is of \EMP{weak global dimension at most one} if any submodule of a flat $R$-module is flat, or equivalently, that $\Tor_2^R(-,-)$ is a zero functor $\ModR \times \ModRop \rightarrow \Ab$, which also demonstrates that this is a left-right symmetric property of a ring.
	
	The main aim of this section is to use the K\"{u}nneth formula to prove that definable subcategories in the derived category of a ring of weak global dimension at most one are fully determined by cohomology. We start with a reformulation of the definition of a definable subcategory in the derived category of a ring. Given a ring $R$, let $\Rop$ be the opposite ring, so that $\ModRop$ is identified with the category of all \EMP{left} $R$-modules. 	
	\subsection{Determination on cohomology}\label{SS:determined}
	\begin{lem}\label{L:definable}
		Let $R$ be a ring, and let $\mathcal{C}$ be a subcategory of $\Der(R)$. Then the following are equivalent:
			\begin{enumerate}
					\item[(i)] $\mathcal{C}$ is definable in $\Der(R)$,
				\item[(ii)] there is a set $\Phi \subseteq \Mor(\Der(R)^c)$ such that 
						$$\mathcal{C} = \{ X \in \Der(R) \mid \Hom{}_{\Der(R)}(f,X) \text{ is injective for each $f \in \Phi$}\},$$
				\item[(iii)] there is a set $\Phi \subseteq \Mor(\Der(R)^c)$ such that 
						$$\mathcal{C} = \{ X \in \Der(R) \mid \Hom{}_{\Der(R)}(f,X) \text{ is zero for each $f \in \Phi$}\},$$
				\item[(iv)] there is a set $\Phi \subseteq \Mor(\Der(\Rop)^c)$ such that 
						$$\mathcal{C} = \{ X \in \Der(R) \mid H^0(X \otimes_R^\mathbf{L} f) \text{ is zero for each $f \in \Phi$}\}.$$
			\end{enumerate}
	\end{lem}
	\begin{proof}
			Let $f: C \rightarrow D$ be a map in $\Der(R)^c$, and consider the induced triangle
			$$C \xrightarrow{f} D \xrightarrow{g} E \xrightarrow{h} C[1],$$
			in the triangulated category $\Der^c(R)$. Applying $\Hom_{\Der(R)}(-,X)$, we obtain an exact sequence 
			$$\Hom{}_{\Der(R)}(D,X) \xrightarrow{\Hom{}_{\Der(R)}(f,X)}  \Hom{}_{\Der(R)}(C,X) \xrightarrow{\Hom{}_{\Der(R)}(h[-1],X)} $$  
			$$\xrightarrow{\Hom{}_{\Der(R)}(h[-1],X)} \Hom{}_{\Der(R)}(E[-1],X) \xrightarrow{\Hom{}_{\Der(R)}(g[-1],X)} \Hom{}_{\Der(R)}(D[-1],X)$$
			of abelian groups. It follows that $\Hom{}_{\Der(R)}(f,X)$ is surjective if and only if $\Hom{}_{\Der(R)}(h[-1],X)$ is zero if and only if $\Hom{}_{\Der(R)}(g[-1],X)$ is injective. This establishes the equivalence of $(i) - (iii)$.

			Suppose that $\Phi$ is a set of maps between objects from $\Der^c(R)$ such that 
			$$\mathcal{C} = \{ X \in \Der(R) \mid \Hom{}_{\Der(R)}(f,X) \text{ is zero for each $f \in \Phi$}\}.$$ 
			We define the set $\Phi^*$ of maps in $\Der(\Rop)$ as follows:
			$$\Phi^* = \{\RHom{}_R(f,R) \mid f \in \Phi\}.$$ 
			Recalling that $\RHom_R(-,R)$ induces an equivalence $\Der^c(R) \rightarrow \Der^c(\Rop)$, we have that $\Phi^*$ is actually a set of maps between objects from $\Der^c(\Rop)$. Then the equivalence of $(iii)$ and $(iv)$ comes from the following standard isomorphism in $\Der(\Ab)$, natural in $C \in \Der^c(R)$:
			$$\RHom(C,X) \simeq X \otimes^\mathbf{L}_R \RHom{}_R(C,R),$$
			which implies that for any $f \in \Phi$ we have the following isomorphism of maps in $\Ab$:
			$$\Hom{}_{\Der(R)}(f,X) \simeq H^0\RHom(f,X) \simeq H^0(X \otimes^\mathbf{L}_R \RHom{}_R(f,R)).$$
	\end{proof}

		\begin{definition}
		Let $\mathcal{V}$ be a subcategory of $\Der(R)$. We say that the subcategory $\mathcal{V}$ is \EMP{determined on cohomology} if the following equivalence holds for each $X \in \Der(R)$:
		$$X \in \mathcal{V} \iff H^n(X)[-n] \in \mathcal{V} ~\forall n \in \mathbb{Z}.$$
	\end{definition}
	The characterization (iv) of Lemma~\ref{L:definable} of definable subcategories using tensor product will be useful here, and as in the proof of an analogous statement for localizing pairs in \cite[\S 3]{BS}, the K\"{u}nneth's theorem will play a crucial role. 
	\begin{lem}\label{L:kunneth}
			Let $R$ be a ring of weak global dimension at most one, let $X$ be any object in $\Der(R)$, and let
			$$E \xrightarrow{h} C \xrightarrow{f} D \xrightarrow{g} E[1]$$
			be any triangle in $\Der(R)$. Then the following conditions are equivalent:
			\begin{enumerate}
				\item[(i)] $H^n(X \otimes_R^\mathbf{L} f)$ is a zero map in $\Ab$, 
				\item[(ii)] the following two conditions hold:
					\begin{itemize} 
						\item for all $p+q = n$ the map $H^p(X) \otimes_R H^q(f)$ is zero in $\Ab$, and
						\item for all $p+q = n+1$ the map $\Tor{}_1^R(H^p(X),H^q(h))$ is surjective.
					\end{itemize}
			\end{enumerate}
	\end{lem}
	\begin{proof}
			We start by making the assumption that $f: C \rightarrow D$ and $h: E \rightarrow C$ are actually (represented by) maps of chain complexes. This is a harmless assumption, as we can for example replace $C$ by a quasi-isomorphic $K$-projective replacement $C'$ (these always exist by \cite[Corollary 2.8]{S}), then replace $f$ by its image $f'$ in the isomorphism $\Hom_{\Der(R)}(C,D) \simeq \Hom_{\mathbf{K}(R)}(C',D)$, and finally replace $h$ by the mapping cocone of $f'$.
			
			Next, let $P$ be a $K$-projective complex of right $R$-modules which is quasi-isomorphic to $X$. Then the components of $P$ are projective $R$-modules, and since $R$ has weak global dimension at most one, all coboundary and cocycle modules of $P$ are flat $R$-modules. Therefore, we can use the K\"{u}nneth's formula \cite[\S VI Theorem 3.1]{CE}, and its naturality \cite[\S IV Theorem 8.1]{CE}, in degree $n$ for the chain maps $P \otimes_R f: P \otimes_R C \rightarrow P \otimes_R D$ and $P \otimes_R h: P \otimes_R E \rightarrow P \otimes_R C$, to obtain the following commutative diagram
			{\small
$$
\begin{tikzcd}[column sep=tiny]
		0 \arrow{r} & \displaystyle\bigoplus_{p + q = n}H^p(P) \otimes{}_R H^q(E) \arrow{d}{\bigoplus_{}H^p(P) \otimes{}_R H^q(h)}\arrow{r} & H^{n}(P \otimes{}_R E) \arrow{d}{H^{n}(P \otimes{}_R h)} \arrow{r}{\pi_E} & \displaystyle\bigoplus_{p+q = n+ 1} \Tor{}_1^R(H^p(P),H^q(E)) \arrow{d}{\bigoplus_{} \Tor{}_1^R(H^p(P),H^q(h))} \arrow{r} & 0 \\
		0 \arrow{r} & \displaystyle\bigoplus_{p + q = n}H^p(P) \otimes{}_R H^q(C) \arrow{d}{\bigoplus_{}H^p(P) \otimes{}_R H^q(f)}\arrow{r} & H^n(P \otimes{}_R C) \arrow{d}{H^n(P \otimes{}_R f)} \arrow{r}{\pi_C} & \displaystyle\bigoplus_{p+q = n+1} \Tor{}_1^R(H^p(P),H^q(C)) \arrow{d}{\bigoplus_{} \Tor{}_1^R(H^p(P),H^q(f))} \arrow{r} & 0. \\
		0 \arrow{r} & \displaystyle\bigoplus_{p + q = n}H^p(P) \otimes{}_R H^q(D) \arrow{r} & H^n(P \otimes{}_R D) \arrow{r} & \displaystyle\bigoplus_{p+q = n+1} \Tor{}_1^R(H^p(P),H^q(D)) \arrow{r} & 0 \\
\end{tikzcd}
$$
			}
			with rows being the short exact sequences provided by the K\"{u}nneth formula. Also, the middle column of the diagram is exact, because it is a part of the long exact sequence on cohomologies induced by the triangle
			$$P \otimes_R^\mathbf{L} E \xrightarrow{P \otimes_R^\mathbf{L} h}P \otimes_R^\mathbf{L} C \xrightarrow{P \otimes_R^\mathbf{L} f}P \otimes_R^\mathbf{L} D \xrightarrow{P \otimes_R^\mathbf{L} g} P \otimes_R^\mathbf{L} E[1],$$
			and by the fact that $P \otimes_R^\mathbf{L} -: \Der(\Rop) \rightarrow \Der(\Ab)$ is represented by the ordinary tensor product $P \otimes_R -$, because $P$ is $K$-projective.

			Assume first that $H^n(X \otimes_R^\mathbf{L} f)$ is a zero map. Since $P$ is the $K$-projective replacement of $X$, we have an isomorphism of maps $H^n(X \otimes_R^\mathbf{L} f) \simeq H^n(P \otimes_R f)$. Then $H^n(P \otimes_R f)$ is a zero map, and the exactness of the rows and commutativity of the diagram implies that $\bigoplus_{} H^p(P) \otimes_R H^q(f)$ and $\bigoplus_{}\Tor{}_1^R(H^p(P),H^q(f))$ are zero maps in $\Ab$, and therefore all of their direct sum components are zero maps. By the exactness of the middle column, $H^n(P \otimes_R f)$ being zero forces $H^n(P \otimes_R h)$ to be surjective. The commutativity of the upper right square then implies that the map $\bigoplus_{}\Tor{}_1^R(H^p(P),H^q(h))$ is surjective, and therefore the component maps are surjective as well. Because $H^p(P) \simeq H^p(X)$ for all $p \in \mathbb{Z}$, we have proved $(i) \implies (ii)$.

			Now suppose that $(ii)$ holds. Then $\bigoplus_{} H^p(P) \otimes_R H^q(f)$ is a zero map, and therefore $H^n(P \otimes_R f)$ factors through the epimorphism $\pi_C: H^n(P \otimes_R C) \rightarrow \bigoplus_{p + q = n+1} \Tor_1^R(H^p(P),H^q(C))$, say 
			$$H^n(P \otimes_R f) = \varphi \circ \pi_C$$
			for some map $\varphi: \bigoplus_{p + q = n+1} \Tor_1^R(H^p(P),H^q(C)) \rightarrow H^n(P \otimes_R D)$. Using the commutativity of the diagram, we can compute the composition of maps as follows
			$$\varphi \circ \bigoplus_{}\Tor{}_1^R(H^p(P),H^q(h)) \circ \pi_E = \varphi \circ \pi_C \circ H^n(P \otimes_R h) = $$ 
			$$= H^n(P \otimes_R f) \circ H^n(P \otimes_R h) = 0.$$
			
			But by $(ii)$, the map $\bigoplus_{}\Tor{}_1^R(H^p(P),H^q(h))$ is an epimorphism, and so is $\pi_E$. Therefore, $\varphi = 0$, and thus $H^n(P \otimes_R f)$ is a zero map. Again, as $P$ is the $K$-projective replacement of $X$, this means that $H^n(X \otimes^\mathbf{L}_R f)$ is a zero map, proving the implication $(ii) \implies (i)$.
	\end{proof}
	\begin{thm}\label{T:cohomology}
		Let $R$ be a ring of weak global dimension at most one, and let $\mathcal{V}$ be a definable subcategory in $\Der(R)$. Then $\mathcal{V}$ is determined on cohomology. 
	\end{thm}
	\begin{proof}
			Since $\mathcal{V}$ is definable, there is by Lemma~\ref{L:definable} a set $\Phi$ of maps from $\Der^c(\Rop)$ such that $\mathcal{V} = \{X \in \Der(R) \mid H^0(X \otimes^\mathbf{L}_R f) \text{ is zero for each $f \in \Phi$}\}$. For each $f \in \Phi$, let $f' \in \Der^c(\Rop)$ be a map such that there is a triangle of the form
			$$E \xrightarrow{f'} C \xrightarrow{f} D \rightarrow E[1]$$
			in $\Der^c(\Rop)$. By Lemma~\ref{L:kunneth}, we have for any $X \in \Der(R)$ and any $f \in \Phi$ the equivalence
			$$H^0(X \otimes^\mathbf{L}_R f) \text{ is a zero map} \iff$$ 
			$$ \iff H^n(X) \otimes_R H^{-n}(f) \text{ is zero and }\Tor{}_1^R(H^n(X),H^{1-n}(f')) \text{ is surjective} ~\forall n \in \mathbb{Z}.$$ 
			Since the latter condition is formulated just by means of the cohomology modules of $X$, we see that for any $X \in \Der(R)$ we have the equivalence
			$$X \in \mathcal{V} \iff \prod_{n \in \mathbb{Z}} H^n(X)[-n] \in \mathcal{V}.$$
			As $\mathcal{V}$ is closed under products and direct summands, it follows that $\mathcal{V}$ is determined on cohomology.
	\end{proof}
	\begin{prop}\label{P:definablecorr}
		Let $R$ be a ring of weak global dimension at most one. Then there is a 1-1 correspondence:
		$$\left \{ \begin{tabular}{ccc} \text{ definable subcategories $\mathcal{V}$} \\ \text{in $\Der(R)$} \end{tabular}\right \}  \leftrightarrow  \left \{ \begin{tabular}{ccc} \text{ collections $\{\mathcal{V}_n \mid n \in \mathbb{Z}\}$ of } \\ \text{ definable subcategories of $\ModR$} \end{tabular}\right \}.$$
			The correspondence is given by assignments
			$$\mathcal{V} \mapsto \mathcal{V}_n = \{H^n(X) \mid X \in \mathcal{V}\} ~\forall n \in \mathbb{Z},$$
			$$\{\mathcal{V}_n \mid n \in \mathbb{Z}\} \mapsto \mathcal{V} = \{X \in \Der(R) \mid H^n(X) \in \mathcal{V}_n ~\forall n \in \mathbb{Z}\}.$$
	\end{prop}
	\begin{proof}
			By Theorem~\ref{T:cohomology}, any definable subcategory $\mathcal{V}$ is determined on cohomology, and thus $\mathcal{V}$ is uniquely determined by a collection of subcategories $\mathcal{V}_n = \{H^n(X) \mid X \in \mathcal{V}\}$, $n \in \mathbb{Z}$. Also, since $\mathcal{V}_n[-n] \subseteq \mathcal{V}$ for all $n \in \mathbb{Z}$, then clearly the classes $\mathcal{V}_n$ are closed under direct products and direct limits by the closure properties of $\mathcal{V}$. Recall that any pure-exact sequence in $\ModR$ becomes a pure triangle in $\Der(R)$, this follows e.g. from the characterization \cite[Proposition 3.7]{L} of pure triangles together with \cite[Theorem 16.1.16]{P}. Thus, since $\mathcal{V}$ is closed under pure monomorphisms in $\Der(R)$, it follows that $\mathcal{V}_n$ is closed under pure submodules in $\ModR$. Therefore, $\mathcal{V}_n$ is a definable subcategory of $\ModR$ for each $n \in \mathbb{Z}$. 
			
			Conversely, let $\{\mathcal{V}_n \mid n \in \mathbb{Z}\}$ be any collection of definable subcategories of $\ModR$ and let us prove that $\mathcal{V} = \{X \in \Der(R) \mid H^n(X) \in \mathcal{V}_n ~\forall n \in \mathbb{Z}\}$ is a definable subcategory of $\Der(R)$. Let $\mathscr{X} \in \Der((\ModR)^I)$ be a coherent diagram of a directed shape $I$ such that $\mathscr{X}_i \in \mathcal{V}$ for all $i \in I$. In particular, $H^n(\mathscr{X}_i) \in \mathcal{V}_n$ for all $n \in \mathbb{Z}$. Then $H^n(\hocolim_{i \in I} \mathscr{X}) \simeq \varinjlim_{i \in I}H^n(\mathscr{X}_i) \in \mathcal{V}_n$ for all $n \in \mathbb{Z}$, and thus $\hocolim_{i \in I} \mathscr{X} \in \mathcal{V}$. Similar argument shows that $\mathcal{V}$ is closed under products. Finally, consider a pure monomorphism $f: Y \rightarrow X$ in $\Der(R)$ with $X \in \mathcal{V}$. For each $n \in \mathbb{Z}$ we have that $\Hom_R(R[-n],f) \simeq H^n(f)$ is a pure monomorphism of $R$-modules by \cite[17.3.17]{P}. Therefore, $H^n(Y) \in \mathcal{V}_n$ for all $n \in \mathbb{Z}$, and therefore $Y \in \mathcal{V}$. Using Theorem~\ref{T:rosie1} we conclude that $\mathcal{V}$ is a definable subcategory. This establishes the correspondence.
	\end{proof}
	\subsection{Ziegler spectra}
	We can reformulate Proposition~\ref{P:definablecorr} using the Ziegler spectra of the derived category and of the module category. We refer to \cite{P} for the theory of Ziegler spectra of module categories. If $R$ is a ring, the natural embedding $\ModR[-n] \subseteq \Der(R)$ for some $n \in \mathbb{Z}$ induces a closed embedding $\Zgs(R)[-n] \rightarrow \Zgs(\Der(R))$. Clearly, $\bigcup_{n \in \mathbb{Z}}\Zgs(R)[-n]$ forms a disjoint union inside $\Zgs(\Der(R))$. One can ask for which rings it is true that $\Zgs(\Der(R)) = \bigcup_{n \in \mathbb{Z}}\Zgs(R)[-n]$. Equivalently, for which rings is it true that every indecomposable pure-injective object inside $\Der(R)$ is quasi-isomorphic to a stalk complex. This is not true in general, but it is known to hold for example for right hereditary or von Neumann regular rings, see \cite[17.3.22 and 17.3.23]{P}. The following provides a common generalization for those two results.
	\begin{cor}
		Let $R$ be a ring of weak global dimension at most one. Then we have $\Zgs(\Der(R)) = \bigcup_{n \in \mathbb{Z}} \Zgs(R)[-n]$.
	\end{cor}
	\begin{proof}
			Let $\mathcal{V}$ be the definable subcategory of $\Der(R)$ corresponding to the closed subset $U = \bigcup_{n \in \mathbb{Z}} \Zgs(R)[-n]$ of $\Zgs(\Der(R))$ (cf. \cite[Theorem 17.3.20]{P}). As $\Zgs(R)[-n] \subseteq U$, we see that $\ModR[-n] \subseteq \mathcal{V}$ for all $n \in \mathbb{Z}$. But this means that $\mathcal{V}_n = \{H^n(X) \mid X \in \mathcal{V}\} = \ModR$ for all $n \in \mathbb{Z}$, which in turns means that $\mathcal{V} = \Der(R)$ by Proposition~\ref{P:definablecorr}, and therefore $U = \Zgs(\Der(R))$.
	\end{proof}
	\subsection{Definable coaisles}
	In the rest of the paper, we will be concerned with the definable subcategories which are coaisles of t-structures, that is, in view of Corollary~\ref{C:closure}, definable subcategories of $\Der(R)$ closed under extensions and cosuspensions. We therefore restrict the correspondence of Proposition~\ref{P:definablecorr} to such definable subcategories.
	\begin{prop}\label{P:moduletheoretic}
			Let $R$ be a ring of weak global dimension at most one. The 1-1 correspondence of Proposition~\ref{P:definablecorr} restricts to another 1-1 correspondence between the following collections:
			\begin{enumerate}
				\item[(i)] definable coaisles $\mathcal{V}$ in $\Der(R)$,
				\item[(ii)] increasing sequences $\cdots \mathcal{V}_n \subseteq \mathcal{V}_{n+1} \subseteq \cdots$ of definable subcategories closed under extensions in $\ModR$ indexed by $n \in \mathbb{Z}$, satisfying the following condition: Whenever $f: V_n \rightarrow V_{n+1}$ is a map with $V_n \in \mathcal{V}_n$ and $V_{n+1} \in \mathcal{V}_{n+1}$ for some $n \in \mathbb{Z}$, then $\Ker(f) \in \mathcal{V}_n$ and $\Coker(f) \in \mathcal{V}_{n+1}$.
			\end{enumerate}
			\begin{proof}
	Let $\mathcal{V}$ be a definable coaisle and let $\{\mathcal{V}_n \mid n \in \mathbb{Z}\}$ be the sequence of definable subcategories of modules corresponding to $\mathcal{V}$ via Proposition~\ref{P:definablecorr}. Since $\mathcal{V}_n[-n] = \mathcal{V} \cap \ModR[-n]$ by Theorem~\ref{T:cohomology}, this already implies that $\mathcal{V}_n$ is closed under extensions, and that $\mathcal{V}_n \subseteq \mathcal{V}_{n+1}$ for each $n \in \mathbb{Z}$. Suppose now that $f: V_n \rightarrow V_{n+1}$ is a map as in the condition $(ii)$. Then $f$ induces a triangle
					$$V_n[-n-1] \xrightarrow{f[-n-1]} V_{n+1}[-n-1] \rightarrow Z \rightarrow V_n[-n]$$
					in $\Der(R)$. Since $V_{n} \in \mathcal{V}_n$ and $V_{n+1} \in \mathcal{V}_{n+1}$, we have $V_n[-n],V_{n+1}[-n-1] \in \mathcal{V}$, and thus $Z \in \mathcal{V}$. Consider the following part of the long exact sequence of cohomologies induces by the triangle:
					$$0 \rightarrow H^n(Z) \rightarrow V_n \xrightarrow{f} V_{n+1} \rightarrow H^{n+1}(Z) \rightarrow 0.$$
					The leftmost and the rightmost term are zero, because they are equal to the cohomologies of the stalk complexes --- namely, $H^n(V_{n+1}[-n-1]), H^{n+2}(V_n[-n-1])$. Since $Z \in \mathcal{V}$, then $\Ker(f) \simeq H^n(Z) \in \mathcal{V}_n$ and $\Coker(f) \simeq H^{n+1}(Z) \in \mathcal{V}_{n+1}$, showing that the condition $(ii)$ is satisfied.

					Suppose now that $\{\mathcal{V}_n \mid n \in \mathbb{Z}\}$ is a collection of definable subcategories of $R$-modules satisfying all of the conditions in $(ii)$, and let us show that $\mathcal{V} = \{ X \in \Der(R) \mid H^n(X) \in \mathcal{V}_n ~\forall n \in \mathbb{Z}\}$ is a definable coaisle. We already know by Proposition~\ref{P:definablecorr} that $\mathcal{V}$ is definable. By Corollary~\ref{C:closure}, it is enough to check that $\mathcal{V}$ is closed under cosuspensions, and extensions. The closure under cosuspensions clearly follows from $\mathcal{V}_n \subseteq \mathcal{V}_{n+1}$ for each $n \in \mathbb{Z}$. Next, suppose that 
					$$X \rightarrow Y \rightarrow Z \rightarrow X[1]$$
					is a triangle with $X,Z \in \mathcal{V}$, and consider the long exact sequence on cohomologies:
					$$\cdots H^{n-1}(Z) \xrightarrow{f} H^n(X) \rightarrow H^n(Y) \rightarrow H^n(Z) \xrightarrow{g} H^{n+1}(X) \rightarrow \cdots$$
				By the assumption from $(ii)$, we have that $\Ker(g) \in \mathcal{V}_n$, and $\Coker(f) \in \mathcal{V}_n$. Because $\mathcal{V}_n$ is closed under extensions, this implies that $H^n(Y) \in \mathcal{V}_n$ using the short exact sequence
					$$0 \rightarrow \Ker(g) \rightarrow H^n(Y) \rightarrow \Coker(f) \rightarrow 0.$$
					Therefore, $H^n(Y) \in \mathcal{V}_n$ for all $n \in \mathbb{Z}$, and thus $Y \in \mathcal{V}$.
			\end{proof}
	\end{prop}
	\begin{conv}
			Given a coaisle $\mathcal{V}$ in $\Der(R)$, we will from now on always implicitly use the notation $\mathcal{V}_n = \{H^n(X) \mid X \in \mathcal{V}\}$ for the essential image of the $n$-th cohomology functor of $\mathcal{V}$ in $\ModR$.
	\end{conv}
	\begin{remark}
		A similar condition to (ii) of Proposition~\ref{P:moduletheoretic} appears in a slightly different formulation in \cite{AR}, where sequences of subcategories of the module category determining a coaisle of a t-structure over a hereditary ring are called ``reflective co-narrow sequences''. In our setting, the reflectivity is ensured by the definability of the members of the sequence.
	\end{remark}
	\subsection{Compactly generated coaisles}
	Under some extra conditions, we are also able to prove a useful criterion for deciding whether a definable coaisle is compactly generated, meaning by this that the associated t-structure is compactly generated.
	\begin{prop}\label{P:compinj}
			Let $R$ be a ring of weak global dimension at most one, and let $\mathcal{V}$ be a definable coaisle in $\Der(R)$. Consider the two following conditions:
			\begin{enumerate}
				\item[(i)] $\mathcal{V}$ is compactly generated,
				\item[(ii)] $\mathcal{V}_n$ is closed under taking injective envelopes for all $n \in \mathbb{Z}$.
			\end{enumerate}
			If $R$ is commutative, then $(i) \iff (ii)$. If $R$ is right semihereditary, then $(ii) \implies (i)$.
	\end{prop}
	\begin{proof}
			Let us start by assuming $(ii)$. For each $n \in \mathbb{Z}$, let $\mathcal{C}_n$ be the closure of the class $\mathcal{V}_n$ under submodules. One can argue the same way as in \cite[Lemma 5.6]{HS} that $\mathcal{C}_n$ is a definable torsion-free class closed under injective envelopes. In particular, there is a hereditary torsion pair $(\mathcal{T}_n,\mathcal{C}_n)$ for each $n \in \mathbb{Z}$. Using \cite[Lemma 2.4]{H}, it follows that there is for each $n \in \mathbb{Z}$ a set $\mathcal{I}_n$ of finitely generated ideals such that $\mathcal{C}_n = \{R/I \mid I \in \mathcal{I}_n\}\Perp{0}$. Put $\mathcal{V}' = \{R/I[-n] \mid ~\forall I \in \mathcal{I}_n ~\forall n \in \mathbb{Z}\}\Perp{0}$, which defines a coaisle of a t-structure. If $R$ is right semihereditary, then $R/I[-n]$ is a compact object in $\Der(R)$ for any finitely generated ideal $I$. If $R$ is commutative, then $\mathcal{V}'$ can be written as a right orthogonal to a set of suspensions of Koszul complexes (see \cite[Lemma 5.4]{HCG}). In both cases, $\mathcal{V}'$ is a compactly generated coaisle, and therefore is determined on cohomology by Theorem~\ref{T:cohomology}. Let $(\mathcal{V}'_n \mid n \in \mathbb{Z})$ be the sequence of definable subcategories of $\ModR$ associated to $\mathcal{V}'$ via Proposition~\ref{P:moduletheoretic}. We will show that $\mathcal{V}_n = \mathcal{V}'_n$ for all $n \in \mathbb{Z}$. The subcategory $\mathcal{V}'_n = \bigcap_{i \geq n} \bigcap_{I \in \mathcal{I}_n} \Ker \Ext_R^{i-n}(R/I,-)$ of $\ModR$ is closed under injective envelopes for all $n \in \mathbb{Z}$. Therefore, both the subcategories $\mathcal{V}_n$ and $\mathcal{V}'_n$ are closed under injective envelopes, and by the construction they contain the same injective objects. For any module $M \in \mathcal{V}_n$, we consider the minimal injective coresolution 
			$$0 \rightarrow M \rightarrow E_0 \rightarrow E_1 \rightarrow E_2 \rightarrow \cdots.$$
			By induction, it follows that $E_k \in \mathcal{V}_{n+k}$ for all $k \geq 0$, and therefore $E_k \in \mathcal{V}'_{n+k}$ for all $k \geq 0$. But that implies $M \in \mathcal{V}'_n$, and thus $\mathcal{V}_n \subseteq \mathcal{V}'_n$. A symmetrical argument shows that $\mathcal{V}'_n \subseteq \mathcal{V}_n$ for all $n \in \mathbb{Z}$, proving that $\mathcal{V} = \mathcal{V}'$ by Proposition~\ref{P:moduletheoretic}.

			Finally, suppose that $R$ is commutative, and that $\mathcal{V}$ is a compactly generated coaisle. Because $\mathcal{V}_n[-n] \subseteq \mathcal{V}$, we have that $\mathcal{V}_n$ is closed under injective envelopes by \cite[Lemma 3.3]{HCG}.
	\end{proof}
		\begin{remark}
			The semiheredity imposed on the ring $R$ in the last part of Proposition~\ref{P:compinj} can be weakened to the following condition: $R$ is right coherent and any finitely presented cyclic $R$-module has a finite projective dimension. Indeed, this is enough to ensure that any finitely presented cyclic module $R/I$ is compact as an object of the derived category.
	\end{remark}
	
	At this point we are ready to answer Question~\ref{Q:UTC} in the affirmative for any (not necessarily commutative) von Neumann regular ring. In particular, the telescope conjecture holds for these rings, generalizing \cite[Theorem 4.21]{GS} and the corresponding result in \cite{BS}.
	\begin{cor}\label{C:VNR}
		Let $R$ be a von Neumann regular ring. Then any definable coaisle in $\Der(R)$ is compactly generated. In particular, the Telescope Conjecture holds for $R$.
	\end{cor}
	\begin{proof}
		Recall that over any von Neumann regular ring, the injective envelopes coincide with the pure-injective envelopes. Therefore, any definable subcategory of $\ModR$ is closed under injective envelopes by \cite[Theorem 3.4.8]{P}. Since $R$ is semihereditary, the rest follows from Proposition~\ref{P:compinj}.
	\end{proof}
	\subsection{Definable coaisles induced by homological epimorphisms}\label{SS:homological}
	There is a general construction described in \cite[\S 5]{AH} which assigns a definable coaisle to a double-infinite chain of homological ring epimorphisms based in a ring of weak global dimension at most one. We refer the reader to \textit{loc. cit.} for more details. Recall that a \EMP{homological ring epimorphism} is an epimorphism $\lambda: R \rightarrow S$ in the category of rings, such that $\Tor_i^R(S,S) = 0$ for all $i>0$. Equivalently, this means that the forgetful functor $\ModS \rightarrow \ModR$ induces a fully faithful functor $\Der(S) \rightarrow \Der(R)$. By a \EMP{chain of homological ring epimorphisms} we mean an $\mathbb{Z}$-indexed chain
	\begin{equation}\label{E:chain}\cdots \leftarrow S_{n-1} \xleftarrow{\mu_{n-1}} S_{n} \xleftarrow{\mu_n} S_{n+1} \leftarrow \cdots\end{equation}
			of ring epimorphisms such that there are homological ring epimorphisms $\lambda_n: R \rightarrow S_n$, and such that $\mu_{n-1} \lambda_n = \lambda_{n-1}$ for all $n \in \mathbb{Z}$. This in particular implies that $\mu_n: S_{n+1} \rightarrow S_{n}$ is a homological epimorphism for each $n \in \mathbb{Z}$. A subcategory $\mathcal{B}$ of $\ModR$ is called \EMP{bireflective} if it is closed under products, coproducts, kernels, and cokernels. Equivalently, it is a subcategory $\mathcal{B}$ of $\ModR$ such that the inclusion $\mathcal{B} \subseteq \ModR$ admits both the left and the right adjoint, called the \EMP{reflection} and \EMP{coreflection}, respectively. Recall that two ring epimorphisms $\lambda: R \rightarrow S$ and $\sigma: R \rightarrow S'$ are in the same \EMP{epiclass} if there is a ring isomorphism $\iota: S \rightarrow S'$ such that $\sigma = \iota\lambda$. Then we have the following result:
	\begin{thm}\emph{(\cite[Proposition 4.2]{BS})}\label{T:epiclass}
			Let $R$ be a ring of weak global dimension at most one. Then the assignment 
			$$(\lambda: R \rightarrow S) \mapsto \ModS \simeq \operatorname{Im}(- \otimes_R S) \subseteq \ModR$$ 
			induces a bijection between the following sets: 
			\begin{enumerate}
				\item[(i)] epiclasses of homological ring epimorphisms $\lambda: R \rightarrow S$, 
				\item[(ii)] extension-closed bireflective subcategories $\mathcal{B}$ of $\ModR$.
			\end{enumerate}
	\end{thm}
	Using this correspondence, it is not hard to see that if $R$ is of weak global dimension at most one, then chains of homological ring epimorphisms as in (\ref{E:chain}), up to a choice of epiclass representatives, correspond bijectively to chains
	$$\cdots \mathcal{B}_{n-1} \subseteq \mathcal{B}_n \subseteq \mathcal{B}_{n+1} \subseteq \cdots$$
	of extension-closed bireflective subcategories of $\ModR$, via the assignment $S_n \mapsto \mathcal{B}_n := \ModSn \simeq \operatorname{Im}(- \otimes_R S_n) \subseteq \ModR$. To this data, we assign a subcategory $\mathcal{V}$ of $\Der(R)$ as follows:
	$$\mathcal{V} = \{X \in \Der(R) \mid H^n(X) \in \Cogen(\mathcal{B}_n) \cap \mathcal{B}_{n+1} ~\forall n \in \mathbb{Z}\}.$$
	\begin{prop}\emph{(\cite[Proposition 5.4]{AH})}\label{P:construction}
			Let $R$ be a ring of weak global dimension at most one. Then for any chain of homological ring epimorphisms over $R$, the subcategory $\mathcal{V}$ of $\Der(R)$ defined above is a definable coaisle.
	\end{prop}
	\begin{proof}
			We include a sketch of the proof from \cite{AH} here for convenience. It is enough to check the conditions of Proposition~\ref{P:moduletheoretic} for the chain of subcategories $\mathcal{V}_n := \Cogen(\mathcal{B}_n) \cap \mathcal{B}_{n+1}$. First recall that any bireflective subcategory of $\ModR$ is definable. By \cite[3.4.15]{P}, also $\Cogen(\mathcal{B}_n)$ is definable for any $n \in \mathbb{Z}$, and therefore $\mathcal{V}_n$ is definable for any $n \in \mathbb{Z}$. Let $\lambda_n: R \rightarrow S_n$ be a homological ring epimorphism corresponding to the bireflective subcategory $\mathcal{B}_n$. Since $R$ is of weak global dimension at most one, the character dual $E_n := \Hom_{\mathbb{Z}}(S_n,\mathbb{Q}/\mathbb{Z})$ of $S_n$ as a right $R$-module is of injective dimension at most one. Since $E_n$ is an injective cogenerator of the category $\mathcal{B}_n$, it follows that $\Ext_R^1(M,E_n) = 0$ for any $M \in \Cogen(\mathcal{B}_n)$. Therefore, the class $\Cogen(\mathcal{B}_n) = \Cogen(E_n)$ is closed under extensions. Together, $\mathcal{V}_n$ is a definable subcategory of $\ModR$ closed under extensions.

			Let now $f: M \rightarrow N$ be a map such that $M \in \mathcal{V}_n$ and $N \in \mathcal{V}_{n+1}$, and let us show that $\Ker(f) \in \mathcal{V}_n$ and $\Coker(f) \in \mathcal{V}_{n+1}$. Consider the induced exact sequences:
			$$0 \rightarrow K \rightarrow M \rightarrow I \rightarrow 0, \text{ and}$$
			$$0 \rightarrow I \rightarrow N \rightarrow C \rightarrow 0,$$
			where $K = \Ker(f)$, $C = \Coker(f)$, and $I = \operatorname{Im}(f)$. Then clearly $K \in \Cogen(\mathcal{B}_n)$, and $I \in \Cogen(\mathcal{B}_{n+1})$. Also, as $I$ is an epimorphic image of $M \in \mathcal{B}_{n+1}$, it follows by a diagram chasing argument that the reflection $I \rightarrow I_{\mathcal{B}_{n+1}} \in \mathcal{B}_{n+1}$ of $I$ with respect to the subcategory $\mathcal{B}_{n+1}$ is an isomorphism, and thus $I \in \mathcal{B}_{n+1}$. Since $K$ is the kernel of the morphism $M \rightarrow I$  between two objects in $\mathcal{B}_{n+1}$, then $K \in \mathcal{B}_{n+1}$. Thus, $K \in \mathcal{V}_n$. The Four Lemma implies that the reflection $C \rightarrow C_{\mathcal{B}_{n+1}}$ is a monomorphism, and therefore $C \in \Cogen(\mathcal{B}_{n+1})$. Finally, as $C$ is an epimorphic image of $N \in \mathcal{B}_{n+2}$, it follows again that $C \in \mathcal{B}_{n+2}$.
	\end{proof}
	\begin{remark}\label{R:homepi}In other words, we have an assignment from the chains of homological epimorphisms over $R$ to definable coaisles in $\Der(R)$. It is straightforward to extend the notion of epiclass to introduce an equivalence relation of chains of epimorphisms, and then the induced assignment is easily checked to be injective. In general however, this assignment is not surjective, and there are definable coaisles which do not arise in this way. For the case of valuation domains, this will be discussed in Section~\ref{S:nondense}.
	\end{remark}
\section{Valuation domains and the module-theoretic cosilting classes}
	From now on we will focus on commutative rings of weak global dimension at most one. We will do most of the investigation in the local case, that is, over a valuation domain. \emph{A posteriori}, this will be enough to fully answer Question~\ref{Q:UTC} even in the global case. In this section, we start by studying the definable coaisles in the Happel-Reiten-Smal\o ~situation, which in the light of Section~\ref{S:cosilting} amounts to studying the cosilting classes in the module category. The main aim of this section is to build on the results from \cite{B} and \cite{B2} and establish for any valuation domain a bijective correspondence between cosilting classes and certain systems of formal intervals in the Zariski spectrum.
\subsection{Valuation domains} 
	A commutative domain $R$ is a \EMP{valuation domain} if the ideals of $R$ are totally ordered. We gather some basic properties of valuation domains which we will use freely throughout the paper. Given a prime ideal $\qq$ of a commutative ring $R$, we let $R_{\qq}$ denote the localization of $R$ at $\qq$, and more generally, $M_{\qq} = M \otimes_R R_{\qq}$ the localization of an $R$-module $M$ at $\qq$.
	\begin{lem}\label{L:VD}
		\begin{enumerate}
				\item[(i)] Valuation domains are precisely the local commutative rings of weak global dimension at most one.
				\item[(ii)] Any idempotent ideal in a valuation domain is a prime ideal.
				\item[(iii)] If $\pp \subseteq \qq$ are primes of a valuation domain $R$, then $\pp$ is an $R_{\qq}$-module.
				\item[(iv)] Whenever $S \subseteq \Spec(R)$ is a non-empty subset with no maximal element with respect to $\subseteq$, then $\bigcup S$ is an idempotent prime.
				\item[(v)] For any prime $\pp \in \Spec(R)$, either $\pp$ is idempotent or $\pp R_{\pp}$ is a principal ideal in $R_{\pp}$.
		\end{enumerate}
	\end{lem}
	\begin{proof}
		\begin{enumerate}
			\item[(i)] See \cite[Corollary 4.2.6]{Gl}. 
			\item[(ii)] Obvious.
			\item[(iii)] Obvious.
			\item[(iv)] This is \cite[Lemma 5.3]{BS}.
			\item[(v)] See \cite[\S II, Lemma 4.3 (iv) and (d), p.69]{FS}.
		\end{enumerate}
	\end{proof}
\subsection{Torsion, annihilators, divisibility, and socle}
 Let $R$ be a commutative ring and $\qq$ a prime in $\Spec(R)$. For any $R$-module $M$ and an element $m \in M$, let $\Ann_R(m) = \{r \in R \mid rm = 0\}$, and similarly we put $\Ann_R(M) = \{r \in R \mid rM = 0\}$, both these define ideals of $R$. There is a torsion pair $(\mathcal{T}_{\qq},\mathcal{F}_{\qq})$ in $\ModR$, where $\mathcal{T}_{\qq}$ consists of all modules $M$ such that $\Ann_R(m)$ contains an element from $R \setminus \qq$ for any $m \in M$. This torsion pair is hereditary, that is $\mathcal{T}_{\qq}$ is closed under submodules and $\mathcal{F}_{\qq}$ is closed under taking injective envelopes. Modules from $\mathcal{T}_{\qq}$ will be called \EMP{$\qq$-torsion}, and modules from $\mathcal{F}_{\qq}$ are \EMP{$\qq$-torsion-free}. We denote the torsion functor induced by this torsion pair by $\Gamma_{\qq}: \ModR \rightarrow \mathcal{T}_{\qq}$, and the torsion-free counterpart by $F_{\qq}: \ModR \rightarrow \mathcal{F}_{\qq}$. Recall that $\Gamma_{\qq}$ is a left exact functor, while $F_{\qq}$ preserves monomorphisms and epimorphisms, a fact which we will use freely throughout the paper. We call an $R$-module $M$ \EMP{$\qq$-divisible} provided that $M=sM$ for all $s \in R \setminus \qq$. Recall that an $R$-module $M$ is an $R_{\qq}$-module if $M$ is both $\qq$-torsion-free and $\qq$-divisible.

 It will be useful to recall that given an $R$-module $M$ and a prime ideal $\qq \in \Spec(R)$, we have the natural identifications $\Gamma_{\qq}(M) = \Ker(M \xrightarrow{\text{can}} M \otimes_R R_{\qq})$ and $F_{\qq}(M) \simeq \operatorname{Im}(M \xrightarrow{\text{can}} M \otimes_R R_{\qq})$. Also, note that $\Coker(M \xrightarrow{\text{can}} M \otimes_R R_{\qq}) = 0$ if and only if $F_{\qq}(M)$ is $\qq$-divisible if and only if $F_{\qq}(M) \in \ModRqq$.

 Given a prime ideal $\pp$ and a module $M$, we define the \EMP{$\pp$-socle} of $M$ to be the submodule $\Soc_{\pp}(M) = \{m \in M \mid rm = 0 ~\forall r \in \pp\}$ of $M$.
\subsection{Systems of intervals of $\Spec(R)$} 
Let $R$ be a valuation domain. By an \EMP{interval} in $\Spec(R)$ we mean a formal interval $\chi=[\pp_\chi,\qq_\chi]$, where $\pp_\chi \subseteq \qq_\chi$ are primes from $\Spec(R)$. We consider intervals together with a partial order $<$ defined as follows: for intervals $\chi=[\pp_\chi,\qq_\chi]$ and $\xi=[\pp_\xi,\qq_\xi]$ we have $\chi<\xi$ if and only if $\qq_\chi \subsetneq \pp_\xi$. Any interval denoted by a greek letter will have boundaries denoted like above, e.g. $\theta = [\pp_\theta,\qq_\theta]$ etc. In other occasions, we will denote intervals just by their boundaries, that is, by writing just $[\pp,\qq]$ for a couple of primes $\pp \subseteq \qq$ of $\Spec(R)$.
\begin{definition}\label{D:admissiblesystem}
Following \cite{B2}, we impose the following conditions on a set $\mathcal{X}$ of intervals of $\Spec(R)$:
\begin{enumerate}
		\item[(i)] (\EMP{disjointness}) The system is disjoint, that is, whenever $\chi,\xi \in \mathcal{X}$ are two distinct intervals such that $\pp_\chi \subseteq \pp_\xi$ then $\qq_\chi \subsetneq \pp_\xi$.
		\item[(ii)] (\EMP{idempotency}) For any $\chi \in \mathcal{X}$ we have $\pp_\chi = \pp_\chi^2$.
		\item[(iii)] (\EMP{completeness}) For any non-empty subset $\mathcal{Y} \subseteq \mathcal{X}$, there is an interval $\mu \in \mathcal{X}$ such that $\pp_\mu = \bigcup_{\chi \in \mathcal{Y}}\pp_\chi$, and there is an interval $\nu \in \mathcal{X}$ such that $\qq_\nu = \bigcap_{\chi \in \mathcal{Y}}\qq_\chi$.
\end{enumerate}
Let us call a system of intervals satisfying these conditions an \EMP{admissible system}. We remark that as a consequence of the definition, any admissible system $\mathcal{X}$ together with the above defined partial order $<$ forms a totally ordered set $(\mathcal{X},<)$ such that any non-empty subset $\mathcal{Y}$ of $\mathcal{X}$ has a supremum and an infimum.
\end{definition}

\subsection{From intervals to cosilting classes}
		Recall that given an ideal $I$, we define the prime ideal \EMP{attached} to $I$ as $I^\#=\{r \in R \mid rI \subsetneq I\}$. By $Q$ we always denote the quotient field $Q(R)$ of the valuation domain $R$. It will be also useful to extend the definition of an attached prime to any submodule of $Q$. If $J \subseteq Q$ is an $R$-submodule, then $J^\# = \{r \in R \mid rJ \neq J\} = \bigcup_{r \in Q \setminus J}r^{-1}J$.
		\begin{notation}	For any interval $\chi$ in $\Spec(R)$, we write $\langle \chi \rangle = \langle \pp_\chi, \qq_\chi \rangle$ for the set of all ideals $I$ of $R$ satisfying $\pp_\chi \subseteq I \subseteq I^\# \subseteq \qq_\chi$. \end{notation}
\begin{lem}\label{L00}
		Let $\mathcal{X}$ be an admissible system, let $\Lambda$ be a cardinal and let $\{I_\lambda, \lambda \in \Lambda\}$ be a set of ideals such that for each $\lambda \in \Lambda$ there is $\chi_\lambda \in \mathcal{X}$ with $I_\lambda \in \langle \chi_\lambda \rangle$. Then:
	\begin{itemize}
			\item[(i)] there is $\xi \in \mathcal{X}$ such that $\bigcap_{\lambda \in \Lambda} I_\lambda \in \langle \xi \rangle$,
			\item[(ii)] there is $\xi \in \mathcal{X}$ such that $\bigcup_{\lambda \in \Lambda} I_\lambda \in \langle \xi \rangle$.
	\end{itemize}
\end{lem}
\begin{proof}
		$(i)$ Denote $I = \bigcap_{\lambda \in \Lambda}I_\lambda$. Obviously, we have $\bigcap_{\lambda \in \Lambda}\pp_{\chi_\lambda} \subseteq I  \subseteq\bigcap_{\lambda \in \Lambda}\qq_{\chi_\lambda}$. By the completeness, there is $\xi \in \mathcal{X}$ with $\qq_\xi = \bigcap_{\lambda \in \Lambda} \qq_{\chi_\lambda}$. It is then enough to prove that $\pp_\xi \subseteq I$ and $I^\# \subseteq \qq_\xi$, which we do by distinguishing two cases: 
		\begin{enumerate}
				\item[Case I:] There is $\lambda \in \Lambda$ such that $\qq_\xi = \qq_{\chi_\lambda}$. Then we have $\pp_\xi \subseteq I \subseteq \qq_\xi$, and we are left to show that $I^\# \subseteq \qq_\xi$. Then we can assume without loss of generality that $I_\lambda \in \langle \xi \rangle = \langle \pp_\xi,\qq_\xi \rangle$ for all $\lambda \in \Lambda$. Therefore, for any $r \in R \setminus \qq_\xi$ and any $i \in I$, $r^{-1}i \in I_\lambda$ for all $\lambda \in \Lambda$. It follows that $r^{-1}i \in I$ for any $i \in I$, and thus $r \not\in I^\#$ for all $r \in R \setminus \qq_\xi$, proving that $I \in \langle \xi \rangle$.
				\item[Case II:] There is no $\lambda \in \Lambda$ such that $\qq_\xi = \qq_{\chi_\lambda}$. By the disjointness of $\mathcal{X}$, we have that necessarily
						$$\bigcap_{\lambda \in \Lambda}\pp{}_{\chi_\lambda} = I = \bigcap_{\lambda \in \Lambda} \qq{}_{\chi_\lambda},$$
						and thus, in particular, $I$ is a prime ideal, and whence $I = I^\#$ by \cite[p. 70]{FS}, which establishes that $I \in \langle \pp_\xi, \qq_\xi \rangle$.
		\end{enumerate}
	$(ii)$ Completely analogous. 
\end{proof}
Next, we explain what exactly it means for an ideal $I$ that $I^\# \subseteq \qq$ for some prime $\qq$.
\begin{lem}\label{L23}
	Let $I$ be a proper ideal of $R$. Then $I^\#$ is a prime ideal and the following conditions are equivalent for any $\qq \in \Spec(R)$:
\begin{enumerate}
	\item[(i)] $I^\# \subseteq \qq$,
	\item[(ii)] $I$ is an $R_{\qq}$-module,
	\item[(iii)] $R/I$ is a $\qq$-torsion-free $R$-module, i.e. $\Gamma_{\qq}(R/I) = 0$.
\end{enumerate}
\end{lem}
\begin{proof}
	That $I^\#$ is a prime ideal is clear.
	
	$(i) \Leftrightarrow (ii)$: For a given $\qq \in \Spec(R)$, the canonical map $f: I \rightarrow I_{\qq}$ is injective since it is the restriction of the canonical map $R \rightarrow R_{\qq}$. Therefore $I$ is an $R_{\qq}$-module if and only if $f$ is surjective, which amounts to say that for each $y \in I$ and each $s \in R \setminus \qq$, there exists a $y' \in I$ such that $sy' = y$. That is, if and only if the equality $sI = I$ holds, for all $s \in R \setminus \qq$, if and only if $I^\# \subseteq \qq$.

	$(ii) \Rightarrow (iii)$: Clear.

	$(iii) \Rightarrow (i)$: The canonical map $R/I \rightarrow (R/I)_{\qq}$ is injective since its kernel is $\Gamma_{\qq}(R/I) = 0$. This amount to say that $(I : s) = I$, for all $s \in R \setminus \qq$, where $(I : s) = \{a \in R: sa \in I\}$. It follows that $R \setminus \qq \subseteq R \setminus I$ and so $Rs \not\subseteq I$, which implies that $I \subseteq Rs$ due to the totally ordered condition of the lattice of ideals of $R$. If now $y \in I$ and we write $y = sa$, with $a \in R$, then $a \in (I : s) = I$ and so $I = sI$. That is, we have $s \in R \setminus I^\#$ for all $s \in R \setminus \qq$, and hence $I^\# \subseteq \qq$.
	\end{proof}

\begin{lem}\label{L01}
		Let $\mathcal{X}$ be an admissible system of intervals on $\Spec(R)$. Then the class
		$$\mathcal{C}_\mathcal{X} = \{M \in \ModR \mid \forall 0 \neq m \in M ~\exists \chi \in \mathcal{X}: \Ann{}_R(m) \in \langle \chi \rangle\}$$
		is cosilting.
	
	Furthermore, an $R$-module $M$ belongs to $\mathcal{C}_\mathcal{X}$ if and only if for each non-zero element $m$ of $M$ there is $\chi \in \mathcal{X}$ such that $\pp_{\chi}m = 0$ and $mR$ is $\qq_{\chi}$-torsion-free. In particular, if $I$ is a proper ideal of $R$ then $R/I \in \mathcal{C}_\mathcal{X}$ if and only if $I \in \langle \chi \rangle$ for some $\chi \in \mathcal{X}$.
\end{lem}
\begin{proof}
	First, we remark that the equivalence of the two descriptions of the class $\mathcal{C}_\mathcal{X}$ in the statement follows from Lemma~\ref{L23}.

	We will check that $\mathcal{C} = \mathcal{C}_\mathcal{X}$ is closed under subobjects, direct products, extensions, and direct limits.
	
	\begin{enumerate}
			\item[(a)] Subobjects: Obvious.
			\item[(b)] Products: Follows from Lemma~\ref{L00}(i).
			\item[(c)] Extensions: Suppose that 
			$$0 \rightarrow X \rightarrow Y \xrightarrow{\pi} Z \rightarrow 0$$
					is an exact sequence with $X,Z \in \mathcal{C}$, and let $y \in Y$ be a non-zero element, and let $I = \Ann_R(y)$. Restricting $\pi$ to the cyclic submodule $yR$ yields an exact sequence of the form
					\begin{equation}\label{E12}
							0 \rightarrow J/I \rightarrow R/I \rightarrow R/J \rightarrow 0,
					\end{equation}
					where $J/I, R/J \in \mathcal{C}$. Let $K = \Ann_R(J/I) = \bigcap_{m \in J/I} \Ann_R(m)$. By the definition of $\mathcal{C}$ and by Lemma~\ref{L00}(i), there are $\chi$ and $\xi$ such that $J \in \langle \chi \rangle$ and $K \in \langle \xi \rangle$. We show that necessarily $\xi \leq \chi$. Indeed, $K = \{r \in R \mid rJ \subseteq I\} \subseteq J^\# \subseteq \qq_\chi$, and since $K \subseteq \qq_\xi$, we have the desired inequality. Since $\pp_\xi = \pp_\xi^2 \subseteq JK \subseteq I \subseteq J \cap K \subseteq \qq_\xi$, it is enough to show that $I^\# \subseteq \qq_\xi$. In view of Lemma~\ref{L23}, it is enough to show that $R/I$ is $\qq_{\xi}$-torsion-free. If $\Gamma_{\qq_\xi}(R/I)$ of $R/I$ is non-zero, than by uniseriality of $R/I$ it has to intersect $J/I$ non-trivially, and thus $\Gamma_{\qq_\xi}(J/I) \neq 0$. Since $J/I$ is uniserial, it can be written as a directed union $\bigcup_{\lambda \in \Lambda}R/K_\lambda$ of cyclic submodules, in particular, $K = \bigcap_{\lambda \in \Lambda}K_\lambda$. Since $\Gamma_{\qq_\xi}(J/I) \neq 0$, there is an $s \in (R \setminus \qq_\xi)$ such that $R/K_\lambda$ contains a non-zero element killed by  $s$ for any $\lambda$ from a cofinite subset of $\Lambda$. As $s \not\in \qq_\xi$, and $K = \bigcap_{\lambda \in \Lambda}K_\lambda$, and $J/I \in \mathcal{C}$, there is $\lambda \in \Lambda$ such that $K_\lambda \subseteq K_\lambda^\# \subseteq sR$, and thus $R/K_\lambda$ cannot contain a non-zero element killed by $s$ by Lemma~\ref{L23}, a contradiction. Therefore, $I^\# \subseteq \qq_\xi$, and $I \in \langle \xi \rangle$ as desired.

			\item[(d)] Direct limits: We already know that $\mathcal{C}$ is closed under submodules and products, and thus $\mathcal{C}$ is closed under direct sums. It is then enough to show that $\mathcal{C}$ is closed under pure epimorphic images. Let $\pi: N \twoheadrightarrow_* M$ be a pure epimorphism with $N \in \mathcal{C}$, and let $m \in M$ be non-zero element with annihilator $I$. For each $i \in I$, there is the natural surjection $\sigma_i: R/iR \rightarrow R/I$. As $R/iR$ is finitely presented, the composition $\pi\sigma_i: R/iR \rightarrow M$ admits a factorization through $\pi$:
$$
		\begin{tikzcd}
			N\arrow{r}{\pi} & M \\
			R/iR \arrow{r}{\sigma_i}\arrow[dotted]{u} & R/I \arrow{u}{\subseteq} \\
		\end{tikzcd}
$$
					Therefore, for each $i \in I$ there is an ideal $J_i$ such that $J_i$ is an annihilator of an element of $N$, and $iR \subseteq J_i \subseteq I$. Then $I = \bigcup_{i \in I} J_i$, and because $J_i \in \langle \chi_i \rangle$ for some $\chi_i \in \mathcal{X}$ by the definition of $\mathcal{C}$, Lemma~\ref{L00}(ii) implies that $I \in \langle \mu \rangle$ for some $\mu \in \mathcal{X}$.
	\end{enumerate}
\end{proof}
\subsection{From cosilting classes to intervals}\label{SS:phipsi}
Here we follow \cite{B} and \cite{B2}. We start with a cosilting class $\mathcal{C}$ and assign to it the set $\mathcal{G}=\{I \text{ ideal of $R$} \mid R/I \in \mathcal{C}\}$ of all possible annihilators of elements of modules in $\mathcal{C}$. We put $\mathcal{K} = \mathcal{G} \cap \Spec(R)$. Note that since $\mathcal{C}$ is closed under submodules and direct limits, we can rewrite $\mathcal{K} = \{\pp \in \Spec(R) \mid \kappa(\pp) \in \mathcal{C}\}$, where $\kappa(\pp) = R_{\pp}/\pp$ is the residue field of $R$ at $\pp$. Then we define two functions $\varphi$ and $\psi$ by putting for any $\pp \in \Spec(R)$:
$$\varphi(\pp) = \inf \{\qq \in \mathcal{K} \mid R_{\qq}/\pp \in \mathcal{C}\},$$
$$\psi(\pp) = \sup \{\qq \in \mathcal{K} \mid R_{\varphi(\pp)}/\qq \in \mathcal{C}\}.$$
Since $\kappa(\pp) \in \mathcal{C}$ for any $\pp \in \mathcal{K}$, and $R_{\varphi(\pp)}/\pp \in \mathcal{C}$ by the closure of $\mathcal{C}$ under direct limits, it is easily seen that $\varphi(\pp) \subseteq \pp \subseteq \psi(\pp)$. 

	Finally, we assign to the cosilting class $\mathcal{C}$ a system of intervals defined as follows: $\mathcal{X}_\mathcal{C} = \{[\varphi(\pp), \psi(\pp)] \mid \pp \in \mathcal{K}\}$.
\begin{remark}
	For clarity we rephrase the definition of the admissible system $\mathcal{X}_{\mathcal{C}}$ associated to a cosilting class $\mathcal{C}$ in perhaps a less opaque way (but relying on the results of Section~\ref{S:coaisletoint}). We think of the subset
	$$\mathcal{K} = \{\pp \in \Spec(R) \mid \kappa(\pp) \in \mathcal{C} \} = \{\pp \in \Spec(R) \mid R/\pp \in \mathcal{C} \}$$
	of $\Spec(R)$ as the \textit{support} of the admissible system. Since $\mathcal{C}$ is closed under products and direct limits, it is clear that $\mathcal{K}$ is closed under intersections and unions of non-empty subsets. However, the support $\mathcal{K}$ does not contain sufficient information about $\mathcal{C}$ because it is unable to recover which cyclic modules belong to $\mathcal{C}$ in general. Nevertheless, it will turn out that it is enough to consider which uniserial modules of the form $R_{\pp}/\qq$, for prime ideals $\pp \subseteq \qq$, belong to $\mathcal{C}$. In this light, $\mathcal{X}_{\mathcal{C}}$ can equivalently be defined as follows. Given prime ideals $\qq,\pp \in \mathcal{K}$, we define an equivalence relation $\sim$ on $\mathcal{K}$ by setting $\pp \sim \qq$ (and $\qq \sim \pp$) if and only if $\pp \subseteq \qq$ and $R_{\pp}/\qq \in \mathcal{C}$. Then it follows from Lemma~\ref{L:mc} (see also Proposition~\ref{P:00}) that the intervals of $\mathcal{X}_{\mathcal{C}}$, viewed as closed intervals in the totally ordered set $(\mathcal{K},\subseteq)$, are precisely the equivalence classes of $\mathcal{K}$ with respect to $\sim$.
\end{remark}

\begin{prop}\label{P:00}
	The system of intervals $\mathcal{X}_\mathcal{C}$ is an admissible system.
\end{prop}
\begin{proof}
		This is proved for 1-cotilting classes (that is, cosilting classes containing the projective modules, see \S\ref{SS:cotilting}) in \cite[Definition 3.7 and Proposition 3.8]{B2} and the references to \cite{B} therein. Note that the proof only uses that $\mathcal{C}$ is a definable torsion-free class, and therefore applies to cosilting classes as well. 
		
		Alternatively, we prove this more generally in Section~\ref{S:coaisletoint}. Indeed, this is a special case of Corollary~\ref{C:admissible}, applied to the Happel-Reiten-Smal\o ~t-structure induced by the torsion pair $(\mathcal{T},\mathcal{C})$, and setting $\varphi = \varphi_0$, $\psi = \psi_0$. Note that there is no circularity in our argumentation, as the only part where we need results from this section in Section~\ref{S:coaisletoint} is the proof of Proposition~\ref{P:coaislestointervals} (the degreewise non-density condition).
\end{proof}
\begin{lem}\label{L04}
	Let $\mathcal{C}$ be a cosilting class, and $I$ a proper ideal. Then $R/I \in \mathcal{C}$ if and only if there is an interval $\chi \in \mathcal{X}_\mathcal{C}$ such that $I \in \langle \chi \rangle$.
\end{lem}
\begin{proof}
	This is proved in precisely the same way as \cite[Lemma 3.6]{B2}.
\end{proof}
\begin{lem}\label{L:cyclic}\emph{(\cite[Lemma 3.1]{B})}
		Let $R$ be a valuation domain. Then any subcategory $\mathcal{C}$ of $\ModR$ closed under submodules, pure epimorphisms, direct limits, and extensions is the smallest subcategory containing the cyclic modules in $\mathcal{C}$ closed under the listed operations. In particular, any definable torsion-free class in $\ModR$ is uniquely determined by the cyclic modules it contains.
\end{lem}
\begin{proof}
	Let $\mathcal{C}$ be a cosilting class. Since $\mathcal{C}$ is closed under submodules and directed unions, it is uniquely determined by the finitely generated modules it contains. By \cite[\S I, Lemma 7.8]{FS}, any finitely generated $R$-module admits a finite pure filtration by cyclic modules. Since definable subcategories of $\ModR$ are closed under pure epimorphisms, this shows that $\mathcal{C}$ is uniquely determined by the cyclic modules it contains.
\end{proof}
\subsection{The correspondence}
Now we are ready to state the classification of cosilting classes in the module category of a valuation domain.
\begin{thm}\label{T01}
	Let $R$ be a valuation domain. Then there is a 1-1 correspondence
	$$\left \{ \begin{tabular}{ccc} \text{ admissible systems $\mathcal{X}$} \\ \text{in $\Spec(R)$} \end{tabular}\right \}  \leftrightarrow  \left \{ \begin{tabular}{ccc} \text{ cosilting classes $\mathcal{C}$ } \\ \text{ in $\ModR$} \end{tabular}\right \}$$
		given by the mutually inverse assignments 
		$$\mathcal{X} \mapsto \mathcal{C}_\mathcal{X}, \text{ and}$$
		$$\mathcal{C} \mapsto \mathcal{X}_\mathcal{C}.$$
		In this correspondence, the 1-cotilting classes $\mathcal{C}$ correspond to those admissible systems $\mathcal{X}$ which contain an interval of the form $[0,\qq]$ for some $\qq \in \Spec(R)$.
\end{thm}
\begin{proof}
		The two assignments are well defined by Lemma~\ref{L01} and Proposition~\ref{P:00}. Two cosilting classes coincide if and only if they contain the same cyclic modules, this is Lemma~\ref{L:cyclic}. Together with Lemma~\ref{L04} and Lemma~\ref{L01}, this shows that $\mathcal{C} = \mathcal{C}_{\mathcal{X}_\mathcal{C}}$ for any cosilting class $\mathcal{C}$. On the other hand, we have $\mathcal{X}_{\mathcal{C}_\mathcal{X}} = \mathcal{X}$ for any admissible system $\mathcal{X}$ by Lemma~\ref{L01} and Lemma~\ref{L04}.

		A cosilting class $\mathcal{C}$ is 1-cotilting if and only if it contains $R$ (see \S\ref{SS:cotilting}), which by Lemma~\ref{L04} occurs if and only if there is an interval in $\mathcal{X}_\mathcal{C}$ which contains the zero prime ideal, which means that it is of the form $[0,\qq]$ for some prime $\qq$.
\end{proof}
\section{Density and homological formulas}
In this section we provide an alternative description of cosilting classes in $\ModR$ for a valuation domain $R$ using the Tor functor with certain uniserial modules. This will be useful in the description of definable coaisles in $\Der(R)$.
\subsection{Maximal immediate extensions of valuation domains}
	Here, we follow \cite[\S II]{FS}. A valuation domain $R$ is \EMP{maximal} if it is linearly compact in the discrete topology. A ring map $R \rightarrow S$ between two valuation domains is an \EMP{immediate extension} if the following two conditions are satisfied:
	\begin{enumerate}
		\item[(i)] the assignments $I \mapsto SI$ and $J \mapsto J \cap R$ are mutually inverse bijections between the sets of ideals of $R$ and $S$, respectively, and it restricts to a bijection between $\Spec(R)$ and $\Spec(S)$ (\cite[p. 59]{FS}), and
		\item[(ii)] if $\mm$ is the maximal ideal of $R$, then canonical map $R/\mm \rightarrow S/\mm S$ is an isomorphism of fields.
	\end{enumerate}
	We recall (\cite[\S II, Theorem 1.9]{FS}) that for any valuation domain $R$, there is a \EMP{maximal immediate extension} $R \rightarrow S$, i.e. an immediate extension such that the only immediate extension of $S$ is the trivial one. This is always a faithfully flat ring extension (see \cite[\S II, Exercise 1.5]{FS}, together with condition $(i)$) with the following properties:
	\begin{fact}\label{F:mie}
	\begin{enumerate}
		\item[(i)] $S$ is a maximal valuation domain (\cite[\S II, Theorem 6.7]{FS}), and an immediate extension $R \rightarrow S$ is a maximal immediate extension if and only $S$ is a maximal valuation domain,
		\item[(ii)] for any uniserial module $M$, the module $M \otimes_R S$ is a pure-injective $R$-module (\cite[p. 445]{FS});
		\item[(iii)] $R_{\qq} \otimes_R S \simeq S_{\qq S}$ for any $\qq \in \Spec(R)$, this follows from (\cite[\S II, Lemma 1.6]{FS});
		\item[(iv)] in particular, the quotient field $Q(S)$ of $S$ is equal to $QS$;
		\item[(v)] for any proper ideal $I$, the module $Q/I \otimes_R S \simeq Q(S)/IS$ is injective in $\ModS$ (\cite[\S IX, Theorem 4.4]{FS}).
	\end{enumerate}
	\end{fact}
	The maximal immediate extension is not uniquely determined as a ring homomorphism, but it is always isomorphic to the pure-injective envelope of $R$ as an $R$-module (\cite[\S XIII, Proposition 5.1]{FS}). Next we remark some properties of maximal immediate extensions with respect to localization.

	\begin{lem}\label{L:mie}
		Let $R$ be a valuation domain and $R \rightarrow S$ a maximal immediate extension. Let $\pp \subseteq \qq$ be primes of $R$, and denote by $U = R_{\qq}/\pp$. Then:
			\begin{enumerate}
				\item[(i)] $U$ is a valuation domain and the natural map $R \rightarrow U$ is a ring epimorphism,
				\item[(ii)] $U \otimes_R S \simeq S_{\qq S}/\pp S$ is a maximal valuation domain,
				\item[(iii)] $U \rightarrow U \otimes_R S$ is a faithfully flat ring homomorphism.
				\item[(iv)] $Q(U) \otimes_R S = Q(U \otimes_R S)$ as ring extensions of $R$.
			\end{enumerate}
	\end{lem}
	\begin{proof}
			\begin{enumerate}
				\item[(i)] Obvious.
				\item[(ii)] See e.g. \cite[Proposition 5]{C}.
				\item[(iii)] Since $S$ is a flat $R$-module, $U \otimes_R S$ is a flat $U$-module. Since $R \rightarrow S$ is a faithfully flat ring homomorphism, and $R \rightarrow U$ is a ring epimorphism, clearly $(U \otimes_R S) \otimes_U M = 0$ implies $M = 0$ for any $U$-module $M$.
				\item[(iv)] The quotient field of $U = R_{\qq}/\pp$ is $R_{\pp}/\pp$, while $Q(U \otimes_R S) = Q(S_{\qq S}/\pp S) = S_{\pp S}/\pp S$. Therefore, $Q(U) \otimes_R S = Q(U \otimes_R S)$.
			\end{enumerate}
	\end{proof}
	Finally, we remark an important property of maximal valuation domains.
	\begin{lem}\label{L:maxcogen}
		Let $R$ be a maximal valuation domain with maximal ideal $\mm$ and quotient field $Q$. The module $Q/\mm$ is an injective cogenerator in $\ModR$
	\end{lem}
	\begin{proof}
			By Fact~\ref{F:mie}(v), $Q/\mm$ is an injective $R$-module. Since $Q/\mm$ contains the unique simple $R$-module, it is an injective cogenerator in $\ModR$.
	\end{proof}
	\subsection{Uniserial modules over valuation domains}
	Recall that an $R$-module is called \EMP{uniserial} if its lattice of submodules is totally ordered. Over a valuation domain $R$ with quotient field $Q=Q(R)$, any module of the form $J/I$ is uniserial, where $I \subseteq J \subseteq Q$ are $R$-submodules of the quotient field. Uniserial modules of this form are called \EMP{standard}. In general there can be uniserial modules not isomorphic to a standard uniserial module, see \cite[\S X.4]{FS}. However, over a maximal valuation domain, every uniserial module is standard (\cite[\S X Proposition 3.1]{FS}). A very important fact for us is that definable subcategories of the module category of a valuation domain are completely determined by the standard uniserial modules they contain. This follows from a result due to Ziegler \cite{Z2}, reproved by algebraic methods by Monari-Martinez \cite{MM}, which shows that the indecomposable pure-injective modules over a valuation domain $R$ are up to isomorphism precisely the pure-injective envelopes of the standard uniserial modules over $R$. Note also, that given a standard uniserial $R$-module $J/I$, its pure-injective hull can be expressed explicitly --- it is additively equivalent to $JS/IS$, where $R \rightarrow S$ is any maximal immediate extension of $R$, see \cite[\S XIII, Corollary 5.5]{FS}.

	\begin{lem}\label{L:standarduni}
			Let $R$ be a valuation domain and $\mathcal{C}$ a definable subcategory of $\ModR$. Then $\mathcal{C}$ is determined uniquely as a definable subcategory of $\ModR$ by the standard uniserial modules it contains.
	\end{lem}
	\begin{proof}
			By \cite[\S XIII Theorem 5.9]{FS}, an $R$-module $M$ is indecomposable pure-injective if and only if it is a pure-injective hull of a standard uniserial module. By \cite[Corollary 5.1.4]{P}, any definable subcategory is uniquely determined by the indecomposable pure-injectives it contains. Finally, an $R$-module $M$ belongs to a definable subcategory of $\ModR$ if and only if its pure-injective hull does (\cite[Theorem 3.4.8]{P}), which concludes the proof.
	\end{proof}

	Let $[\pp,\qq]$ be an interval in $\Spec(R)$. We will be especially interested in two kinds of standard uniserial modules --- $R_{\qq}/\pp$ and $R_{\pp}/\qq$. While $R_{\qq}/\pp$ is an epimorphic ring extension of $R$, the role of $R_{\pp}/\qq$ is clarified by the following observation:

\begin{lem}\label{L68}
		Let $R$ be a valuation domain and $S$ its maximal immediate extension. Let $[\pp,\qq]$ be an interval in $\Spec(R)$. Then the module $(R_{\pp}/\qq) \otimes_R S$ is an injective cogenerator in the category $\ModRqqppS$, and therefore it is a cogenerator in $\ModRqqpp$.
	\end{lem}
	\begin{proof}
			Denote $U = R_{\qq}/\pp$, let $Q(U) = R_{\pp}/\pp$ be the quotient field of $U$ and let $\mm(U) = \qq/\pp$ be the maximal ideal of $U$. By Lemma~\ref{L:mie} we know that $Q(U) \otimes_R S$ is the field of quotients of the valuation domain $U \otimes_R S$ and clearly $\mm \otimes_R S$ is its maximal ideal. Also by Lemma~\ref{L:mie}, $U \otimes_R S$ is a maximal valuation domain. Therefore, Lemma~\ref{L:maxcogen} implies that $(R_{\pp}/\qq) \otimes_R S \simeq (Q(U)/\mm(U)) \otimes_R S$ is an injective cogenerator in $\ModU$.	

			Finally, since $U \rightarrow U \otimes_R S$ is a faithfully flat extension by Lemma~\ref{L:mie}, any $U$-module embeds into an $(U \otimes_R S)$-module. Therefore, $(R_{\pp}/\qq) \otimes_R S$ is a cogenerator in $\ModU$.
	\end{proof}

	Now we will be interested in computing the Ext-orthogonal to the modules of the form $(R_{\pp}/\qq) \otimes_R S$. The following lemma is proved in \cite[Lemma 6.6]{B} for the case in which $R$ is already maximal.

	\begin{lem}\label{L:extA}
			Let $\pp = \pp^2 \subseteq \qq$ be a couple of prime ideals, and let $I$ be an ideal. Then $\Ext_R^1(R/I,(R_{\pp}/\qq) \otimes_R S) = 0$ if and only if either
			\begin{enumerate}
					\item[(i)] $\pp \subseteq I$, or
					\item[(ii)] $I^\# \subseteq \pp$ and $I \not\simeq R_{\pp}$.
			\end{enumerate}
	\end{lem}
	\begin{proof}
			If $R$ is a maximal valuation domain, this is precisely \cite[Lemma 6.6]{B}. For general valuation domain, we have 
			$$\Ext{}_R^1(R/I,(R_{\pp}/\qq) \otimes_R S) \simeq \Ext{}_S^1(S/IS,(R_{\pp}/\qq) \otimes_R S)$$ 
			by flatness of $S$ over $R$. Using the maximal case, this means that the vanishing of $\Ext_R^1(R/I,(R_{\pp}/\qq) \otimes_R S)$ occurs if and only if one of the following conditions hold over $S$:
			\begin{enumerate}
					\item[(i')] $S\pp \subseteq SI$, or
					\item[(ii')] $(SI)^\# \subseteq S\pp$ and $SI \not\simeq S_{S\pp}$.
			\end{enumerate}
			The condition $(i')$ is clearly equivalent to $(i)$. It is easy to see that $(SI)^\# = S(I^\#)$, and thus $(SI)^\# \subseteq S\pp$ is equivalent to $I^\# \subseteq \pp$. If $I \simeq R_{\pp}$, then clearly $SI \simeq S_{S\pp} \simeq R_{\pp} \otimes_R S$. Conversely, if $SI \simeq S_{S\pp}$, then there is $t \in S$ such that $SI = tS_{S\pp}$. By \cite[\S II Lemma 1.6]{FS}, there is an element $r \in R$ and a unit $e \in S$ such that $t=re$. Therefore, $SI = rS_{S\pp}$, and thus $I = SI \cap R_{\pp} = tR_{\pp}$. This proves that $(ii)$ is equivalent to $(ii')$.
	\end{proof}

\subsection{Density and gaps of admissible systems}\label{SS:densityandgaps}
	Let $(X,<)$ be a totally ordered set. A non-degenerate interval $x<y$ in $X$ is called \EMP{dense} if for any $x \leq s<t \leq y$ there is an element $z \in X$ with $s<z<t$. If $X$ admits a minimal element $0$ and a maximal element $1$, we say that $X$ is \EMP{dense} if the interval $0<1$ is dense. We say that $X$ is \EMP{nowhere dense} if it contains no dense intervals. We say that a subset $Y \subseteq X$ is \EMP{dense in $X$} if for any interval $x<y$ in $X$ there is $z \in Y$ such that $x<z<y$. Say that an element $y \in X$ \EMP{covers} an element $x \in X$ if $x<y$ and there is no element $z \in X$ such that $x<z<y$.

	If $R$ is a valuation domain and $\mathcal{X}$ an admissible system in $\Spec(R)$, we say that $\mathcal{X}$ is \EMP{nowhere dense} if the totally ordered set $(\mathcal{X},<)$ is such. We say that $\mathcal{X}$ is \EMP{dense everywhere} if $(\mathcal{X},<)$ is dense and if $\mathcal{X}$ contains an interval of the form $[0,\qq]$ and an interval of the form $[\pp,\mm]$.

	Let $\mathcal{X}$ be an admissible system of intervals of $\Spec(R)$. Following \cite[Notation 6.7]{B2}, we introduce first an equivalence relation $\sim$ on $\mathcal{X}$ by setting $\chi \sim \xi$ if either $\chi = \xi$ or whenever the interval $\chi<\xi$ (or $\xi < \chi$) in $(\mathcal{X},<)$ between the two intervals is dense.  Using the completeness we see that each equivalence class $C \in \mathcal{X}/\sim$ of $\mathcal{X}$ under $\sim$ has a minimal element $[\pp,\qq]$ and a maximal element $[\pp',\qq']$. This defines an interval $\tau_C = [\pp,\qq']$ associated to $C$ for each $C \in \mathcal{X}/\sim$. We let $\bar{\mathcal{X}}$ denote the set of intervals $\{\tau_C \mid C \in \mathcal{X}/\sim\}$. It is not hard to check that $\bar{\mathcal{X}}$ is a nowhere dense admissible system on $\Spec(R)$. 

	Also, we let $\mathcal{H}(\mathcal{X})$ be the collection of all equivalence classes from $\mathcal{X}/\sim$ with more than one element. Note, that this set corresponds naturally to the set of all maximal dense intervals in $(\mathcal{X},<)$. We also consider each equivalence class $C$ as a totally ordered subset of $(\mathcal{X},<)$.

	Let $\Spec^*(R) = \Spec(R) \cup \{-\infty,R\}$ be an extension of the spectrum of a valuation domain $R$, where $-\infty$ will be understood as a formal symbol satisfying $-\infty \subsetneq I$ for any ideal $I$ of $R$. Let $\qq \subsetneq \pp$ be two elements of $\Spec^*(R)$. 

	We say that $(\qq,\pp)$ is a \EMP{gap} of the admissible system $\mathcal{X}$ if one of the following conditions is satisfied:
	\begin{enumerate}
		\item[(i)] there are intervals $[\pp',\qq] < [\pp,\qq'] \in \mathcal{X}$ such that $[\pp,\qq']$ covers $[\pp',\qq]$,
		\item[(ii)] $\qq = -\infty$, and the minimal interval of $\mathcal{X}$ is of the form $[\pp,\qq']$, where $\pp \neq 0$,
		\item[(iii)] $\pp = R$, and the maximal interval of $\mathcal{X}$ is of the form $[\pp',\qq]$, where $\qq \neq \mm$,
		\item[(iv)] $\qq = -\infty$ and $\pp = R$ if $\mathcal{X} = \emptyset$.
	\end{enumerate}
	
	We denote the collection of all gaps of $\mathcal{X}$ by $\mathcal{G}(\mathcal{X})$. Observe that $\mathcal{G}(\mathcal{X}) = \mathcal{G}(\bar{\mathcal{X}})$. The relation between density, gaps, and ideals is the content of the following auxiliary result. Given an ideal $I$ and a gap $(\qq,\pp)$, we will denote by $I \in (\qq,\pp)$ the situation $\qq \subsetneq I \subsetneq \pp$.
	\begin{lem}\label{L:denseideal}
		Let $R$ be a valuation domain and $\mathcal{X}$ an admissible system in $\Spec(R)$. The for any ideal $I$ of $R$, one of the following possibilities occurs:  
			\begin{enumerate}
					\item[(i)] there is an interval $[\pp,\qq] \in \mathcal{X}$ such that $\pp \subseteq I \subseteq \qq$, or
					\item[(ii)] there is a gap $(\qq,\pp) \in \mathcal{G}(\mathcal{X})$ such that $I \in (\qq,\pp)$. 
			\end{enumerate}
			Furthermore, if $\mathcal{X}$ is dense everywhere, then $\mathcal{G}(\mathcal{X}) = \emptyset$, and therefore only $(i)$ can occur.
	\end{lem}
	\begin{proof}
			If $\mathcal{X}$ is empty then $(-\infty,R)$ is a gap, and $(ii)$ is clearly true. Then we can assume $\mathcal{X}$ non-empty. Let us assume that $(i)$ is not true. Then $\mathcal{X}$ can be written as a disjoint union $\mathcal{X} = \mathcal{A} \cup \mathcal{B}$, where $\mathcal{A} = \{\chi \in \mathcal{X} \mid I \subsetneq \pp_\chi \}$, and $\mathcal{B} = \{\chi \in \mathcal{X} \mid \qq_\chi \subsetneq I\}$. If $\mathcal{B}$ is empty, then $\mathcal{A}$ is non-empty, and by the completeness $\mathcal{A}$ has a minimal element $[\pp_A,\qq_A]$. Necessarily $I \subsetneq \pp_A$, and thus $(-\infty,\pp_A) \in \mathcal{G}(\mathcal{X})$. The case when $\mathcal{A}$ is empty is handled similarly.
			
			Suppose that both $\mathcal{A}$ and $\mathcal{B}$ are non-empty. By the completeness, there is an interval of the form $[\pp_B,\qq_B]$, where $\pp_B = \bigcup_{\chi \in \mathcal{B}} \pp_\chi$. Since $\pp_\chi \subsetneq I$ for all $\chi \in \mathcal{B}$, we have $\pp_B \subseteq I$, and thus $[\pp_B,\qq_B]$ belongs to $\mathcal{B}$, and it is the maximal element of $(\mathcal{B},\leq)$. Similarly, $\mathcal{A}$ has a minimal element $[\pp_A,\qq_A]$. But then $[\pp_A,\qq_A]$ covers $[\pp_B,\qq_B]$, and therefore there is a gap $(\qq_B,\pp_A) \in \mathcal{G}(\mathcal{X})$ such that $\qq_B \subsetneq I \subsetneq \pp_A$.

			The furthermore part is clear from the definition of a dense everywhere admissible system.
	\end{proof}
	
	\subsection{Cotilting modules corresponding to dense everywhere admissible systems}\label{SS:dense}
For admissible systems which are dense everywhere, the associated 1-cotilting modules have a rather special form, which will turn important in \S \ref{S:coaisletoint}. The following proof is a generalization of \cite[Proposition 5.4]{B2}. 
	\begin{prop}\label{P:cotiltingdense}
		Let $R$ be a valuation domain and $R \subseteq S$ a maximal immediate extension. Suppose that $\mathcal{X}$ is a dense everywhere admissible system in $\Spec(R)$. Then the module 
			$$C = \prod_{[\pp,\qq] \in \mathcal{X}} ((R_{\pp}/\qq) \otimes_R S)$$
		is a 1-cotilting module associated to the 1-cotilting class $\mathcal{C}_\mathcal{X}$.
	\end{prop}
	\begin{proof}
			We show that $C$ is 1-cotilting by proving that $\Cogen(C) = {}^{\perp} C$. Since all the modules of the form $R_{\pp}/\qq$ are standard uniserial $R$-modules, it follows from Fact~\ref{F:mie} that $(R_{\pp}/\qq) \otimes_R S$ is pure-injective for all $[\pp,\qq] \in \mathcal{X}$, and thus $C$ is pure-injective, and therefore of injective dimension at most one. In particular, ${}^\perp C$ is closed under submodules, pure epimorphic images, and direct limits (see \cite[Corollary 6.21]{GT}).

				\textbf{Claim I:} ${}^\perp C = \mathcal{C}_\mathcal{X}$:
		Recall that
			$$\mathcal{C}_\mathcal{X} = \{ M \in \ModR \mid ~\forall 0 \neq m \in M ~\exists [\pp,\qq] \in \mathcal{X}: ~\Ann{}_R(m) \in \langle \pp,\qq \rangle\},$$
			which is a cosilting class by Theorem~\ref{T01}, and even a 1-cotilting class, as $R \in \mathcal{C}_\mathcal{X}$. To show ${}^\perp C = \mathcal{C}_\mathcal{X}$, it is by Lemma~\ref{L:cyclic} enough to show that both classes contain the same cyclic modules. If $R/I \in {}^\perp C$, then by Lemma~\ref{L:denseideal} there is an interval $[\pp,\qq] \in \mathcal{X}$ such that $\pp \subseteq I \subseteq \qq$. Using Lemma~\ref{L:extA}, $I^\# \subseteq \pp'$ for any interval $[\pp',\qq'] \in \mathcal{X}$ with $\qq \subseteq \pp'$, and thus by density of $\mathcal{X}$, necessarily $I^\# \subseteq \qq$, proving that $R/I \in \mathcal{C}_\mathcal{X}$.

			Let $R/I \in \mathcal{C}_\mathcal{X}$, and let us show $R/I \in {}^\perp C$. Choose $[\pp,\qq] \in \mathcal{X}$. If $\pp \subseteq I$, we apply Lemma~\ref{L:extA}(i). Assume now that $I \subsetneq \pp$. Since $R/I \in \mathcal{C}_\mathcal{X}$, then necessarily $I^\# \subsetneq \pp$. Because $I \simeq I_{I^\#}$, we infer that $I$ cannot be isomorphic to $R_{\pp}$, and thus $R/I \in {}^\perp C$ by Lemma~\ref{L:extA}(ii).
	
		\textbf{Claim II:} $\Cogen(C) \subseteq {}^\perp C$:

			Since ${}^\perp C = \mathcal{C}_\mathcal{X}$ is a cosilting class, it is enough to show that $C \in \mathcal{C}_\mathcal{X}$. This amounts to checking that $(R_{\pp}/\qq) \otimes_R S \in \mathcal{C}_\mathcal{X}$ for any $[\pp,\qq] \in \mathcal{X}$. As $S$ is a flat $R$-module and $\mathcal{C}_\mathcal{X}$ is closed under direct limits, the task finally reduces to showing that $R_{\pp}/\qq \in \mathcal{C}_\mathcal{X}$ for all $[\pp,\qq] \in \mathcal{X}$. For any non-zero element $x \in R_{\pp}/\qq$, we have $\Ann_R(x) = s^{-1}\qq$ for some $s \in R_{\qq} \setminus \pp$. Therefore $\pp \subseteq \Ann_R(x) \subseteq \qq$, and clearly also $\Ann_R(x)^\# \subseteq \qq$. Therefore, $R_{\pp}/\qq \in \mathcal{C}_\mathcal{X}$.
								
			\textbf{Claim III:} ${}^\perp C \subseteq \Cogen(C)$:

			By Claim II we know that $\Cogen(C)$ is closed under extensions, that is, $\Cogen(C)$ is a torsion-free class. Choose $M \in {}^\perp C = \mathcal{C}_\mathcal{X}$ and let $T$ be its maximal torsion submodule with respect to the torsion pair with torsion-free class $\Cogen(C)$. Towards a contradiction, assume that there is a non-zero element $t \in T $, and let $I = \Ann_R(t)$. Claim I then implies that $I \in \langle \pp,\qq \rangle$ for some $[\pp,\qq] \in \mathcal{X}$.

			Put $T' = \Soc_{\pp}(T) = \{m \in T \mid \pp m = 0\}$. We claim that $\Hom_R(T',C)=0$. Since $\Hom_R(T,C) = 0$, it is enough to show that $T/T' \in {}^\perp C$. Pick $m + T' \in T/T'$ non-zero, and let $J = \Ann_R(m+T')$ and $K = \Ann_R(m)$. Since $T \in \mathcal{C}_\mathcal{X}$, there is $[\pp',\qq'] \in \mathcal{X}$ with $K \in \langle \pp',\qq' \rangle$. As $m \not\in T'=\Soc_{\pp}(T)$, and $\pp=\pp^2$, clearly $K \subsetneq \pp$, and thus $\qq' \subsetneq \pp$. Clearly $K \subseteq J$. If $r \in J \setminus K$, then $rm \in T'$, and we have inclusions $\pp \subseteq \Ann_R(rm) = r^{-1}K \subseteq K^\#  \subseteq \qq'$, which is a contradiction with $\qq' \subsetneq \pp$. Therefore $J = K$, showing that $T/T' \in \mathcal{C}_\mathcal{X} = \Perp{} C$, and thus $\Hom_R(T',C)=0$.

			Consider the localization map $f: T' \rightarrow T'_{\qq}$. The module $T'_{\qq}$ is an $R_{\qq}/\pp$-module, and whence is cogenerated by $(R_{\pp}/\qq) \otimes_R S$ due to Lemma~\ref{L68}, and therefore belongs to $\Cogen(C)$. Then also $T' / \Ker(f) \in \Cogen(C)$, as it is a submodule in $T'_{\qq}$. Together with $\Hom_R(T',C) = 0$, this forces $T' = \Ker(f)$, or in other words, $T'$ is $\qq$-torsion. But since $I \in \langle \pp,\qq \rangle$, $0 \neq t \in T' \setminus \Gamma_{\qq}(T')$ by Lemma~\ref{L23}, a contradiction.
	\end{proof}
	\subsection{Description via homology}\label{SS:homology}
	It will be useful to express the cosilting classes homologically, using the derived tensor functor with respect to certain uniserial modules coming from the intervals and gaps. For this, we introduce the following notation. Let $\mathcal{X}$ be an admissible system and $(\qq,\pp) \in \mathcal{G}(\mathcal{X})$ be a gap. Then we define a complex
			$$K(\qq,\pp) = (\cdots \rightarrow 0 \rightarrow \pp \xrightarrow{i} R_{\qq} \rightarrow 0 \rightarrow \cdots),$$ 
			where $\pp$ is in degree 0, and $i$ is the natural inclusion. In the case where $\qq = -\infty$, the symbol $R_{\qq}$ will be interpreted as zero, and thus $K(-\infty,\pp)$ is just a stalk complex of the prime ideal $\pp$ concentrated in degree $0$. Note that the zero cohomology of $K(\qq,\pp) \otimes_R -$ is then computed as follows
			$$H^0(K(\qq,\pp) \otimes_R -) = \begin{cases} \Tor_1^R(R_{\qq}/\pp,-), & \text{ if} \qq \in \Spec(R) \\ \pp \otimes_R -, & \text{ if} \qq = - \infty.  \end{cases}$$

					Also, it will be convenient to let $\Gamma_{-\infty}$ be the identity functor, while $F_{-\infty}$ and $\Soc_{R}$ will both stand for the zero functor on $\ModR$.
	\begin{lem}\label{L:compute}
			Let $R$ be a valuation domain, and $\mathcal{X}$ an admissible system in $\Spec(R)$. Let $M$ be an  $R$-module $M$ and $I$ an ideal of $R$.
			\begin{enumerate}
					\item[(i)] For any interval $[\pp,\qq] \in \mathcal{X}$ we have:
							
							\begin{itemize} 
											
									\item $\Tor_1^R(R_{\qq}/\pp,R/I) = 0$ if and only if either $I \supseteq \pp$, or if $I \subsetneq \pp$ then $I^\# \subseteq \pp$ and $I \not\simeq R_{\pp}$.
							\end{itemize}
					\item[(ii)] For any gap $(\qq,\pp) \in \mathcal{G}(\mathcal{X})$ we have:
							\begin{itemize}
									\item $H^0(K(\qq,\pp) \otimes_R M)=0$ if and only if $\Gamma_{\qq}(M) \subseteq \Soc_{\pp}(M)$.
							\end{itemize}
			\end{enumerate}
	\end{lem}
	\begin{proof}
	\item[(i)] This is \cite[Theorem 6.11, Claim (i)]{B2}. 		

	\item[(ii)] We start by remarking that $\pp^2 = \pp$ implies that for any $R$-module $N$ we have the equivalence 
$$\pp \otimes_R N = 0 \Leftrightarrow \pp N = 0.$$ 
Indeed, consider the canonical exact sequence $$0 \rightarrow \Tor{}_1^R(R/\pp,N) \rightarrow \pp \otimes_R N \rightarrow \pp N \rightarrow 0,$$
then immediately we see that $\pp \otimes_R N = 0$ implies $\pp N = 0$. Conversely, if $\pp N = 0$ then $\pp(\pp \otimes_R N) = 0$ since $\Tor_1^R(R/\pp,N)$ gets killed by $\pp$. Since $\pp^2 = \pp$, we obtain $\pp \otimes_R N = 0$.

 Note that if $\qq \in \Spec(R)$, then $H^0(K(\qq,\pp) \otimes_R M)=\Tor_1^R(R_{\qq}/\pp,M)=0$ if and only if the natural multiplication map $\pp \otimes_R M \rightarrow M_{\qq}$ is injective. We claim that the kernel of this map is zero if and only if $\pp \Gamma_{\qq}(M) = 0$, or equivalently, $\Gamma_{\qq}(M) \subseteq \Soc_{\pp}(M)$. Indeed, since $\qq \subsetneq \pp$, we have $\pp R_{\qq} = R_{\qq}$, and so $\pp \otimes_R M_{\qq} \simeq \pp R_{\qq} \otimes_R M_{\qq} \simeq M_{\qq}$. It follows that the multiplication map $\pp \otimes_R M \rightarrow M_{\qq}$ is identified with the map $(\pp \otimes_R f_M): \pp \otimes_R M \rightarrow \pp \otimes_R M_{\qq}$, where $f_M: M \rightarrow M_{\qq}$ is the canonical map. By flatness of $\pp$, one has that $\Ker(\pp \otimes_R f_M) = \pp \otimes_R \Gamma_{\qq}(M)$. Finally, it follows from the first paragraph that $\pp \otimes_R \Gamma_{\qq}(M) = 0$ if and only if $\pp \Gamma_{\qq}(M) = 0$.

	It remains to address the case of $\qq = -\infty$. Then $H^0(K(\qq,\pp) \otimes_R M)=\pp \otimes_R M = 0$ is equivalent to $\pp M = 0$ by the first paragraph again, and the latter can be rewritten as $M = \Gamma_{-\infty}(M) \subseteq \Soc_{\pp}(M)$.
	\end{proof}
	We are ready to show that any cosilting class in $\ModR$ is given by derived tensor product. Notice that $\mathcal{G}(\mathcal{X})$ does not contain a gap of the form $(-\infty,\pp)$ if and only if the cosilting class does not contain $R$, which is further equivalent to it not being a 1-cotilting class. In this case, we express the class as a Tor-orthogonal class, recovering \cite[Theorem 6.11]{B2}. For the definition of the set $\mathcal{H}(\mathcal{X})$ we refer the reader to \S~\ref{SS:densityandgaps}.
	\begin{prop}\label{P:tor}
			Let $\mathcal{C}$ be a cosilting class corresponding to an admissible system $\mathcal{X}$ via Theorem~\ref{T01}. For each $C \in \mathcal{H}(\mathcal{X})$, let $\mathcal{Y}_C$ be a dense subset of $C$. Then 
			$$\mathcal{C} = \bigcap_{(\qq,\pp) \in \mathcal{G}(\bar{\mathcal{X}})}\Ker H^0(K(\qq,\pp) \otimes_R -) \cap \bigcap_{[\pp,\qq] \in \mathcal{Y}_C, C \in \mathcal{H}(\mathcal{X})}\Ker \Tor{}_1^R(R_{\qq}/\pp,-).$$
	\end{prop}
	\begin{proof}
			Recall that 
			$$\mathcal{C} = \mathcal{C}_\mathcal{X} =\{M \in \ModR \mid \forall 0 \neq m \in M ~\exists \chi \in \mathcal{X}: \Ann{}_R(m) \in \langle \chi \rangle\}.$$
			Denote 
			$$\mathcal{C}' = \bigcap_{(\qq,\pp) \in \mathcal{G}(\bar{\mathcal{X}})}\Ker H^0(K(\qq,\pp) \otimes_R -) \cap \bigcap_{[\pp,\qq] \in \mathcal{Y}_C, C \in \mathcal{H}(\mathcal{X})}\Ker \Tor{}_1^R(R_{\qq}/\pp,-),$$ 
			and let us prove that $\mathcal{C} = \mathcal{C}'$. In view of Lemma~\ref{L:cyclic}, it is enough to show that $\mathcal{C}$ and $\mathcal{C}'$ contain the same cyclic modules. Let $R/I \in \mathcal{C}$. Then there is an interval $[\pp,\qq] \in \mathcal{C}$ such that $\pp \subseteq I$ and $I^\# \subseteq \qq$. By Lemma~\ref{L:compute}(ii), $H^0(K(\qq',\pp') \otimes_R R/I)=0$ for any gap $(\qq',\pp') \in \mathcal{G}(\mathcal{X})$. Indeed, either $\pp' \subseteq \pp \subseteq I$, and thus $R/I = \Soc_{\pp'}(R/I)$, or $I^\# \subseteq \qq \subseteq \qq'$, and thus $\Gamma_{\qq'}(M) = 0$. Let $C \in \mathcal{H}(\mathcal{X})$, and let $[\pp',\qq'] \in \mathcal{Y}_C$. If $I \subsetneq \pp'$, then also $I^\# \subsetneq \pp'$, and thus $\Tor_1^R(R_{\qq'}/\pp',R/I) = 0$ by Lemma~\ref{L:compute}(i).

			For the converse, let $R/I \in \mathcal{C}'$. Since $R/I \in \bigcap_{(\qq,\pp) \in \mathcal{G}(\bar{\mathcal{X}})}\Ker H^0(K(\qq,\pp) \otimes_R -)$, we see by Lemma~\ref{L:compute}(ii) and Lemma~\ref{L:denseideal} that there is an interval $[\pp,\qq] \in \bar{\mathcal{X}}$ such that $\pp \subseteq I \subseteq I^\# \subseteq \qq$. Then either $[\pp,\qq] \in \mathcal{X}$, and we are done, or there is $C \in \mathcal{H}(\mathcal{X})$ such that $[\pp,\qq] = \tau_C$. Again by Lemma~\ref{L:denseideal}, there is an interval $[\pp',\qq'] \in C$ such that $\pp' \subseteq I \subseteq \qq'$. Because $\mathcal{Y}_C$ is dense in $C$, together with the completeness of $C$, there is a sequence of intervals $[\pp_\alpha,\qq_\alpha] \in \mathcal{Y}_C, \alpha<\lambda$, such that $\bigcap_{\alpha< \lambda}\pp_\alpha = \qq'$. Since $R/I \in \mathcal{C}'$, we have $\Tor_1^R(R_{\qq_{\alpha}}/\pp_{\alpha},R/I)=0$ for all $\alpha<\lambda$. By Lemma~\ref{L:compute}(i), we have $I^\# \subseteq \pp_\alpha$ for all $\alpha<\lambda$, and therefore $I^\# \subseteq \qq'$. We showed that $\pp' \subseteq I \subseteq I^\# \subseteq \qq'$, and since $[\pp',\qq'] \in \mathcal{X}$, we conclude that $R/I \in \mathcal{C}$.
	\end{proof}

\section{From definable coaisles to admissible filtrations}\label{S:coaisletoint}
	The goal of this section is to associate to a coaisle of a homotopically smashing t-structure in the derived category of a valuation domain $R$ a sequence of admissible systems on $\Spec(R)$ indexed by the cohomological degrees, in a way which leads to a bijective correspondence when restricted to definable coaisles. 

	\begin{definition}\label{D:phipsi}
			Let $\mathcal{V}$ be a coaisle of a  homotopically smashing t-structure (so, in particular, $\mathcal{V}$ can be a definable coaisle) in the derived category $\Der(R)$ of a valuation domain $R$. Denote $\mathcal{V}_n = \{M \in \ModR \mid M[-n] \in \mathcal{V}\}$, and let $\mathcal{K}_n = \{\pp \in \Spec(R) \mid \kappa(\pp) \in \mathcal{V}_n\}$ for each $n \in \mathbb{Z}$. Inspired by \cite{B},\cite{B2}, we define the two following assignments on prime ideals in the same way as in Section~\ref{SS:phipsi}:
			$$\varphi_n(\pp) = \inf\{\qq \in \Spec(R) \mid R_{\qq}/\pp \in \mathcal{V}_n\},$$
			$$\psi_n(\pp) = \sup \{ \qq \in \Spec(R) \mid R_{\varphi_n(\pp)}/\qq \in \mathcal{V}_n\}.$$
			Finally, we define for each $n \in \mathbb{Z}$ a set $\mathcal{X}_n=\{[\varphi_n(\pp),\psi_n(\pp)] \mid \pp \in \mathcal{K}_n\}$ of formal intervals in $\Spec(R)$. 	
	\end{definition}
	\textit{In the rest of this section, we will be in the situation of Definition~\ref{D:phipsi} over a valuation domain $R$ and we will show in several steps that $\mathbb{X}=(\mathcal{X}_n \mid n \in \mathbb{Z})$ forms a nested sequence of admissible systems on $\Spec(R)$.}
	\begin{lem}\label{L:belong}
			For any $\pp \in \mathcal{K}_n$ we have $R_{\varphi_n(\pp)}/\pp \in \mathcal{V}_n$ and $R_{\varphi_n(\pp)}/\psi_n(\pp) \in \mathcal{V}_n$.
	\end{lem}
	\begin{proof}
		The first claim is proved by noting that $R_{\varphi_n(\pp)}/\pp = \varinjlim_{\qq, R_{\qq}/\pp \in \mathcal{V}_n} R_{\qq}/\pp$, and by the fact that $\mathcal{V}_n$ is closed under direct limits. The second follows similarly from $R_{\varphi_n(\pp)}/\psi_n(\pp) = \varinjlim_{\qq, R_{\varphi_n(\pp)}/\qq \in \mathcal{V}_n} R_{\varphi_n(\pp)}/\qq$. 
	\end{proof}
			\begin{lem}\label{L:mono1} For any $\pp \in \mathcal{K}_n$ we have $\varphi_n(\pp) \subseteq \pp$ and $\psi_n(\pp) \supseteq \pp$.
			\end{lem}
			\begin{proof}
					It is enough to show that $\kappa(\pp)$ and $R_{\varphi_n(\pp)}/\pp$ are in $\mathcal{V}_n$ whenever $\pp \in \mathcal{K}_n$. The first claim follows directly from the definition of $\mathcal{K}_n$, while the second from Lemma~\ref{L:belong}.
			\end{proof}
			The following lemma follows from an application of a \emph{d\'{e}vissage} technique and is valid for an arbitrary coaisle in the derived category of any commutative ring.
			\begin{lem}\label{L:RHom}
					Suppose that $\pp$ is a prime, $X \in \mathcal{V}$, and that $\Hom_{\Der(R)}(\kappa(\pp)[-n],X) \neq 0$. Then $\kappa(\pp) \in \mathcal{V}_n$.
			\end{lem}
			\begin{proof}
					Recall that 
					$$\Hom{}_{\Der(R)}(\kappa(\pp)[-n],X) \simeq \Hom{}_{\Der(R)}(\kappa(\pp),X[n]) \simeq H^n\RHom{}_R(\kappa(\pp),X).$$ 
					By \cite[Proposition 2.3]{HCG}(ii), the complex $\RHom_R(\kappa(\pp),X)$ belongs to $\mathcal{V}$. But $\RHom_R(\kappa(\pp),X)$ also lives in the essential image of the forgetful functor $\Der(\kappa(\pp)) \rightarrow \Der(R)$, and thus is isomorphic in $\Der(R)$ to a complex of vector spaces over the field $\kappa(\pp)$. In particular, $\RHom_R(\kappa(\pp),X)$ is isomorphic in $\Der(R)$ to a split complex. Therefore, $H^n\RHom_R(\kappa(\pp),X)[-n] \in \mathcal{V}$. Since $H^n\RHom_R(\kappa(\pp),X)$ is a non-zero vector space over $\kappa(\pp)$, it follows that $\kappa(\pp)[-n] \in \mathcal{V}$, or in other words $\kappa(\pp) \in \mathcal{V}_n$.
			\end{proof}
	\begin{lem}\label{L67}
		Let $\pp \subseteq \pp' \subsetneq \pp''$. Then $R_{\pp}/R_{\pp'}$ is isomorphic to a direct limit of copies of $R_{\pp}/\pp''$.
	\end{lem}
	\begin{proof}
		Since $\pp' \subsetneq \pp''$, we have $R_{\pp'} \simeq \pp'' \otimes_R R_{\pp'} \simeq \varinjlim_{r \not\in \pp'} r^{-1}\pp''$. Then $R_{\pp}/R_{\pp'} \simeq \varinjlim_{r \not\in \pp'} R_{\pp}/r^{-1}\pp''$. But since $\pp \subseteq \pp'$, we have $R_{\pp}/r^{-1}\pp'' \simeq R_{\pp}/\pp''$ for any $r \in R \setminus \pp'$.
	\end{proof}
Since the coaisle $\mathcal{V}$ is closed under directed homotopy colimits, it follows that the subcategory $\mathcal{V}_n = \{M \in \ModR \mid M[-n] \in \mathcal{V}\} = \mathcal{V}[n] \cap \ModR$ is closed under directed limits in $\ModR$. We will be mostly interested in the case when $\mathcal{V}$ is a definable subcategory of $\Der(R)$, and in this situation we know by the results of Section~\ref{SS:determined} that $\mathcal{V}_n = \{H^n(X) \mid X \in \mathcal{V}\}$ and that $\mathcal{V}_n$ is a definable subcategory of $\ModR$.
			\begin{lem}\label{L:mono}The assignments $\varphi_n, \psi_n$ are monotone functions $\mathcal{K}_n \rightarrow \mathcal{K}_n$.
			\end{lem}
			\begin{proof}
					First, we show that $\varphi_n$ and $\psi_n$ are functions $\mathcal{K}_n \rightarrow \mathcal{K}_n$, that is, they take values in $\mathcal{K}_n$. To do this, we need to show that $\kappa(\varphi_n(\pp))$ and $\kappa(\psi_n(\pp))$ are in $\mathcal{V}_n$ whenever $\pp \in \mathcal{K}_n$. By Lemma~\ref{L:belong}, we know that $R_{\varphi_n(\pp)}/\psi_n(\pp) \in \mathcal{V}_n$. Note that there are non-zero canonical maps
				$$\kappa(\varphi_n(\pp)) \rightarrow R_{\varphi_n(\pp)}/\psi_n(\pp)\text{, and}$$
$$\kappa(\psi_n(\pp)) \rightarrow R_{\varphi_n(\pp)}/\psi_n(\pp).$$
					By Lemma~\ref{L:RHom}, we have $\kappa(\varphi_n(\pp)), \kappa(\psi_n(\pp)) \in \mathcal{V}_n$.

				Now we need to show that $\varphi_n$ and $\psi_n$ are monotone. Consider $\pp_1 \subsetneq \pp_2$ in $\mathcal{K}_n$. If $\pp_1 \subseteq \varphi_n(\pp_2)$, there is nothing to prove in case of $\varphi_n$. Otherwise, if $\varphi_n(\pp_2) \subsetneq \pp_1$ there is an exact sequence
				$$0 \rightarrow \kappa(\pp{}_1) \rightarrow R_{\varphi_n(\pp_2)}/\pp{}_1 \rightarrow R_{\varphi_n(\pp_2)}/R_{\pp_1} \rightarrow 0.$$
				By our assumption, we have $\varphi_n(\pp_2) \subsetneq \pp_1 \subsetneq \pp_2$, and thus we can use Lemma~\ref{L67} and infer that $R_{\varphi_n(\pp_2)}/R_{\pp_1}$ is a direct limit of copies of $R_{\varphi_n(\pp_2)}/\pp_2 \in \mathcal{V}_n$. Therefore, we conclude that $R_{\varphi_n(\pp_2)}/\pp{}_1 \in \mathcal{V}_n$, and thus $\varphi_n(\pp_1) \subseteq \varphi_n(\pp_2)$. 

				Finally, we show that $\psi_n$ is monotone. Consider the exact sequence
				$$0 \rightarrow R_{\varphi_n(\pp_2)}/\psi_n(\pp{}_1) \rightarrow R_{\varphi_n(\pp_1)}/\psi_n(\pp{}_1) \rightarrow R_{\varphi_n(\pp_1)}/R_{\varphi_n(\pp_2)} \rightarrow 0.$$
				The middle term is in $\mathcal{V}_n$, and the rightmost term is in $\mathcal{V}_n$ by a similar argument as in the previous paragraph, since we can assume $\psi_n(\pp_1) \supsetneq \pp_2$ (otherwise the monotony is clear), and apply Lemma~\ref{L67} to prime ideals $\varphi_n(\pp_1) \subseteq \varphi_n(\pp_2) \subsetneq \psi_n(\pp_1)$. Then the leftmost term $R_{\varphi_n(\pp_2)}/\psi_n(\pp_1)$ belongs to $\mathcal{V}_n$, because $\mathcal{V}_n$ is closed under extensions and kernels of epimorphisms.

		\end{proof}
Given a couple of intervals $\chi, \xi$ in $\Spec(R)$, we say that $\chi$ is \EMP{contained} in $\xi$, denoted $\chi \subseteq \xi$, if $\pp_\xi \subseteq \pp_\chi \subseteq \qq_\chi \subseteq \qq_\xi$. The following Lemma explains the relation between the intervals of $\mathcal{X}_n$ and certain uniserial modules belonging to $\mathcal{V}_n$.
\begin{lem}\label{L:mc}
		\begin{enumerate}
			\item[(i)] For any $[\pp,\qq]$ in $\mathcal{X}_n$ we have $\ModRqqpp \subseteq \mathcal{V}_n$. 
			\item[(ii)] If $R_{\pp}/\qq \in \mathcal{V}_n$ for some prime ideals $\pp \subseteq \qq$ in $\Spec(R)$ then there is an interval $\chi \in \mathcal{X}_n$ which contains the formal interval $[\pp,\qq]$.
			\item[(iii)] If $\pp,\qq \in \mathcal{K}_n$ and $[\pp,\qq] \in \mathcal{X}_{n+1}$ then $R_{\qq}/\pp \in \mathcal{V}_n$.
		\end{enumerate}
	\end{lem}
	\begin{proof}
			$(i)$ Since $[\pp,\qq] \in \mathcal{X}_n$, we have $R_{\pp}/\qq[-n] \in \mathcal{V}$ by Lemma~\ref{L:belong}. Let $S$ be a maximal immediate extension of $R$, then also $C = R_{\pp}/\qq \otimes_R S \in \mathcal{V}_n$, since $S$ is flat and $\mathcal{V}_n$ is closed under direct limits. By Lemma~\ref{L68}, $C = R_{\pp}/\qq \otimes_R S$ cogenerates  $\ModRqqpp$. Therefore, there is a coresolution for any $M \in \ModRqqpp$ of the form
			$$0 \rightarrow M \rightarrow C^{\varkappa_0} \rightarrow C^{\varkappa_1} \rightarrow C^{\varkappa_2} \rightarrow \cdots$$
			for some cardinals $\varkappa_n, n \geq 0$.
			
			Since $\mathcal{V}$ is closed under cosuspensions, extensions, products, and homotopy limits, the truncated complex
			$$\cdots \rightarrow 0 \rightarrow C^{\varkappa_0} \rightarrow C^{\varkappa_1} \rightarrow C^{\varkappa_2} \rightarrow \cdots$$
			with the first non-zero component situated in degree $n$ belongs to $\mathcal{V}[n]$, and therefore $M[-n] \in \mathcal{V}$, which in turn means $M \in \mathcal{V}_n$. (The use of homotopy limits comes from expressing this complex as a countable directed homotopy limit of its stupid truncations from above.)

			$(ii)$ As in the proof of $(i)$, $R_{\pp}/\qq \in \mathcal{V}_n$ implies that $\ModRqqpp \subseteq \mathcal{V}_n$. In particular, $\kappa(\qq) \in \mathcal{V}_n$, and thus $\qq \in \mathcal{K}_n$. Then there is an interval $\chi = [\varphi_n(\qq),\psi_n(\qq)] \in \mathcal{X}_n$. By Lemma~\ref{L:mono1}, $\qq \subseteq \psi_n(\qq)$. On the other hand, the definition of the map $\varphi_n$ together with $R_{\pp}/\qq \in \mathcal{V}_n$ ensures that $\varphi_n(\qq) \subseteq \pp$. Therefore, $\chi$ contains the interval $[\pp,\qq]$.

			$(iii)$ There is the following exact sequence
			$$0 \rightarrow R_{\qq}/\pp \rightarrow \kappa(\qq) \oplus \kappa(\pp) \rightarrow R_{\pp}/\qq \rightarrow 0$$	
			where the map $R_{\qq}/\pp \rightarrow \kappa(\qq) \oplus \kappa(\pp)$ is given by the canonical projection and injection, respectively. Since $\pp, \qq \in \mathcal{K}_n$, the middle term of the sequence belongs to $\mathcal{V}_n$, and since the interval $[\pp,\qq]$ belongs to $\mathcal{X}_{n+1}$, we have $R_{\pp}/\qq \in \mathcal{V}_{n+1}$ by Lemma~\ref{L:belong}. Therefore, $R_{\qq}/\pp$ belongs to $\mathcal{V}_n$ by Proposition~\ref{P:moduletheoretic}.
	\end{proof}
			\begin{lem}\label{L:disjoint} The system of intervals $\mathcal{X}_n$ is disjoint.\end{lem}
					\begin{proof}
					With respect to \cite[Lemma 6.2]{B} and Lemma~\ref{L:mono1} and Lemma~\ref{L:mono}, it is enough to show that for any $n \in \mathbb{Z}$ we have the identities
				\begin{equation}\label{E001}\varphi_n \circ \psi_n = \varphi_n ~\&  ~\varphi_n \circ \varphi_n = \varphi_n,\end{equation}
					and
				\begin{equation}\label{E002}\psi_n \circ \varphi_n = \psi_n ~\&  ~\psi_n \circ \psi_n = \psi_n.\end{equation}
						Fix $\pp \in \mathcal{K}_n$. By Lemmas~\ref{L:mono1} and \ref{L:mono}, we have $\varphi_n(\pp) \subseteq \varphi_n(\psi_n(\pp))$. On the other hand, as $R_{\varphi_n(\pp)}/\psi_n(\pp) \in \mathcal{V}_n$ by Lemma~\ref{L:belong}, we have $\varphi_n(\psi_n(\pp)) \subseteq \varphi_n(\pp)$ by the definition of $\varphi_n$. 

							By Lemma~\ref{L:mono1}, $\varphi_n(\varphi_n(\pp)) \subseteq \varphi_n(\pp)$. There is an exact sequence

						$$0 \rightarrow R_{\varphi_n(\pp)}/\pp \rightarrow R_{\varphi_n^2(\pp)}/\pp \rightarrow R_{\varphi_n^2(\pp)}/R_{\varphi_n(\pp)} \rightarrow 0.$$
							The leftmost term is in $\mathcal{V}_n$. If we prove that also the rightmost term belongs to $\mathcal{V}_n$, then also $R_{\varphi_n^2(\pp)}/\pp \in \mathcal{V}_n$, which in turn implies $\varphi_n(\varphi_n(\pp)) = \varphi_n(\pp)$ by the definition of $\varphi_n$. First note that $R_{\varphi_n^2(\pp)}/R_{\varphi_n(\pp)}$ is an $R_{\varphi_n(\pp)}/\varphi_n^2(\pp)$-module. Indeed, $R_{\varphi_n^2(\pp)}/R_{\varphi_n(\pp)}$ is an $R_{\varphi_n(\pp)}$-module, and as it is clearly $(R \setminus \varphi_n^2(\pp))$-torsion, it is annihilated by $\varphi_n^2(\pp)$. Since $\varphi_n \in \mathcal{K}_n$, we have that $[\varphi_n^2(\pp),\varphi_n(\pp)]$ is contained in an interval from $\mathcal{X}_n$. Therefore Lemma~\ref{L:mc}(i) implies that any $R_{\varphi_n(\pp)}/\varphi_n^2(\pp)$-module belongs to $\mathcal{V}_n$, and thus in particular, $R_{\varphi_n^2(\pp)}/R_{\varphi_n(\pp)} \in \mathcal{V}_n$.
													
							Again by Lemma~\ref{L:mono1} and Lemma~\ref{L:mono}, we have $\psi_n(\varphi_n(\pp)) \subseteq \psi_n(\pp)$. Using $\varphi_n^2 = \varphi_n$, we have that $R_{\varphi_n(\pp)}/\psi_n(\pp) \in \mathcal{V}_n$ implies the other inclusion.

				To finish the proof of $(\ref{E002})$, we are left with showing that $\psi_n(\psi_n(\pp)) = \psi_n(\pp)$. Since clearly $\psi_n(\pp) \subseteq \psi_n(\psi_n(\pp))$, we have using (\ref{E001}) that applying $\varphi_n$ on the latter inequality yields $\varphi_n(\pp) = \varphi_n(\psi_n(\pp))$, and thus $R_{\varphi_n(\pp)}/\psi_n(\varphi_n(\pp)) \in \mathcal{V}_n$. This yields $\psi_n(\varphi_n(\pp)) \subseteq \psi_n(\pp)$, as desired.

				Using \cite[Lemma 6.2]{B}, we conclude that $\mathcal{X}_n$ is a disjoint system.
					\end{proof}
			\begin{lem}\label{L:idempotent} The prime ideal $\varphi_n(\pp)$ is idempotent for any $\pp \in \mathcal{K}_n$.\end{lem}
		\begin{proof}
				If $\varphi_n(\pp)$ is not idempotent, it is well known (see Lemma~\ref{L:VD}) that $\varphi_n(\pp) = rR_{\varphi_n(\pp)}$ for some element $r \in \varphi_n(\pp)$. Consider for any $n>0$ the exact sequence
				$$0 \rightarrow R_{\varphi_n(\pp)}/\pp \rightarrow r^{-n}R_{\varphi_n(\pp)}/\pp \rightarrow R_{\varphi_n(\pp)}/r^{n}R_{\varphi_n(\pp)} \rightarrow 0.$$
						Then the leftmost element belongs to $\mathcal{V}_n$, and since $R_{\varphi_n(\pp)}/r^{n}R_{\varphi_n(\pp)}$ is isomorphic to an $(n-1)$-fold extension of $R_{\varphi_n(\pp)}/\varphi_n(\pp) \simeq \kappa(\varphi_n(\pp))$, it also belongs to $\mathcal{V}_n$. Therefore $\varinjlim_{n>0} r^{-n}R_{\varphi_n(\pp)}/\pp \in \mathcal{V}_n$. Since $r \in \varphi_n(\pp)$, the localization $\varinjlim_{n>0} r^{-n}R_{\varphi_n(\pp)}$ is isomorphic to $R_{\qq}$ for some prime ideal $\qq \subsetneq \varphi_n(\pp)$. Then the direct limit $\varinjlim_{n>0} r^{-n}R_{\varphi_n(\pp)}/\pp$ is isomorphic to $R_{\qq}/\pp$. But then $R_{\qq}/\pp \in \mathcal{V}_n$, which is a contradiction with the definition of $\varphi_n(\pp)$.
		\end{proof}

				In order to prove the completeness condition, we will make an essential use of the recent deep result \cite[Theorem A]{SSV}, which states that a t-structure in the underlying category of a strong and stable derivator can be naturally lifted to the category of coherent diagrams of any shape. In the case of a homotopically smashing t-structure, this allows in a sense to ``commute'' the coaisle approximation functor with a directed homotopy colimit, as in the following proof.
			\begin{lem}\label{L:nogaps} The set $\mathcal{X}_n$ satisfies the completeness condition of Definition~\ref{D:admissiblesystem}.\end{lem}
					\begin{proof}
							It is enough to show the following claim: for any $n \in \mathbb{Z}$ and any non-empty subset $A$ of $\mathcal{K}_n$ the primes $\bigcup_{\pp \in A} \pp$ and $\bigcap_{\pp \in A} \pp$ belong to $\mathcal{K}_n$. Indeed, once we have this, then given any non-empty subset $B$ of $\mathcal{X}_n$, we let $A = \{\qq \mid [\pp,\qq] \in B\}$. Suppose that $B$ does not have a maximal element, then $\bigcup_{\qq \in A} \qq = \bigcup_{[\pp,\qq] \in B}\pp$. Therefore, the claim gives $\bigcup_{[\pp,\qq] \in B}\pp \in \mathcal{K}_n$, and so there is an interval $[\varphi_n(\bigcup_{[\pp,\qq] \in B}\pp),\psi_n(\bigcup_{[\pp,\qq] \in B}\pp)] \in \mathcal{X}_n$. By Lemma~\ref{L:mono} and (\ref{E001}), we have $\pp = \varphi_n(\pp) \subseteq \varphi_n(\bigcup_{[\pp,\qq] \in B}\pp)$ for any $[\pp,\qq] \in B$, and thus necessarily $\varphi_n(\bigcup_{[\pp,\qq] \in B}\pp) = \bigcup_{[\pp,\qq] \in B}\pp$ by Lemma~\ref{L:mono1}. The second part of the completeness condition follows by an analogous argument.

							It remains to prove the claim. Let $A$ be a non-empty subset of $\mathcal{K}_n$ and let $(\Lambda, \leq)$ be a totally ordered set (considered naturally as a small category) such that we can write $A = \{\pp_\alpha \mid \alpha \in \Lambda\}$ in a way that $\alpha \leq \beta$ if and only if $\pp_\alpha \subseteq \pp_\beta$ for all $\alpha, \beta \in \Lambda$. We put $\pp = \bigcup_{\alpha \in \Lambda}\pp_\alpha$ and $\qq = \bigcap_{\alpha \in \Lambda}\pp_\alpha$. We need to prove that $\kappa(\pp)$ and $\kappa(\qq)$ belong to $\mathcal{V}_n$.

				Let us express $\kappa(\pp)$ as the direct limit $\varinjlim_{\alpha \in \Lambda}R_{\pp}/\pp_\alpha$ of the direct system $\mathcal{Y} = (R_{\pp}/\pp_\alpha \mid \alpha \in \Lambda) \in (\ModR)^\Lambda$ consisting of the natural surjections. Let $\mathscr{Y} \in \Der((\ModR)^\Lambda)$ be the coherent diagram induced by $\mathcal{Y} \in (\ModR)^\Lambda$. By \cite[Theorem A]{SSV}, there is a t-structure $(\mathcal{U}_\Lambda,\mathcal{V}_\Lambda)$ in $\Der((\ModR)^\Lambda)$, where $\mathcal{V}_\Lambda$ (resp. $\mathcal{U}_\Lambda$) consists of all coherent diagrams of shape $\Lambda$ with all coordinates in $\mathcal{V}$ (resp. $\mathcal{U}$). Let
$$\Delta: \mathscr{U} \rightarrow \mathscr{Y}[-n] \xrightarrow{f} \mathscr{V} \rightarrow \mathscr{Y}[1]$$
be the approximation triangle in $\Der((\ModR)^\Lambda)$ of the coherent diagram $\mathscr{Y}[-n]$ with respect to the t-structure $(\mathcal{U}_\Lambda,\mathcal{V}_\Lambda)$. For each $\alpha \in \Lambda$, denote by $U_\alpha$ and $V_\alpha$ the $\alpha$-th coordinates of $\mathscr{U}$ and $\mathscr{V}$. By passing to a coordinate $\alpha \in \Lambda$, $\Delta$ induces a triangle
$$\Delta_\alpha: U_\alpha \rightarrow R_{\pp}/\pp{}_\alpha[-n] \xrightarrow{f_\alpha} V_\alpha \rightarrow U_\alpha[1],$$
which is the approximation triangle of $R_{\pp}/\pp_\alpha[-n]$ with respect to the t-structure $(\mathcal{U},\mathcal{V})$ in $\Der(R)$. 

 Note that for any $\alpha \in \Lambda$, there is a canonical embedding $R_{\pp}/\pp_\alpha \subseteq \kappa(\pp_\alpha)$. Let $\iota_\alpha: R_{\pp}/\pp_\alpha[-n] \rightarrow \kappa(\pp_\alpha)[-n]$ be a map in $\Der(R)$ inducing this embedding in the $n$-th cohomology. Since $\kappa(\pp_\alpha) \in \mathcal{V}_n$ by the assumption, applying the coaisle approximation functor $\tau_\mathcal{V}: \Der(R) \rightarrow \mathcal{V}$ onto $\iota_\alpha$ yields a commutative diagram in $\Der(R)$ as follows:
				\begin{equation}\label{E004}
				\begin{CD}
					R_{\pp}/\pp_\alpha[-n] @>\iota_\alpha >> \kappa(\pp_\alpha)[-n] \\
					@V f_\alpha VV @V\simeq VV \\
					V_\alpha @>\tau_\mathcal{V}(\iota_\alpha)>> \kappa(\pp_\alpha)[-n]
				\end{CD}
				\end{equation}
				Applying the $n$-th homology functor on (\ref{E004}) yields a commutative diagram in $\ModR$:
$$
				\begin{CD}
					R_{\pp}/\pp_\alpha @>\subseteq >> \kappa(\pp_\alpha) \\
					@V H^n(f_\alpha) VV @V\simeq VV \\
					H^n(V_\alpha) @>H^n(\tau_\mathcal{V}(\iota_\alpha))>> \kappa(\pp_\alpha)
				\end{CD}
$$
 The latter diagram shows that $H^n(f_\alpha): R_{\pp}/\pp_\alpha \rightarrow H^n(V_\alpha)$ is a monomorphism for each $\alpha \in \Lambda$. By \cite[Corollary 4.19]{G}, there is a triangle obtained by taking the homotopy colimit of triangle $\Delta$ in $\Der((\ModR)^\Lambda)$:
\begin{equation}\label{E007}\hocolim_{\alpha \in \Lambda} \Delta: \hocolim_{\alpha \in \Lambda} \mathscr{U} \rightarrow \kappa(\pp)[-n] \xrightarrow{\hocolim_{\alpha \in \Lambda}f} \hocolim_{\alpha \in \Lambda} \mathscr{V} \rightarrow \hocolim_{\alpha \in \Lambda}\mathscr{U}[1].\end{equation}
Since both $\mathcal{U}$ (\cite[Proposition 4.2]{SSV}) and $\mathcal{V}$ are closed under directed homotopy colimits, we have that (\ref{E007}) is the approximation triangle of $\kappa(\pp)[-n]$ with respect to the t-structure $(\mathcal{U},\mathcal{V})$. We compute the $n$-th cohomology of the coaisle approximation map,
$$H^n(\hocolim_{\alpha \in \Lambda}f) \simeq \varinjlim_{\alpha \in \Lambda}H^n(f_\alpha),$$
which together with the previous computation and the exactness of direct limits in $\ModR$ shows that $H^n(\hocolim_{\alpha \in \Lambda}f): \kappa(\pp) \rightarrow H^n(\hocolim_{\alpha \in \Lambda} \mathscr{V})$ is a monomorphism in $\ModR$. In particular, we proved that
$$\Hom{}_{\Der(R)}(\kappa(\pp)[-n],\hocolim_{\alpha \in \Lambda} \mathscr{V}) \neq 0.$$ 
Since $\hocolim_{\alpha \in \Lambda} \mathscr{V} \in \mathcal{V}$, Lemma~\ref{L:RHom} shows that $\kappa(\pp)[-n] \in \mathcal{V}$, and thus $\kappa(\pp) \in \mathcal{V}_n$.

We prove that $\kappa(\qq) \in \mathcal{V}_n$ using a similar argument. This time we express $\kappa(\qq)$ as the direct limit $\varinjlim_{\alpha \in \Lambda^{\text{op}}} R_{\pp_\alpha}/\qq$ of the direct system $\mathcal{Y}=(R_{\pp_\alpha}/\qq \mid \alpha \in \Lambda^{\text{op}})$ consisting of canonical embeddings, which again lifts to a coherent diagram $\mathscr{Y} \in \Der((\ModR)^{\Lambda^{\text{op}}})$. We observe that there are monomorphisms $R_{\pp_\alpha}/\qq \xhookrightarrow{} \prod_{\beta < \alpha \in \Lambda}\kappa(\pp_\beta)$ given by canonical maps in each coordinate $\beta < \alpha$ of the product, and $\prod_{\beta < \alpha \in \Lambda}\kappa(\pp_\beta) \in \mathcal{V}_n$ using that $\mathcal{V}_n$ is closed under products. As in the previous part of the proof, these embeddings can be used to show that the coaisle approximation maps $R_{\pp_\alpha}/\qq[-n] \xrightarrow{f_\alpha} V_\alpha:=\tau_\mathcal{V}(R_{\pp_\alpha}/\qq[-n])$ induce monomorphisms $H^n(f_\alpha)$ in the $n$-th cohomology. Repeating the argument with the homotopy colimit to show that the coaisle approximation map $\kappa(\qq)[-n] \rightarrow \hocolim_{\alpha \in \Lambda}V_\alpha$ is non-zero in $n$-cohomology, we again conclude that $\kappa(\qq)[-n] \in \mathcal{V}$ by Lemma~\ref{L:RHom}. 
\end{proof}
Putting together Lemma~\ref{L:disjoint}, \ref{L:idempotent}, and \ref{L:nogaps}, we obtain:
\begin{cor}\label{C:admissible}
		In the setting of Definition~\ref{D:phipsi}, the set $\mathcal{X}_n$ is an admissible system on $\Spec(R)$ for any $n \in \mathbb{Z}$.
\end{cor}
	The sequence $\mathbb{X} = (\mathcal{X}_n \mid n \in \mathbb{Z})$ of admissible systems in $\Spec(R)$ satisfies two additional properties that will characterize it as a sequence associated to a definable coaisle.

	\begin{definition}\label{D:admissiblefiltration}
		Let $R$ be a valuation domain. We say that a sequence $\mathbb{X} = (\mathcal{X}_n \mid n \in \mathbb{Z})$ of admissible systems on $\Spec(R)$ is a \EMP{nested} sequence if $\mathcal{X}_n$ is a \EMP{nested subsystem} of $\mathcal{X}_{n+1}$ for each $n \in \mathbb{Z}$, meaning that for any $\chi \in \mathcal{X}_n$ there is $\xi \in \mathcal{X}_{n+1}$ such that $\chi \subseteq \xi$.
	
		We say that $\mathbb{X} = (\mathcal{X}_n \mid n \in \mathbb{Z})$ satisfies the \EMP{degreewise non-density condition} if the following holds:
		\begin{itemize}
			\item For any $n \in \mathbb{Z}$ and any dense interval $\chi < \xi$ in $\mathcal{X}_{n+1}$, there is an interval $\chi < \tau < \xi$ in $\mathcal{X}_{n+1}$ such that $\tau$ does not contain any interval from $\mathcal{X}_n$.
	\end{itemize}

		An \EMP{admissible filtration} in $\Spec(R)$ is a nested sequence $\mathbb{X} = (\mathcal{X}_n \mid n \in \mathbb{Z})$ of admissible systems satisfying the degreewise non-density condition.
	\end{definition}

	\begin{remark}\label{R:gap}
		The condition of $\tau \in \mathcal{X}_{n+1}$ not containing any interval from $\mathcal{X}_n$ in the definition of the degreewise non-density condition above can be rephrased by $\tau$ being strictly contained in a gap from $\mathcal{G}(\mathcal{X}_n)$. Indeed, assume that $\tau$ does not contain any interval from $\mathcal{X}_n$. Let $\chi$ be the maximal interval in $\mathcal{X}_n$ such that $\chi < \tau$ and $\xi$ be a minimal interval in $\mathcal{X}_n$ with $\xi > \tau$, such intervals exist by the completeness condition satisfied by any admissible system (Definition~\ref{D:admissiblesystem}). Since $\mathcal{X}_n$ is a nested subsystem of the admissible system $\mathcal{X}_{n+1}$, we have $\qq_\chi \subsetneq \pp_\tau \subseteq \qq_\tau \subsetneq \pp_\xi$. If there was an interval $\theta \in \mathcal{X}_n$ with $\chi < \theta < \xi$, then necessarily $\theta \subseteq \tau$, a contradiction. Therefore, $(\qq_\chi,\pp_\xi)$ is the desired gap in $\mathcal{G}(\mathcal{X}_n)$ containing the interval $\tau$ strictly.

	Conversely, the definition of a gap (\S~\ref{SS:densityandgaps}) implies that if $\tau$ is contained in a gap from $\mathcal{G}(\mathcal{X}_n)$ then $\tau$ cannot contain any interval from $\mathcal{X}_n$. Therefore, the degreewise non-density condition can be equivalently formulated as follows:
		\begin{itemize}
			\item For any $n \in \mathbb{Z}$ and any dense interval $\chi < \xi$ in $\mathcal{X}_{n+1}$, there is an interval $\chi < \tau < \xi$ in $\mathcal{X}_{n+1}$ and a gap $(\qq,\pp) \in \mathcal{G}(\mathcal{X}_n)$ which strictly contains $\tau$ (that is, $\qq \subsetneq \pp_\tau \subseteq \qq_\tau \subsetneq \pp$).
	\end{itemize}
	\end{remark}
	
	Before proving that the sequence we associated to a definable coaisle is indeed an admissible filtration, we remark a useful equivalent formulation and one consequence of the degreewise non-density condition. We point the reader to the definition of the set $\mathcal{H}(\mathcal{X}_n)$ of maximal dense intervals of $\mathcal{X}_n$ in \S~\ref{SS:densityandgaps}.

	\begin{lem}\label{L:dnequiv}
			Let $\mathbb{X} = (\mathcal{X}_n \mid n \in \mathbb{Z})$ be a nested sequence of admissible systems. Then the degreewise non-density condition is equivalent to the following:
			\begin{itemize}
				\item For any $n \in \mathbb{Z}$ and for any maximal dense interval $C \in \mathcal{H}(\mathcal{X}_{n+1})$, the subset
				$$\mathcal{Z}_C = \{\tau \in C \mid \text{ $\tau$ does not contain any interval from $\mathcal{X}_n$}\}$$
				is dense in $C$ (that is, for any $\chi < \xi$ in $C$ there is $\tau \in \mathcal{Z}_C$ with $\chi < \tau < \xi$).
			\end{itemize}
	\end{lem}
	\begin{proof}
		This follows easily from the definition of $\mathcal{H}(\mathcal{X}_{n+1})$. Indeed, for any dense interval $\chi < \xi$ in $\mathcal{X}_{n+1}$ we have $\chi \sim \xi$, and thus there is $C \in \mathcal{H}(\mathcal{X}_{n+1})$ such that for any $\chi \leq \tau \leq \xi$ we have $\tau \in C$. Since $\mathcal{Z}_C$ is dense in $C$, we infer that there is $\tau \in \mathcal{Z}_C$ such that $\chi < \tau < \xi$, and thus the condition of the Lemma implies the degreewise non-density condition. For the converse, let $\chi < \xi$ be in $C$, then the interval $\chi < \xi$ is dense by the definition of $C$, and therefore there is $\tau$ in between $\chi$ and $\xi$ belonging to $\mathcal{Z}_C$ by the degreewise non-density condition.
	\end{proof}
	\begin{lem}\label{L:dncons}
		Let $\mathbb{X} = (\mathcal{X}_n \mid n \in \mathbb{Z})$ be an admissible filtration. If $\chi < \xi$ is a dense interval in $\mathcal{X}_n$ then there is an interval $\mu \in \mathcal{X}_{n+1}$ such that $\chi, \xi \subseteq \mu$.
	\end{lem}
	\begin{proof}
		We define a set 
		$$A = \{\tau \in \mathcal{X}_{n+1} \mid \text{$\tau$ contains an interval from $\mathcal{X}_n$ in between $\chi$ and $\xi$}\}.$$
		Since $\mathbb{X}$ is nested, each interval from $\mathcal{X}_n$ between $\chi$ and $\xi$ is contained in some $\tau \in A$. Our aim is to show that $A$ is a singleton, because then the interval $\mu \in \mathcal{X}_{n+1}$ with $A = \{\mu\}$ has the desired property.
	
		First, we remark that $A$ is clearly non-empty, so it is enough to show that $A$ does not contain two distinct elements. Suppose that there are intervals $\tau < \theta$ in $A$. The intervals $\tau$ and $\theta$ are disjoint by the definition of an admissible system. For any ideal $I$ with $\qq_\tau \subseteq I \subseteq \pp_\theta$, there is an interval $\gamma \in \mathcal{X}_n$ satisfying $\pp_\gamma \subseteq I \subseteq \qq_\gamma$ and $\chi < \gamma < \xi$, this follows from Lemma~\ref{L:denseideal}. Then $\gamma$ is contained in some interval $\delta \in A$, and necessarily $\tau < \delta < \theta$. In this way the density of the interval $\chi < \xi$ in $\mathcal{X}_n$ implies the density of the interval $\tau < \theta$ in $\mathcal{X}_{n+1}$. Since each interval from $\mathcal{X}_{n+1}$ which lies in between $\tau$ and $\theta$ contains an interval from $\mathcal{X}_n$ by Lemma~\ref{L:denseideal} again, this is in contradiction with the degreewise non-density condition.
	\end{proof}

	\begin{prop}\label{P:coaislestointervals}
			Let $R$ be a valuation domain. The assignment $\mathcal{V} \mapsto \mathbb{X} = (\mathcal{X}_n)_{n \in \mathbb{Z}}$ of Definition~\ref{D:phipsi} yields a map
			$$\Theta: \left \{ \begin{tabular}{ccc} \text{ definable coaisles $\mathcal{V}$} \\ \text{in $\Der(R)$} \end{tabular}\right \}  \rightarrow  \left \{ \begin{tabular}{ccc} \text{ admissible filtrations $\mathbb{X}$ } \\ \text{ in $\Spec(R)$} \end{tabular}\right \}.$$
	\end{prop}
	\begin{proof}
			Let $\mathcal{V}$ be a definable coaisle in $\Der(R)$. Then the sequence of admissible systems associated to $\mathcal{V}$ via Definition~\ref{D:phipsi} is clearly nested, because $\mathcal{V}_n \subseteq \mathcal{V}_{n+1}$ for all $n \in \mathbb{Z}$.

			The only thing which remains to be proved is that the nested sequence $\mathbb{X} = (\mathcal{X}_n \mid n \in \mathbb{Z})$ associated to $\mathcal{V}$ satisfies the degreewise non-density condition. Towards a contradiction, let us assume that there is $n \in \mathbb{Z}$ and a dense interval $\chi < \xi$ in $\mathcal{X}_{n+1}$ such that every interval $\tau \in \mathcal{X}_{n+1}$ with $\chi < \tau < \xi$ contains some interval $\tau_0$ from $\mathcal{X}_n$.

			The density of the interval together with the completeness property satisfied by admissible systems implies that for each $\chi < \tau \leq \xi$, we can write $\pp_\tau$ as the union $\pp_\tau = \bigcup_{\chi < \theta < \tau}\pp_\theta$. By our assumption there is an interval $\theta_0 \in \mathcal{X}_n$ contained in $\theta$ for each $\chi < \theta < \tau$, and therefore the completeness property yields an interval of the form $[\pp_\tau, \qq]$ in $\mathcal{X}_n$ with $\pp_\tau \subseteq \qq \subseteq \qq_\tau$. By a completely dual argument, there is also an interval of the form $[\pp,\qq_\tau]$ in $\mathcal{X}_n$ for any $\chi \leq \tau < \xi$ with $\pp_\tau \subseteq \pp \subseteq \qq_\tau$.

			We claim that the module $M = R_{\pp_\xi}/{\qq_\chi}$ belongs to $\mathcal{V}_n$ and show that this leads to the desired contradiction. We prove this in several steps.

			\textbf{Step I.} The module $\prod_{\chi \leq \tau \leq \xi} M \otimes_R R_{\qq_\tau}/\pp_\tau$ belongs to $\mathcal{V}_n$ for any $\chi \leq \tau \leq \xi$. 

			Since $\mathcal{V}_n$ is closed under products, it is enough to show that the factors of the product belong to $\mathcal{V}_n$. First, $M \otimes_R R_{\qq_\chi}/\pp_\chi \simeq \kappa(\qq_\chi) \in \mathcal{V}_n$, as there is an interval $[\pp,\qq_\chi] \in \mathcal{X}_n$, and so $\qq_\chi \in \mathcal{K}_n$. For any $\chi < \tau < \xi$, we have $M \otimes_R R_{\qq_\tau}/\pp_\tau \simeq R_{\qq_\tau}/\pp_\tau \in \mathcal{V}_n$ by Lemma~\ref{L:mc}(iii), since there are intervals of the form $[\pp_\tau,\qq]$ and $[\pp,\qq_\tau]$ in $\mathcal{X}_n$ contained both in the interval $[\pp_\tau,\qq_\tau] \in \mathcal{X}_{n+1}$. Finally, $M \otimes_R R_{\qq_\xi}/\pp_\xi \simeq \kappa(\pp_\xi) \in \mathcal{V}_n$, as there is an interval $[\pp_\xi,\qq] \in \mathcal{X}_n$.

			\textbf{Step II.} The natural map $\eta: M \rightarrow \prod_{\chi \leq \tau \leq \xi} M \otimes_R R_{\qq_\tau}/\pp_\tau$ is a monomorphism and $\Coker(\eta) \in \mathcal{V}_{n+1}$.

			Let $\mathcal{Y} = \{\tau \in \mathcal{X}_{n+1} \mid \chi \leq \tau \leq \xi\}$. Note that $\mathcal{Y}$ is naturally an admissible system in the spectrum of the valuation domain $U = R_{\qq_\xi}/\pp_\chi$, and that $\mathcal{Y}$ is dense everywhere as such. Recall that the idempotency of $\pp_\chi$ ensures that the natural ring homomorphism $R \rightarrow U$ is a homological ring epimorphism, and therefore both the homomorphisms and the extensions in $\ModU$ can be equivalently computed over $R$. Let $S$ be a maximal immediate extension of $U$. Then Proposition~\ref{P:cotiltingdense} yields that 
		$$C:=\prod_{\tau \in \mathcal{Y}} ((U_{\pp_\tau}/\qq{}_\tau) \otimes_U S) = \prod_{\tau \in \mathcal{Y}} ((R_{\pp_\tau}/\qq{}_\tau) \otimes_R S)$$ 
		is a 1-cotilting $U$-module corresponding to the cotilting class $\mathcal{C}_\mathcal{Y} = \Cogen(C) = {}^{\perp_1}C$ in $\ModU$. The module $M = R_{\pp_\xi}/{\qq_\chi}$ is a $U$-module an clearly belongs to $\mathcal{C}_\mathcal{Y}$. Consider the universal map $\nu: M \rightarrow C^{\Hom_R(M,C)}.$

		 Since $M \in \mathcal{C}_\mathcal{Y} = \Cogen(C)$, the map $\nu$ is a monomorphism, and by applying $\Hom_U(-,C)$ we obtain an exact sequence
		$$\xrightarrow{\Hom_U(\nu,C)} \Hom{}_U(M,C) \rightarrow \Ext{}_U^1(\Coker(\nu),C) \rightarrow \Ext{}_U^1(C^{\Hom_U(M,C)},C).$$
		By the universality of $\nu$, $\Hom_R(\nu,C)$ is an epimorphism, and because $C$ is a 1-cotilting $U$-module, we also have $\Ext_U^1(C^{\Hom_R(M,C)},C) = 0$. Therefore, we can infer $\Ext_U^1(\Coker(\nu),C) = 0$, and thus $\Coker(\nu) \in \mathcal{C}_\mathcal{Y}$.

		As $C = \prod_{\tau \in \mathcal{Y}} ((R_{\pp_\tau}/\qq_\tau) \otimes_U S)$, the map $\nu$ is given coordinate-wise by maps $\nu_{\tau}: M \rightarrow ((R_{\pp_\tau}/\qq_\tau) \otimes_U S)^{\Hom_R(M,C)}$. Since $(R_{\pp_\tau}/\qq_\tau) \otimes_U S$ is an $R_{\qq_\tau}/\pp_\tau$-module, the map $\nu_{\tau}$ factors through the natural map $M \rightarrow M \otimes_R R_{\qq_{\tau}}/\pp_{\tau}$. Therefore, we have the following commutative diagram:
$$
\begin{CD}
	0 @>>> M @>\eta>> \prod_{\chi \leq \tau \leq \xi} M \otimes_R R_{\qq_\tau}/\pp_\tau @>>> \Coker(\eta) @>>> 0 \\
	& & @| @V f VV @V g VV \\
	0 @>>> M @>\nu>> C^{\Hom_R(M,C)} @>>> \Coker(\nu) @>>> 0 
\end{CD}
$$
The map $\eta$ is a monomorphism, because $\nu$ is a monomorphism. The module $\prod_{\chi \leq \tau \leq \xi} M \otimes_R R_{\qq_\tau}/\pp_\tau$ belongs to $\mathcal{C}_\mathcal{Y}$, and so does its submodule $\Ker(f)$. By the Snake Lemma, we have $\Ker(f) \simeq \Ker(g)$. Because $\Coker(\eta)$ is an extension of $\Ker(g)$ and a submodule of $\Coker(\nu) \in \mathcal{C}_{\mathcal{Y}}$, we see that $\Coker(\eta) \in \mathcal{C}_{\mathcal{Y}}$. We claim that this implies that $\Coker(\eta) \in \mathcal{V}_{n+1}$. Indeed, since $\Coker(\eta) \in \mathcal{C}_\mathcal{Y} = \Cogen(C) = {}^{\perp_1}C$, then by iterating the natural map to products of $C$, we get that $\Coker(\eta)$ admits a $\Prod(C)$-coresolution (see also \cite[Proposition 15.5(a)]{GT}), that is, an exact sequence of the form
			$$0 \rightarrow \Coker(\eta) \rightarrow C_0 \rightarrow C_1 \rightarrow C_2 \rightarrow \cdots$$
where $C_i$ is a direct product of copies of $C$. Since $\tau \in \mathcal{X}_{n+1}$ for each $\chi \leq \tau \leq \xi$, we have $R_{\pp_\tau}/\qq_\tau \in \mathcal{V}_{n+1}$ by Lemma~\ref{L:belong}. As $S$ is a flat $U$-module, the module $(R_{\pp_\tau}/\qq_\tau) \otimes_U S$ is isomorphic to a direct limit of copies of $R_{\pp_\tau}/\qq_\tau$ for any $\tau \in \mathcal{Y}$, and therefore $(R_{\pp_\tau}/\qq_\tau) \otimes_U S \in \mathcal{V}_{n+1}$. Since $\mathcal{V}_{n+1}$ is closed under direct products, we showed that $C = \prod_{\tau \in \mathcal{Y}} ((R_{\pp_\tau}/\qq_\tau) \otimes_U S)$ belongs to $\mathcal{V}_{n+1}$, and thus so does also the module $C_i$ for each $i \geq 0$. Arguing as in the proof of Lemma~\ref{L:mc}(i), we conclude that $\Coker(\eta)$ belongs to $\mathcal{V}_{n+1}$.

		\textbf{Step III.} The module $M$ belongs to $\mathcal{V}_n$, and this leads to a contradiction.

		In Steps I. and II. we showed that $M$ is the kernel of a map $\prod_{\chi \leq \tau \leq \xi} M \otimes_R R_{\qq_\tau}/\pp_\tau \rightarrow \Coker(\eta)$ between a module from $\mathcal{V}_n$ and from $\mathcal{V}_{n+1}$, respectively. Therefore, $M$ belongs to $\mathcal{V}_n$ by Proposition~\ref{P:moduletheoretic}. Now consider the obvious map $M = R_{\pp_\xi}/\qq_\chi \rightarrow \kappa(\qq_\chi) \oplus \kappa(\pp_\xi)$. This map is a monomorphism between modules from $\mathcal{V}_n$ and its cokernel is $R_{\qq_\chi}/\pp_\xi$, which implies $R_{\qq_\chi}/\pp_\xi \in \mathcal{V}_{n+1}$ by Proposition~\ref{P:moduletheoretic} again. But then Lemma~\ref{L:mc}(ii) yields that there is an interval in $\mathcal{X}_{n+1}$ which contains $[\qq_\chi,\pp_\xi]$, a contradiction with the interval $\chi < \xi$ in $\mathcal{X}_{n+1}$ being dense.
	\end{proof}

	\begin{remark}\label{R:spec}
	Let us remark what the degreewise non-density condition means in two ``extremal'' cases --- that of a stable t-structure, and that of the Happel-Reiten-Smal\o ~t-structure. 
			\begin{itemize}
				\item If $\mathcal{V}$ is closed under suspension, then the associated admissible filtration $\mathbb{X}$ is necessarily constant, that is, $\mathcal{X}_n = \mathcal{X}_{n+1}$ for all $n \in \mathbb{Z}$. The degreewise non-density then simply means, in view of Lemma~\ref{L:dncons}, that $\mathcal{X}_n$ is nowhere dense --- cf. \cite[Theorem 5.23]{BS}.
				\item If $\mathcal{V}$ belongs to a Happel-Reiten-Smal\o ~t-structure, then 
						$$\mathcal{X}_n = \begin{cases} \emptyset, & n<0 \\ \{[0,\mm]\}, & n>0 \end{cases},$$
								and the only interesting admissible system is $\mathcal{X}_0$, for which the degreewise non-density condition is vacuous. This fits nicely with Theorem~\ref{T01}, where no condition on density was required.
			\end{itemize}
	\end{remark}

	\section{Construction of definable coaisles}\label{S:inttocoaisle}
	The purpose of this section is to construct an injective assignment from admissible filtrations to definable coaisles. Given an admissible filtration $\mathbb{X}=(\mathcal{X}_n \mid n \in \mathbb{Z})$ in the spectrum of a valuation domain $R$, we define the following subcategories of $\ModR$ (see Lemma~\ref{L01}):
	$$\mathcal{C}_n = \mathcal{C}_{\mathcal{X}_n} = \{M \in \ModR \mid \forall 0 \neq m \in M ~\exists \chi \in \mathcal{X}_n: \Ann{}_R(m) \in \langle \chi \rangle\}$$
	and
	$$\mathcal{D}_n = \{M \in \ModR \mid F{}_{\qq}(M) \in \ModRqq ~\forall [\pp,\qq] \in \mathcal{X}_n\}.$$
We recall that for any $M \in \ModR$ we have $F_{\qq}(M) = \operatorname{Im}(M \rightarrow M \otimes_R R_{\qq})$, and thus $F_{\qq}(M) \in \ModRqq$ if and only if $F_{\qq}(M)$ is $\qq$-divisible. Our goal here is to show that there is a definable coaisle $\mathcal{V}$ defined on cohomology by putting $\mathcal{V}_n = \mathcal{C}_n \cap \mathcal{D}_{n+1}$ for each $n \in \mathbb{Z}$.

We refer the reader to \S \ref{SS:homology} for the definition of the complex $K(\qq,\pp)$ for a gap $(\qq,\pp)$.


		\begin{lem}\label{L:Dn}
			Let $M$ be an $R$-module, $\qq \in \Spec(R)$, and $J$ any ideal of $R$ such that $\qq \subsetneq J$. Then $F_{\qq}(M) \in \ModRqq$ if and only if $(R_{\qq}/J) \otimes_R M = 0$. In particular, for any $\qq \in \Spec^*(R)$ we have $F_{\qq}(M) \in \ModRqq$ if and only if $H^1(K(\qq,\pp) \otimes_R M) = 0$ for any gap of the form $(\qq,\pp)$ for $\qq \subsetneq \pp$. 
			
			As a consequence, the class $\mathcal{D}_n$ is closed under pure submodules, direct limits, extensions, and epimorphic images.
	\end{lem}
	\begin{proof}
		This is straightforward for $J=R$. Since for any ideal $J$ such that $\qq \subsetneq J \subsetneq R$ there is a chain of epimorphisms $R_{\qq}/sR \rightarrow R_{\qq}/J \rightarrow R_{\qq}/R$ for some $s \in (J \setminus \qq)$, it is enough to check the statement for $J$ being a principal ideal. But for any $s \in (R \setminus \qq)$, the module $R_{\qq}/sR$ is clearly isomorphic to $R_{\qq}/R$, and so this follows from the case $J = R$.
			
			For the ``in particular'' claim, note that $H^1(K(\qq,\pp) \otimes_R M)$ is always zero if $\qq = -\infty$, and is equal to $(R_{\qq}/\pp) \otimes_R M$ if $\qq \in \Spec(R)$.
	\end{proof}
	Given an admissible filtration $\mathbb{X}$, let us denote shortly $\mathcal{G}_n = \mathcal{G}(\mathcal{X}_n)$ the set of all gaps of the admissible system $\mathcal{X}_n$. We recall the definition of the set $\mathcal{H}(\mathcal{X}_n)$ from \S~\ref{SS:densityandgaps}, and also use the shorter notation $\mathcal{H}_n = \mathcal{H}(\mathcal{X}_n)$. The elements $C$ of $\mathcal{H}_n$ can be viewed as maximal dense intervals in $\mathcal{X}_n$, and by Lemma~\ref{L:dnequiv}, each $C \in \mathcal{H}_{n+1}$ contains a dense subset $\mathcal{Z}_C$ consisting of those intervals which do not contain any interval from the preceding system $\mathcal{X}_{n}$. Our first step is to describe the classes $\mathcal{V}_n$ homologically. Before that, it will be useful to record the following dichotomy which follows from the degreewise non-density condition.
	\begin{lem}\label{L:denscons}
			Let $\mathbb{X}=(\mathcal{X}_n \mid n \in \mathbb{Z})$ be an admissible filtration. Then for any interval $\chi \in \mathcal{X}_{n}$, one (and only one) of the following two conditions must be true: 
		\begin{enumerate}
				\item[(i)] there is a (possibly not strictly) decreasing sequence $((\qq_\alpha,\pp_\alpha) \mid \alpha < \lambda)$ of gaps in $\mathcal{G}_{n}$ such that $\bigcap_{\alpha < \lambda}\qq_{\alpha} = \qq_\chi$, or
				\item[(ii)] there is a limit ordinal $\lambda$ and a strictly decreasing sequence $([\pp_\alpha,\qq_\alpha] \mid \alpha < \lambda)$ of intervals in $\mathcal{Z}_C$ for some $C \in \mathcal{H}_n$ such that $\bigcap_{\alpha < \lambda}\qq_{\alpha} = \qq_\chi$.
			\end{enumerate}
	\end{lem}
	\begin{proof}
			Suppose that the first condition is not true. In view of Lemma~\ref{L:denseideal}, the condition $(i)$ is not true if and only if there is $\xi \in \mathcal{X}_n$ such that $\chi < \xi$ and the interval $\chi < \xi$ is dense in $\mathcal{X}_n$. Then there is $C \in \mathcal{H}_n$ such that $\chi,\xi \in C$. Let $Z = \{\tau \in \mathcal{Z}_C \mid \chi < \tau < \xi\}$. From the density of $\mathcal{Z}_C$ in $C$ it follows that that $\bigcap_{\tau \in Z}\qq_\tau = \qq_\chi$. Therefore we can choose a strictly decreasing sequence 
$\tau_0 > \tau_1 > \tau_2 > \cdots > \tau_\alpha > \cdots$
indexed by some limit ordinal $\lambda$ such that $\bigcap_{\alpha < \lambda}\qq_{\tau_\alpha} = \qq_\chi$, and whence $(ii)$ holds.

			If $(ii)$ holds, then $\chi$ belongs to some $C \in \mathcal{H}_n$ and $\chi$ is not maximal in $C$. Then there is a dense interval $\chi < \xi$ in $\mathcal{X}_n$, and so $(i)$ cannot hold.
	\end{proof}

	\begin{lem}\label{L:tensordesc}
		Let $M$ be a module from $\mathcal{C}_n$. Then $M \in \mathcal{D}_{n+1}$ if and only if the two following conditions hold:
	\begin{itemize}
		\item[(i)] for any gap $(\qq,\pp) \in \mathcal{G}_{n+1}$ we have $H^1(K(\qq,\pp) \otimes_R M) = 0$, and
		\item[(ii)] for any $C \in \mathcal{H}_{n+1}$ and any $[\pp,\qq] \in \mathcal{Z}_C$ we have $(R_{\qq}/\pp) \otimes_R M = 0$.
	\end{itemize}
	\end{lem}
	\begin{proof}
			Let $[\pp,\qq] \in \mathcal{X}_{n+1}$. We separate the two different cases as prescribed in Lemma~\ref{L:denscons}.
		
			\textbf{Case 1:} There is a decreasing sequence of gaps $((\qq_\alpha,\pp_\alpha) \mid \alpha < \lambda)$ in $\mathcal{G}_{n+1}$ such that $\bigcap_{\alpha < \lambda}\qq_\alpha = \qq$. Then by Lemma~\ref{L:Dn}, for each $\alpha < \lambda$ the condition $F_{\qq_\alpha}(M) \in \ModRqqalpha$ holds if and only if $H^1(K(\qq_\alpha,\pp_\alpha) \otimes_R M) = 0$. We can express $F_{\qq}(M)$ as the direct limit $\varinjlim_{\alpha < \lambda} F_{\qq_\alpha}(M)$ where the structure maps are the obvious projections. If $F_{\qq_\alpha} \in \ModRqqalpha$ for all $\alpha < \lambda$ then the direct limit expression clearly forces $F_{\qq}(M) \in \ModRqq$. 

			\textbf{Case 2:} There is $C \in \mathcal{H}_{n+1}$ and a strictly decreasing sequence $([\pp_\alpha,\qq_\alpha] \mid \alpha < \lambda)$ of intervals indexed by a limit ordinal $\lambda$ in $\mathcal{Z}_C$ such that $\bigcap_{\alpha < \lambda}\qq_\alpha = \qq$. Suppose that the condition $(ii)$ holds. Then $R_{\qq_\alpha}/\pp_\alpha \otimes_R M = 0$ for each $\alpha < \lambda$. Consequently also $R_{\qq_\alpha}/\pp_\alpha \otimes_R F_{\qq_\alpha}M = 0$, or equivalently, $F_{\qq_\alpha}M = \pp_\alpha F_{\qq_\alpha}M$. Writing again $F_{\qq} = \varinjlim_{\alpha < \lambda} F_{\qq_\alpha}(M)$, we see that $F_{\qq}(M) = \pp_\alpha F_{\qq}(M)$ for each $\alpha < \lambda$. Since we also have $\bigcap_{\alpha < \lambda}\pp_\alpha = \qq$, we conclude that $F_{\qq}(M)$ is $\qq$-divisible, and thus $F_{\qq}(M) \in \ModRqq$.

			What remains to be proved is that if $M \in \mathcal{C}_n \cap \mathcal{D}_{n+1}$ and $[\pp,\qq] \in \mathcal{Z}_C$ for some $C \in \mathcal{X}_{n+1}$ then $R_{\qq}/\pp \otimes_R M=0$. Since $[\pp,\qq] \in \mathcal{Z}_C$ there is a gap $(\qq',\pp') \in \mathcal{G}_n$ which strictly contains $[\pp,\qq]$, see Remark~\ref{R:gap}. Since $M \in \mathcal{C}_n$, the inclusions $\qq' \subsetneq \pp \subseteq \qq \subsetneq \pp'$ imply that (cf. Lemma~\ref{L:compute} and Proposition~\ref{P:tor})
$$\Gamma_{\qq}(M) \subseteq \Gamma_{\qq'}(M) = \Soc{}_{\pp'}(M) \subseteq \Gamma_{\qq}(M),$$
and thus $F_{\qq}(M) = F_{\qq'}(M)$. Then 
$$R_{\qq}/\pp \otimes_R M \simeq R_{\qq}/\pp \otimes_R F_{\qq}(M) = R_{\qq}/\pp \otimes_R F_{\qq'}(M).$$ 
Since $\qq' \subsetneq \pp$ and $C$ is dense, there is an interval $\chi \in C$ such that $\qq' \subseteq \qq_\chi \subsetneq \pp$. Then $M \in \mathcal{D}_{n+1}$ implies that the module $F_{\qq_\chi}(M)$ is $\qq_\chi$-divisible, and thus also its quotient $F_{\qq'}(M)$ is $\qq_\chi$-divisible. But then $F_{\qq'}(M) = \pp F_{\qq'}(M)$, and so $R_{\qq}/\pp \otimes_R M \simeq R_{\qq}/\pp \otimes_R F_{\qq'}(M) = 0$.

			Altogether, the three paragraphs above establish the desired equivalence.
	\end{proof}
	\begin{cor}\label{C:tensorvanish}
		Set $\mathcal{V}_n = \mathcal{C}_n \cap \mathcal{D}_{n+1}$ for all $n \in \mathbb{Z}$, then 
		$$\mathcal{C}_n = \bigcap_{(\qq,\pp) \in \mathcal{G}_n}\Ker H^0(K(\qq,\pp) \otimes_R -) \cap \bigcap_{[\pp,\qq] \in \mathcal{Z}_C, C \in \mathcal{H}_n}\Ker \Tor{}_1^R(R_{\qq}/\pp,-),$$
		and
		$$\mathcal{V}_n = \mathcal{C}_n \cap \bigcap_{(\qq,\pp) \in \mathcal{G}_{n+1}}\Ker H^1(K(\qq,\pp) \otimes_R -) \cap \bigcap_{[\pp,\qq] \in \mathcal{Z}_C, C \in \mathcal{H}_{n+1}}\Ker (R_{\qq}/\pp \otimes_R -).$$
	\end{cor}
	\begin{proof}
		Since $\mathcal{Z}_C$ is a dense subset of $C$ for any $C \in \mathcal{H}_n$, and because $\mathcal{G}_n = \mathcal{G}(\mathcal{X}_n) = \mathcal{G}(\bar{\mathcal{X}}_n)$ for each $n \in \mathbb{Z}$, Proposition~\ref{P:tor} implies
		$$\mathcal{C}_n = \bigcap_{(\qq,\pp) \in \mathcal{G}_n}\Ker H^0(K(\qq,\pp) \otimes_R -) \cap \bigcap_{[\pp,\qq] \in \mathcal{Z}_C, C \in \mathcal{H}_n}\Ker \Tor{}_1^R(R_{\qq}/\pp,-).$$
		The rest follows from Lemma~\ref{L:tensordesc}.
	\end{proof}
	\begin{lem}\label{L:seq}
		Let $0 \rightarrow X \rightarrow Y \rightarrow Z \rightarrow 0$ be an exact sequence. Then:
		\begin{enumerate}
		\item[(i)] If $Y \in \mathcal{C}_{n+1}$ and $X$ is an epimorphic image of some $V_n \in \mathcal{V}_n$ then $Z \in \mathcal{C}_{n+1}$.
		\item[(ii)] If $Y \in \mathcal{V}_n$ and $Z \in \mathcal{C}_{n+1}$ then $X \in \mathcal{V}_{n}$.
		\end{enumerate}
	\end{lem}
	\begin{proof}
		Throughout the proof, we will use the description of $\mathcal{V}_n$ of Corollary~\ref{C:tensorvanish}.

		Let $(\qq,\pp) \in \mathcal{G}_{n+1}$ with $\qq \neq -\infty$. Then $H^0(K(\qq,\pp) \otimes_R -) \simeq \Tor_1^R(R_{\qq}/\pp,-)$, and we have an exact sequence
			$$\Tor{}_1^R(R_{\qq}/\pp,Y) \rightarrow \Tor{}_1^R(R_{\qq}/\pp,Z) \rightarrow R_{\qq}/\pp \otimes_R X \rightarrow R_{\qq}/\pp \otimes_R Y.$$
			If $Y \in \mathcal{C}_{n+1}$ then the leftmost term vanishes, and $X$ being en epimorphic image of some object from $V_n$ ensures that $R_{\qq}/\pp \otimes_R X = 0$. The condition $(i)$ thus implies  $\Tor{}_1^R(R_{\qq}/\pp,Z) = 0$. In the situation of $(ii)$, the rightmost term is zero and $\Tor{}_1^R(R_{\qq}/\pp,Z)$ vanishes, and thus $R_{\qq}/\pp \otimes_R X = 0$.

		If $(-\infty,\pp)$ is a gap in $\mathcal{G}_{n+1}$, then $H^1(K(-\infty,\pp) \otimes_R -)$ is identically zero and $H^0(K(-\infty,\pp),Z) = \pp \otimes_R Z$ always vanishes because $Z$ is an epimorphic image of $Y \in \mathcal{C}_{n+1}$.
			
		Let $[\pp,\qq] \in \mathcal{Z}_C$ for some $C \in \mathcal{H}_{n+1}$. Then we have an exact sequence
			$$\Tor{}_1^R(R_{\qq}/\pp,Y) \rightarrow \Tor{}_1^R(R_{\qq}/\pp,Z) \rightarrow R_{\qq}/\pp \otimes_R X \rightarrow R_{\qq}/\pp \otimes_R Y.$$
		Similarly to the case above, the condition $(i)$ makes the first and third entry of the sequence from the left vanish, while condition $(ii)$ zeros out the second and forth term of the sequence, both time using Corollary~\ref{C:tensorvanish}.

		Putting the conditions together, in $(i)$ we have $H^0(K(\qq,\pp),Z) = 0$ for all gaps $(\qq,\pp) \in \mathcal{G}_n$ and $\Tor{}_1^R(R_{\qq}/\pp,Z) = 0$ for all $[\pp,\qq] \in \mathcal{Z}_C$ for all $C \in \mathcal{H}_{n+1}$, and thus $Z \in \mathcal{C}_{n+1}$ by Proposition~\ref{P:tor}. Under the assumptions of $(ii)$ we obtained that $H^1(K(\qq,\pp) \otimes_R X) = 0$ for all gaps $(\qq,\pp) \in \mathcal{G}_{n+1}$ and $R_{\qq}/\pp \otimes_R X = 0$ for all $[\pp,\qq] \in \mathcal{Z}_C$ and $C \in \mathcal{H}_{n+1}$. It follows that $X \in \mathcal{D}_{n+1}$. Indeed, since $X$ is a submodule of $Y$, $X$ belongs to $\mathcal{C}_n$, and therefore $X \in \mathcal{V}_n$ by Corollary~\ref{C:tensorvanish}.
		\end{proof}
		\begin{lem}\label{L:extensions}
		Let $\mathcal{V}_n, n \in \mathbb{Z}$ be the classes as above. If $f: V_n \rightarrow V_{n+1}$ is a map from $V_n \in \mathcal{V}_{n}$ to $V_{n+1} \in \mathcal{V}_{n+1}$ then $\Ker(f) \in \mathcal{V}_n$ and $\Coker(f) \in \mathcal{V}_{n+1}$.
	\end{lem}
\begin{proof}
	Consider the induced exact sequences
	$$0 \rightarrow K \rightarrow V_n \rightarrow I \rightarrow 0,$$
	and
	$$0 \rightarrow I \rightarrow V_{n+1} \rightarrow C \rightarrow 0,$$
	where $K = \Ker(f)$, $C = \Coker(f)$, and $I = \operatorname{Im}(f)$. Since $I \in \mathcal{C}_{n+1}$, Lemma~\ref{L:seq} shows that $K \in \mathcal{V}_{n}$. Because $I$ is an epimorphic image of a module $V_n \in \mathcal{V}_n$ and $V_{n+1} \in \mathcal{C}_{n+1}$, the same lemma establishes that $C \in \mathcal{C}_{n+1}$. Since $C$ is also an epimorphic image of $V_{n+1} \in \mathcal{D}_{n+2}$, also $C \in \mathcal{D}_{n+2}$, and therefore $C \in \mathcal{V}_{n+1}$.
\end{proof}
	\begin{lem}\label{L:contain}
		For each $n \in \mathbb{Z}$ we have an inclusion $\mathcal{C}_n \cap \mathcal{D}_n \subseteq \mathcal{D}_{n+1}$.
	\end{lem}
\begin{proof}
		Let $M \in \mathcal{C}_n \cap \mathcal{D}_n$ and consider an interval $[\pp,\qq] \in \mathcal{X}_{n+1}$. We distinguish the two cases provided by Lemma~\ref{L:denscons}. 

		One possibility is that there is a decreasing sequence $((\qq_\alpha,\pp_\alpha) \mid \alpha < \lambda)$ of gaps in $\mathcal{G}_{n+1}$ such that $\bigcap_{\alpha < \lambda}\qq_{\alpha} = \qq$. Since $\mathcal{X}_n$ is a nested subsystem of $\mathcal{X}_{n+1}$, for each $\alpha < \lambda$ there is a gap $(\qq'_\alpha,\pp'_\alpha) \in \mathcal{G}_n$ which contains the gap $(\qq_\alpha,\pp_\alpha)$ (meaning that $\qq'_\alpha \subseteq \qq_\alpha \subseteq \pp_\alpha \subseteq \pp'_\alpha$). 

		The only other possibility is that there is $C \in \mathcal{H}_{n+1}$ and a strictly decreasing sequence $([\pp_\alpha,\qq_\alpha] \mid \alpha < \lambda)$ of intervals in $\mathcal{Z}_C$ indexed by a limit ordinal $\lambda$ such that $\bigcap_{\alpha < \lambda}\qq_\alpha = \qq$. Again, in view of Remark~\ref{R:gap}, each $[\pp_\alpha,\qq_\alpha]$ is contained in a gap $(\qq'_\alpha,\pp'_\alpha) \in \mathcal{G}_n$. 

		In both cases, since $M \in \mathcal{C}_n$ and there is a gap $(\qq'_\alpha,\pp'_\alpha) \in \mathcal{G}_n$ with $\qq'_\alpha \subseteq \qq_\alpha \subsetneq \pp'_\alpha$, we have
		$$\Gamma_{\qq_\alpha}(M) \subseteq \Gamma_{\qq'_\alpha}(M) = \Soc{}_{\pp'_\alpha}(M) \subseteq \Gamma_{\qq_\alpha}(M),$$
and thus $F_{\qq_\alpha}(M) = F_{\qq'_\alpha}(M)$ for each $\alpha < \lambda$. Because $M \in \mathcal{D}_n$, $F_{\qq_\alpha}(M) \in \ModRqqalphaprime \subseteq \ModRqqalpha$ for each $\alpha < \lambda$. As a conclusion, we infer that $F_{\qq}(M) = \varinjlim_{\alpha < \lambda}F_{\qq_\alpha}(M) \in \ModRqq$.
	\end{proof}

\begin{prop}\label{P:intervalstocoaisles}
	Let $R$ be a valuation domain. Then there is an assignment
			$$\Xi: \left \{ \begin{tabular}{ccc} \text{ admissible filtrations $\mathbb{X}$} \\ \text{in $\Spec(R)$} \end{tabular}\right \}  \rightarrow  \left \{ \begin{tabular}{ccc} \text{ definable coaisles $\mathcal{V}$ } \\ \text{ in $\Der(R)$} \end{tabular}\right \},$$
			defined by setting
			$$\Xi(\mathbb{X}) = \{X \in \Der(R) \mid H^n(X) \in \mathcal{C}_n \cap \mathcal{D}_{n+1} ~\forall n \in \mathbb{Z}\}.$$
	\end{prop}
	\begin{proof}
			Denote $\mathcal{V} = \Xi(\mathbb{X})$. It is enough to check that the classes $\mathcal{V}_n = \mathcal{C}_n \cap \mathcal{D}_{n+1}$ satisfy the conditions of Proposition~\ref{P:moduletheoretic}. By Corollary~\ref{C:tensorvanish}, it is clear that $\mathcal{V}_n$ is closed under direct limits, pure submodules, and extensions for each $n \in \mathbb{Z}$. For the rest of the proof, we fix $n \in \mathbb{Z}$ and prove all of the other conditions.
			
			First we check that $\mathcal{V}_n \subseteq \mathcal{V}_{n+1}$. By Lemma~\ref{L:contain}, we have $\mathcal{V}_n = \mathcal{C}_n \cap \mathcal{D}_{n+1} \subseteq \mathcal{C}_{n+1} \cap \mathcal{D}_{n+1} \subseteq \mathcal{D}_{n+2}$. Because the inclusion $\mathcal{C}_n \subseteq \mathcal{C}_{n+1}$ is clear from the description in Lemma~\ref{L01}, this step is established.
	
			If $f: V_n \rightarrow V_{n+1}$ is a map with $V_i \in \mathcal{V}_{i}$ for $i=n,n+1$, then $\Ker(f) \in \mathcal{V}_n$ and $\Coker(f) \in \mathcal{V}_{n+1}$ by Lemma~\ref{L:extensions}
		
			Finally, we need to show that $\mathcal{V}_n$ is closed under direct products. Let $(M_i \mid i \in I)$ be a sequence of modules from $\mathcal{V}_n$. Since $\mathcal{C}_n$ is a cosilting class, clearly $\prod_{i \in I}M_i \in \mathcal{C}_n$. We need to check that $\prod_{i \in I} M_i \in \mathcal{D}_{n+1}$. Let $[\pp,\qq] \in \mathcal{X}_{n+1}$. We again separate the two cases given by Lemma~\ref{L:denscons}. 

The first possibility is that there is a decreasing sequence $((\qq_\alpha,\pp_\alpha) \mid \alpha < \lambda)$ of gaps from $\mathcal{G}_{n+1}$ with $\bigcap_{\alpha < \lambda} \qq_\alpha = \qq$. Since $\mathcal{X}_n$ is a nested subsystem of $\mathcal{X}_{n+1}$, each gap $(\qq_\alpha,\pp_\alpha)$ is contained in some gap from $\mathcal{G}_n$. The other possibility is by Lemma~\ref{L:denscons} the existence of $C \in \mathcal{H}_{n+1}$ and a strictly decreasing sequence $([\pp'_\alpha,\qq_\alpha] \mid \alpha < \lambda)$ of intervals (note the change of notation here) in $\mathcal{Z}_C$ indexed by a limit ordinal $\lambda$ such that $\bigcap_{\alpha < \lambda}\qq_\alpha = \qq$. Again, in view of Remark~\ref{R:gap}, each $[\pp'_\alpha,\qq_\alpha]$ is strictly contained in a gap $(\qq'_\alpha,\pp_\alpha) \in \mathcal{G}_n$. 

In both cases, $M_i \in \mathcal{D}_{n+1}$ implies $F_{\qq_\alpha}(M_i) \in \ModRqqalpha$ for each $i \in I$. Also, because $M_i \in \mathcal{C}_n$ for each $i \in I$, the gaps in $\mathcal{G}_n$ obtained in the preceding paragraph with $\qq_\alpha \subsetneq \pp_\alpha$ yield in view of Lemma~\ref{L:compute} that $\Gamma_{\qq_\alpha}(M_i) = \Soc_{\pp_\alpha}(M_i)$ for all $i \in I$ and $\alpha < \lambda$. Therefore, we can compute as follows:
$$\prod_{i \in I}\Gamma_{\qq_\alpha}(M_i) = \prod_{i \in I}\Soc{}_{\pp_\alpha}(M_i) = \Soc{}_{\pp_\alpha}(\prod_{i \in I}M_i) \subseteq \Gamma_{\qq_\alpha}(\prod_{i \in I}M_i) \subseteq \prod_{i \in I}\Gamma_{\qq_\alpha}(M_i).$$
From this we infer that $F_{\qq_\alpha}(\prod_{i \in I}M_i) \simeq \prod_{i \in I}F_{\qq_\alpha}(M_i) \in \ModRqqalpha$. Since $\bigcap_{\alpha < \lambda}\qq_\alpha = \qq$, we finally conclude that $F_{\qq}(\prod_{i \in I}M_i) \simeq \varinjlim_{\alpha < \lambda} F_{\qq_\alpha}(\prod_{i \in I}M_i) \in \ModRqq$.
				\end{proof}

			Finally, let us check that this construction is well-behaved with respect to the admissible filtration constructed in Section~\ref{S:coaisletoint}.
	\begin{prop}\label{P:sameint}
		The composition $\Theta \circ \Xi$ of the assignments defined in Proposition~\ref{P:coaislestointervals} and Proposition~\ref{P:intervalstocoaisles} is the identity on the set of all admissible filtrations in $\Spec(R)$.
	\end{prop}
	\begin{proof}
		Let $\mathbb{X} = (\mathcal{X}_n \mid n \in \mathbb{Z})$ be an admissible filtration, and let $\Theta(\Xi(\mathbb{X})) = (\mathcal{X}'_n \mid n \in \mathbb{Z})$ be the admissible filtration associated to the definable coaisle $\mathcal{V} = \Xi(\mathbb{X})$. Because $\mathcal{V}_n \subseteq \mathcal{C}_n$, where $\mathcal{C}_n$ is the cosilting class corresponding to $\mathcal{X}_n$ via Theorem~\ref{T01}, we clearly have that $\mathcal{X}'_n$ is a nested subsystem of $\mathcal{X}_n$ for each $n \in \mathbb{Z}$. It is enough to show that for each $n \in \mathbb{Z}$ and each interval $[\pp,\qq] \in \mathcal{X}_n$, the module $R_{\pp}/\qq$ belongs to $\mathcal{D}_{n+1}$. Indeed, by the construction of $\mathcal{C}_n = \mathcal{C}_{\mathcal{X}_n}$ from Lemma~\ref{L01}, the module $R_{\pp}/\qq$ belongs to $\mathcal{C}_n$ for any $[\pp,\qq] \in \mathcal{X}_n$. If $R_{\pp}/\qq$ belongs also to $\mathcal{D}_{n+1}$, it belongs by definition to $\mathcal{V}_n$, and thus $[\pp,\qq]$ is contained in some interval from $\mathcal{X}'_n$ by Lemma~\ref{L:mc}. As $\mathcal{X}'_n$ is a nested subsystem of $\mathcal{X}_n$, this means that $[\pp,\qq] \in \mathcal{X}'_n$.

Let $[\pp',\qq'] \in \mathcal{X}_{n+1}$. Then either $[\pp',\qq'] < [\pp,\qq]$, and then $R_{\pp}/\qq = \Gamma_{\qq'}(R_{\pp}/\qq)$, or $[\pp,\qq] < [\pp',\qq']$, and then $R_{\pp}/\qq$ is already an $R_{\qq'}$-module, or finally $[\pp',\qq']$ contains $[\pp,\qq]$, in which case again  $R_{\pp}/\qq$ is already an $R_{\qq'}$-module. In all of the cases, $F_{\qq'}(R_{\pp}/\qq) \in \ModRqqprime$, showing that $R_{\pp}/\qq \in \mathcal{D}_{n+1}$.
	\end{proof}
	Corollary~\ref{C:tensorvanish} suggests that the constructed definable coaisles are given in the derived category as orthogonal classes with respect to the derived tensor product. This is indeed the case, and should be seen as the correct generalization of the module theoretic case of Proposition~\ref{P:tor} (cf. \cite[Theorem 6.11]{B2} and also \cite[Proposition 5.10]{HCG} for a similar type of description in the case of compactly generated t-structures).
	\begin{prop}\label{P:givenbytensor}
		Let $R$ be a valuation domain and $\mathbb{X} = (\mathcal{X}_n \mid n \in \mathbb{Z})$ an admissible filtration in $\Spec(R)$. Let us define the following subset of $\Der(R)$:
			$$\mathcal{S}_{\mathbb{X}} = \{\enspace K(\qq,\pp)[n] \enspace \mid n \in \mathbb{Z}, (\qq,\pp) \in \mathcal{G}_n\} \enspace\cup $$
			$$\cup\enspace \{\enspace R_{\qq}/\pp[n-1] \enspace\mid n \in \mathbb{Z}, C \in \mathcal{H}_n, [\pp,\qq] \in \mathcal{Z}_C\}.$$
		Then the definable coaisle $\mathcal{V} = \Xi(\mathbb{X})$ constructed from $\mathbb{X}$ by Proposition~\ref{P:intervalstocoaisles} is tensor-semi-orthogonal to the set $\mathcal{S}_\mathbb{X}$ in the following sense:
			$$\mathcal{V} = \{X \in \Der(R) \mid S \otimes_R^\mathbf{L} X \in \Der{}^{>0} ~\forall S \in \mathcal{S}_{\mathbb{X}}\}.$$
	\end{prop}
	\begin{proof}
		Let us denote $\mathcal{C} = \{X \in \Der(R) \mid S \otimes_R^\mathbf{L} X \in \Der{}^{>0} ~\forall S \in \mathcal{S}_{\mathbb{X}}\}$ and prove that $\mathcal{V} = \mathcal{C}$. By Proposition~\ref{P:moduletheoretic}, the definable coaisle $\mathcal{V}$ is determined on cohomology. By an application K\"{u}nneth formula as in the proof of \cite[Proposition 3.6]{BS}, also the class $\mathcal{C}$ is determined on cohomology. It is therefore enough to check that $M[-n] \in \mathcal{V}$ if and only if $M[-n] \in \mathcal{C}$ for any $n \in \mathbb{Z}$ and any $R$-module $M$. Using that $R$ is of weak global dimension at most one, we observe that $M[-n] \in \mathcal{C}$ if and only if 
	\begin{enumerate}
		\item[(i)] $H^0(K(\qq,\pp) \otimes_R M) = 0 = \Tor_1^R(R_{\qq'}/\pp', M)$ for each $(\qq,\pp) \in \mathcal{G}_n$ and $[\pp',\qq'] \in \mathcal{Z}_C, ~\forall C \in \mathcal{H}_n$, and
		\item[(ii)] $K(\qq,\pp) \otimes_R M$ and $R_{\qq'}/\pp' \otimes_R^{\mathbf{L}} M$ are zero in $\Der(R)$ for each $(\qq,\pp) \in \mathcal{G}_k$ and $[\pp',\qq'] \in \mathcal{Z}_C, ~\forall C \in \mathcal{H}_k$ for any $k>n$.
	\end{enumerate}
	By Proposition~\ref{P:tor}, condition $(i)$ is equivalent to $M \in \mathcal{C}_n$. The condition $(ii)$ says equivalently that the $0$-th and $1$-th cohomologies of the complexes $K(\qq,\pp) \otimes_R M$ and $R_{\qq}/\pp \otimes_R M$ vanish for all the prescribed indexing choices. The vanishing of $0$-th cohomology again translates as $M \in \mathcal{C}_k \supseteq \mathcal{C}_n$, and therefore is vacuous. The vanishing of the second cohomology is in view of Lemma~\ref{L:tensordesc} equivalent to $M \in \mathcal{D}_{k}$ for all $k>n$. By Lemma~\ref{L:contain}, this is equivalent to $M \in \mathcal{C}_n \cap \mathcal{D}_{n+1} = \mathcal{V}_n$, as desired.
	\end{proof}
	We finish this section by an example of a definable coaisle constructed from an admissible filtration whose admissible systems are not all nowhere dense, to illustrate the degreewise non-density condition. Note that the resulting coaisle is co-intermediate, and thus corresponds to an equivalence class of bounded cosilting complexes via Theorem~\ref{T:MV}.
	\begin{example}\label{EX0}
			The following example comes by adjusting \cite[Example 5.1]{B2}. Let $R$ be a valuation domain with $(\Spec(R),\subseteq)$ order isomorphic to the set $P=[0,1] \times \{0,1\}$ equipped with the lexicographic order (here $[0,1]$ denotes the closed real interval), and such that all primes from $\Spec(R)$ are idempotent. Such a valuation domain exists --- there is a valuation domain $R$ with $\Spec(R)$ order isomorphic to $P$ by \cite[\S II, Theorem 2.5 and Proposition 4.7]{FS}, and it can be constructed in a such a way that all primes are idempotent by \cite[\S II, Proposition 5.7 and the following paragraph]{FS}. Let $\pp_x$ (resp. $\qq_x$) be the prime of $\Spec(R)$ corresponding to the element $[x,0]$ (resp. $[x,1]$) of $P$. 

Let $Z$ be a nowhere dense closed subset of $[0,1]$. Then we define an admissible filtration $\mathbb{X} = (\mathcal{X}_n \mid n \in \mathbb{Z})$ on $\Spec(R)$ as follows:
			$$
			\mathcal{X}_n = \begin{cases}
						\emptyset, & n<0, \\
						\{[\pp_x,\pp_x],[\qq_x,\qq_x] \mid x \in Z\}, & n=0, \\
						\{[\pp_x,\qq_x] \mid x \in [0,1]\}, & n=1, \\
						\{[0,\mm]\}, & n>1.
			\end{cases}
			$$
			Note that the sequence $(\mathcal{X}_n \mid n \in \mathbb{Z})$ is clearly nested. Since $Z$ is closed, $\mathcal{X}_0$ satisfies the completeness condition of Definition~\ref{D:admissiblesystem}, and therefore forms an admissible system. The set $\mathcal{X}_1$ is easily checked to form a dense everywhere admissible system (cf. \cite[Example 5.1]{B2}). The degreewise non-density condition follows directly from $Z$ being a nowhere dense subset in $[0,1]$.
	\end{example}
	\section{Bijective correspondence}
	Now it is time to finally establish that, working over any valuation domain $R$, the two sections \ref{S:coaisletoint} and \ref{S:inttocoaisle} provide two mutually inverse assignments for the set of all definable coaisles in $\Der(R)$ and the set of all admissible filtrations in $\Spec(R)$.

We already know that the classes $\mathcal{V}_n$ coming from a definable coaisle $\mathcal{V}$ are uniquely determined by the standard uniserial modules they contain --- this is Lemma~\ref{L:standarduni}. The next Lemma shows that in this situation, these uniserial modules are determined by the associated admissible filtration.
	\begin{lem}\label{L:unicores}
			Let $R$ be a valuation domain, and let $\mathcal{V}$ be a definable coaisle in $\Der(R)$. Let $\mathbb{X} = (\mathcal{X}_n \mid n \in \mathbb{Z})$ be the admissible filtration associated to $\mathcal{V}$ by Definition~\ref{D:phipsi}. Suppose that $J/I \in \mathcal{V}_n$ for some $R$-submodules $I \subseteq J \subseteq Q$ of the quotient field and some integer $n$. Then there are intervals $[\pp,\qq] \in \mathcal{X}_n$ and $[\pp',\qq'] \in \mathcal{X}_{n+1}$ such that $J/I$ admits a coresolution of the form
			$$0 \rightarrow J/I \rightarrow M \rightarrow N \rightarrow 0,$$
			where $M$ is an $R_{\qq}/\pp$-module and $N$ is an $R_{\qq'}/\pp'$-module.
	\end{lem}
	\begin{proof}
			Since the case $J=I$ is trivial, we can assume $I \subsetneq J$. By multiplying by a scalar, we can also assume that $I \subsetneq R \subseteq J$. Since $J_{I^\#}/I \simeq J/I \otimes_R R_{I^\#}$, we have $J_{I^\#}/I \in \mathcal{V}_n$. For each $r \in R \setminus I$, we have $r^{-1}J_{I^\#}/r^{-1}I \simeq J_{I^\#}/I \in \mathcal{V}_n$, and therefore also $\varinjlim_{r \in R \setminus I}r^{-1}J_{I^\#}/r^{-1}I \simeq J'/I^\# \in \mathcal{V}_n$, where $J' = \bigcup_{r \in R \setminus I}r^{-1}J_{I^\#}$. Because $R_{I^\#} \subseteq J_{I^\#} \subseteq J'$, we have a natural inclusion $\kappa(I^\#) \subseteq J'/I^\#$, and therefore $\kappa(I^\#) \in \mathcal{V}_n$ by Lemma~\ref{L:RHom}. In other words, $I^\# \in \mathcal{K}_n$. Let $\pp = \varphi_n(I^\#)$ and $\qq = \psi_n(I^\#)$, so that $[\pp,\qq] \in \mathcal{X}_n$. 

			For all $r \in R \setminus I$, consider the map $f_r: J/I \xrightarrow{\cdot r} J/I$ given by multiplication by $r$. Then $\Ker(f_r) = r^{-1}I/I \in \mathcal{V}_n$. Taking the directed union, we see that $I^\#/I = \bigcup_{r \in R \setminus I}r^{-1}I/I \in \mathcal{V}_n$. As $\pp = \varphi_n(I^\#)$, we have $R_{\pp}/I^\# \in \mathcal{V}_n$ by Lemma~\ref{L:belong}. Since also $I^\#/I \in \mathcal{V}_n$, we have that $R_{\pp}/I \in \mathcal{V}_n$.

			Denote $K=\Ann_R(J/I)$, and let us show that $\pp \subseteq K$. Towards a contradiction, suppose that there is $t \in \pp \setminus K$. Because $t \not\in K$, necessarily $t^{-1}I \subseteq J$. Then $\Ker(J/I \xrightarrow{\cdot t} J/I) = t^{-1}I/I \in \mathcal{V}_n$. Since $\mathcal{V}_n$ is closed under extensions, also $t^{-k}I/I \simeq I/t^kI \in \mathcal{V}_n$ for all $k > 0$. It follows that $R_{\pp}/t^kI \simeq t^{-k}R_{\pp}/I \in \mathcal{V}_n$ for all $k>0$, and therefore by passing to the direct limit over $k>0$, $R_{\oo_t}/I \in \mathcal{V}_n$ for a prime ideal $\oo_t \subsetneq tR \subseteq \pp$. Doing this for all $t \in \pp \setminus K$, and taking the direct limit, we can see that $R_{\oo}/I \in \mathcal{V}$, where $\oo = \bigcap_{t \in \pp \setminus K}\oo_t$. Then $\oo \subseteq K \subseteq I$, and thus $R_{\oo}/I^\# \simeq \varinjlim_{r \in R \setminus I}R_{\oo}/r^{-1}I \in \mathcal{V}_n$, a contradiction with the definition of $\pp = \varphi_n(I^\#)$. Therefore, indeed $\pp \subseteq K$.

			We set $M = J_{\qq}/I$. As $M = J/I \otimes_R R_{\qq}$, we have $M \in \mathcal{V}_n$. Since $I^\# \subseteq \qq$, $M$ is an $R_{\qq}$-module. Observe that $\Ann_R(M) = K$, and thus by the previous paragraph $M$ is an $R_{\qq}/\pp$-module. Denote $N$ the cokernel of the natural inclusion $J/I \subseteq J_{\qq}/I$, that is, $N \simeq J_{\qq}/J$. Let $\pp' = \varphi_{n+1}(I^\#)$, and $\qq' = \psi_{n+1}(I^\#)$. Since $[\pp,\qq] \in \mathcal{X}_n$, we know that $\pp' \subseteq \pp$ and $\qq \subseteq \qq'$. We have $U/J^\# \simeq \varinjlim_{q \in Q \setminus J} (q^{-1}J_{\qq}/q^{-1}J) \in \mathcal{V}_{n+1}$, where $U = \bigcup_{q \in Q \setminus J}q^{-1}J_{\qq}$ ($U=Q$ in the case $J=Q$). We want to show that $J^\# \subseteq \qq'$. Towards a contradiction, assume $\qq' \subsetneq J^\#$. There are two cases. Either $R_{\qq} \subseteq U$, then consider the exact sequence:
			$$0 \rightarrow R_{\qq}/J^\# \rightarrow U/J^\# \rightarrow U/R_{\qq} \rightarrow 0.$$
			The middle term is in $\mathcal{V}_{n+1}$, and $U/R_{\qq}$ is an $R_{\qq}$-module. Since $U \subseteq J_{\qq}$, we have $\Ann_{R}(U/R_{\qq}) \supseteq \Ann_R(J_{\qq}/R_{\qq}) = \Ann_R(J/R)_{\qq} \supseteq \Ann_R(J/I)_{\qq} = K_{\qq} \supseteq K \supseteq \pp$. Therefore, $U/R_{\qq}$ is an $R_{\qq}/\pp$-module, and whence $U/R_{\qq} \in \mathcal{V}_{n+1}$. It follows that $R_{\qq}/J^\# \in \mathcal{V}_{n+1}$. The other case is $U \subsetneq R_{\qq}$, here we consider the exact sequence:
			$$0 \rightarrow U/J^\# \rightarrow R_{\qq}/J^\# \rightarrow R_{\qq}/U \rightarrow 0.$$
			We know that $U/J^\# \in \mathcal{V}_{n+1}$, and that $R_{\qq}/U$ is an $R_{\qq}$-module. We know that $J^\# \subseteq U$, and by the assumption, $\qq \subseteq \qq' \subseteq J^\#$. Therefore, $\pp \subseteq \qq' \subseteq J^\# \subseteq U = \Ann_{R_{\qq}}(R_{\qq}/U)$, and therefore $R_{\qq}/U$ is an $R_{\qq}/\pp$-module, and thus $R_{\qq}/U \in \mathcal{V}_{n+1}$ by Lemma~\ref{L:mc}(i). It follows again that $R_{\qq}/J^\# \in \mathcal{V}_{n+1}$.

			We showed that $R_{\qq}/J^\# \in \mathcal{V}_{n+1}$, which is a contradiction with $\qq' \subsetneq J^\#$, since $\psi_{n+1}(\qq) = \psi_{n+1}(I^\#) = \qq'$. Finally, note that $\Ann_{R}(J_{\qq}/J) = R_{J^\#}\qq \supseteq \qq \supseteq \pp \supseteq \pp'$. Because we already proved that $\qq' \supseteq J^\#$, we see that $N=J_{\qq}/J$ is an $R_{\qq'}/\pp'$-module.
			\end{proof}
			\begin{cor}\label{C:determined}
				Let $R$ be a valuation domain and let $\mathcal{V}$, $\mathcal{V}'$ be two definable coaisles in $\Der(R)$. Then $\Theta(\mathcal{V}) = \Theta(\mathcal{V}')$ implies $\mathcal{V} = \mathcal{V}'$.
			\end{cor}
			\begin{proof}
					Denote the admissible filtration by $\Theta(\mathcal{V}) = (\mathcal{X}_n \mid n \in \mathbb{Z})$. By Proposition~\ref{P:moduletheoretic}, both $\mathcal{V}$ and $\mathcal{V}'$ are determined by the cohomological projections $\mathcal{V}_n = H^n(\mathcal{V})$ and $\mathcal{V}'_n = H^n(\mathcal{V}')$ for all $n \in \mathbb{Z}$, respectively. By Lemma~\ref{L:standarduni}, the classes $\mathcal{V}_n$ and $\mathcal{V}'_n$ are fully determined by the standard uniserial modules of the form $J/I$ they contain, where $I \subseteq J \subseteq Q$. For any $n \in \mathbb{Z}$ and any standard uniserial module $J/I \in \mathcal{V}_n$, we have by Lemma~\ref{L:unicores}, that there are intervals $[\pp,\qq] \in \mathcal{X}_n$ and $[\pp',\qq'] \in \mathcal{X}_{n+1}$, and a coresolution 
					$$0 \rightarrow J/I \rightarrow M \rightarrow N \rightarrow 0,$$
					such that $M \in \ModRqqpp$, and $N \in \ModRqqprimeppprime$. Using Lemma~\ref{L:mc}(i), and the assumption $\Theta(\mathcal{V}') = \Theta(\mathcal{V})$, we see that $M \in \mathcal{V}'_n$ and $N \in \mathcal{V}'_{n+1}$. As $J/I$ is the kernel of a map $M \rightarrow N$, we infer using Proposition~\ref{P:moduletheoretic} that $J/I$ belongs to $\mathcal{V}'_n$. A symmetric argument shows that any standard uniserial module from $\mathcal{V}'_n$ belongs to $\mathcal{V}_n$ for all $n \in \mathbb{Z}$. We conclude that $\mathcal{V}_n = \mathcal{V}'_n$ for all $n \in \mathbb{Z}$, and therefore $\mathcal{V} = \mathcal{V}'$.
			\end{proof}
			\begin{thm}\label{T02}
			Let $R$ be a valuation domain. Then there is a bijective correspondence
					$$\left \{ \begin{tabular}{ccc} \text{ admissible filtrations $\mathbb{X}$} \\ \text{in $\Spec(R)$} \end{tabular}\right \}  \leftrightarrow  \left \{ \begin{tabular}{ccc} \text{ definable coaisles $\mathcal{V}$ } \\ \text{ in $\Der(R)$} \end{tabular}\right \}$$
							induced by the mutually inverse assignments $\Xi$ and $\Theta$ from Proposition~\ref{P:intervalstocoaisles} and Proposition~\ref{P:coaislestointervals}.
			\end{thm}
			\begin{proof}
					By Proposition~\ref{P:intervalstocoaisles} and Proposition~\ref{P:coaislestointervals}, both $\Xi$ and $\Theta$ are well-defined. Furthermore, by Proposition~\ref{P:sameint} and Corollary~\ref{C:determined}, these assignments are mutually inverse.
			\end{proof}
			In view of Remark~\ref{R:spec}, the classification of smashing subcategories \cite[Theorem 5.23]{BS} and of cosilting modules Theorem~\ref{T01} are special cases of Theorem~\ref{T02}. Also we get the following classification of bounded cosilting complexes as another consequence. Let us call an admissible filtration $\mathbb{X} = (\mathcal{X}_n \mid n \in \mathbb{Z})$ \EMP{bounded} provided that there are integers $m < l$ such that $\mathcal{X}_m = \emptyset$ and $\mathcal{X}_l = \{[0,\mm]\}$, where $\mm$ is the maximal ideal.
				\begin{thm}\label{T:cosilt}
			Let $R$ be a valuation domain. Then there is a bijective correspondence
						$$\left \{ \begin{tabular}{ccc} \text{ bounded admissible filtrations $\mathbb{X}$} \\ \text{in $\Spec(R)$} \end{tabular}\right \}  \leftrightarrow  \left \{ \begin{tabular}{ccc} \text{ bounded cosilting complexes } \\ \text{ in $\Der(R)$ up to equivalence} \end{tabular}\right \}.$$
			\end{thm}
			\begin{proof}
					By Theorem~\ref{T:MV}, a t-structure $(\mathcal{U},\mathcal{V})$ in $\Der(R)$ is induced by a bounded cosilting complex if and only if $\mathcal{V}$ co-intermediate and definable. Also, recall that the equivalence of two cosilting complexes amount precisely to them inducing the same t-structure. Finally, it is easy to see that a t-structure corresponding to an admissible filtration $\mathbb{X}$ is co-intermediate precisely when $\mathbb{X}$ is bounded. In this way, the correspondence is given by restricting the correspondence from Theorem~\ref{T02}.
			\end{proof}
			As we will demonstrate in the last section of the paper, there exist pure-injective cosilting complexes over valuation domains which are not bounded.
	\subsection{Compactly generated t-structures}
	The compactly generated t-structures over an arbitrary commutative ring were classified in \cite{HCG}, generalizing the result for noetherian rings in \cite{AJS}, in terms of decreasing sequences of Thomason subsets of the Zariski spectrum of the ring, see \cite[Theorem 5.6]{HCG}. Recall that a subset $X$ of $\Spec(R)$ is \emph{Thomason} if it is an arbitrary union of Zariski closed sets $V(I)$ with $I$ finitely generated. When $R$ is a valuation domain, then the Thomason sets are precisely the sets of the form $X = \Spec(R)$ or $X_{\qq} = \{\pp \in \Spec(R) \mid \qq \subsetneq \pp\}$, where $\qq$ is any prime ideal. Indeed, if $X \neq \Spec(R)$ then $X = X_{\qq}$, where $\qq$ is the greatest element of $(\Spec(R) \setminus X, \subseteq)$. Conversely, for any $\qq \in \Spec(R)$ we have that $X_{\qq} = \bigcup_{r \in R \setminus \qq}V(rR)$ is a Thomason set. In other words, Thomason sets over valuation domains correspond to saturated multiplicative sets of elements. In the following we make explicit the way this result translates for valuation domains in terms of admissible filtrations.
	\begin{prop}\label{P:cg}
		Let $R$ be a valuation domain and $\mathbb{X} = (\mathcal{X}_n)_{n \in \mathbb{Z}}$ an admissible filtration in $\Spec(R)$. Let $\mathcal{V} = \Xi(\mathbb{X})$ be the definable coaisle corresponding to $\mathbb{X}$ via Theorem~\ref{T02} and let $(\mathcal{U},\mathcal{V})$ be the induced t-structure. Then the following conditions are equivalent:
	\begin{enumerate}
		\item[(i)] the t-structure $(\mathcal{U},\mathcal{V})$ is compactly generated,
		\item[(ii)] for each $n \in \mathbb{Z}$, the admissible system $\mathcal{X}_n$ is either empty or it is a singleton of the form $\mathcal{X}_n = \{[0,\qq_n]\}$ for some prime ideal $\qq_n$.
	\end{enumerate}
	\end{prop}
	\begin{proof}
		First, assume $(ii)$. Together with the definition of an admissible filtration, this means that there is a lower bound $N \in \mathbb{Z} \cup \{-\infty\}$ such that
		\begin{equation}\label{E:Xn}\mathcal{X}_n = \begin{cases} \emptyset, & n < N \\ \{[0,\qq_n]\}, & n \geq N.\end{cases}.\end{equation}
		As a consequence, $\mathcal{V}_n = \{0\}$ whenever $n<N$ and $\mathcal{V}_n$ consists precisely of those $R$-modules which are both $\qq_n$-torsion-free and $\qq_{n+1}$-divisible for any $n \geq N$. Therefore, $\mathcal{V}_n$ is closed under injective envelopes for all $n \in \mathbb{Z}$. By Proposition~\ref{P:compinj}, the t-structure $(\mathcal{U},\mathcal{V})$ is compactly generated.

		For the converse implication, we use the classification from \cite{HCG}. Let $\Phi$ be the Thomason filtration $\Phi$ on $\Spec(R)$ corresponding to $\mathcal{U}$ via \cite[Theorem 5.6]{HCG}. By the definition of a Thomason filtration \cite[\S 3]{HCG} and the discussion above, there is a bound $N \in \mathbb{Z} \cup \{-\infty\}$ and prime ideals $\qq_n, n \geq N$ such that
	$$\Phi(n) = \begin{cases} \emptyset, & n < N \\ X_{\qq_n} = \{\pp \in \Spec(R) \mid \qq_n \subsetneq \pp\}, & n \geq N.\end{cases}.$$
	Since $\Phi(n) \supseteq \Phi(n+1)$ by the definition of Thomason filtration, we have $\qq_n \subseteq \qq_{n+1}$ for each $n \in \mathbb{Z}$. Therefore, the formula (\ref{E:Xn}) defines an admissible filtration $\mathbb{X} = (\mathcal{X}_n)_{n \in \mathbb{Z}}$. We claim that $\mathcal{V} = \Xi(\mathbb{X})$. 

	By the way the correspondence \cite[Theorem 5.6]{HCG} works, we have the following description of the aisle:
	$$\mathcal{V} = \{X \in \Der(R) \mid \RHom{}_R(R/sR, X) \in \Der{}^{>n} ~\forall s \in R, V(sR) \subseteq \Phi(n) ~\forall n \in \mathbb{Z}\}.$$
	By Proposition~\ref{P:definablecorr}, it is enough to check that the definable subcategories $\mathcal{V}$ and $\Xi(\mathbb{X})$ coincide on cohomology. But clearly, an $R$-module $M$ belongs to $\mathcal{V}_n$ if and only if $\Hom_R(R/sR,M) = 0 = \Ext^1_R(R/tR,M)$ for each $s \not\in \qq_n$ and $t \not\in \qq_{n+1}$, which by a standard computation amounts to $M$ being $\qq_n$-torsion-free and $\qq_{n+1}$-divisible.
	\end{proof}
	\subsection{Application to Question~\ref{Q:UTC}}
	A valuation domain $R$ is called \EMP{strongly discrete} if the only idempotent ideal of $R$ is zero. The following results should be compared with the case of smashing subcategories \cite[Theorem 7.2]{BS} and 1-cotilting modules \cite[Corollary 4.6]{B}. 
\begin{cor}\label{C:tcvd}
	Let $R$ be a valuation domain. Then the following conditions are equivalent:
\begin{enumerate}
	\item[(i)] $R$ is strongly discrete,
	\item[(ii)] any t-structure on $\Der(R)$ with a definable coaisle is compactly generated.
\end{enumerate}
\end{cor}
\begin{proof}
	If $R$ is strongly discrete then since $0$ is the only idempotent prime ideal of $R$, an admissible system $\mathcal{X}$ in $\Spec(R)$ can only be either empty or of the form $\mathcal{X}=\{[0,\qq]\}$ for some prime ideal $\qq$. Then any definable coaisle belongs to a compactly generated t-structure by Theorem~\ref{T02} together with Proposition~\ref{P:cg}.

		The converse implication follows from \cite[Theorem 7.2]{BS}.
\end{proof}
Recall that valuation domains are precisely the local commutative rings of weak global dimension at most one. The next natural step is therefore to establish a global version of Corollary~\ref{C:tcvd}.
\begin{lem}\label{L:local}
		Let $R$ be a commutative ring and $(\mathcal{U},\mathcal{V})$ a t-structure such that $\mathcal{V}$ is definable. For any prime $\pp \in \Spec(R)$, define subcategories
		$$\mathcal{U}_{\pp}=\{X \otimes_R R_{\pp} \mid X \in \mathcal{U}\}\text{, and}$$
		$$\mathcal{V}_{\pp}=\{X \otimes_R R_{\pp} \mid X \in \mathcal{V}\}$$
		of $\Der(R)$. Then $(\mathcal{U}_{\pp},\mathcal{V}_{\pp})$ is a t-structure in $\Der(R_{\pp})$, $\mathcal{V}_{\pp}$ is definable in $\Der(R_{\pp})$, and we have the inclusions $\mathcal{U}_{\pp} \subseteq \mathcal{U}$ and $\mathcal{V}_{\pp} \subseteq \mathcal{V}$.
\end{lem}
\begin{proof}
		First, recall that $\mathcal{V}$ is closed under directed homotopy colimits, and $\mathcal{U}$ is closed under (any) homotopy colimits by \cite[Proposition 4.2]{SSV}. Since $R_{\pp}$ is a flat $R$-module, $X \otimes_R R_{\pp} \in \mathcal{V}$ for any $X \in \mathcal{V}$, and the analogous statement holds for the aisles. Therefore, $\mathcal{U}_{\pp} \subseteq \mathcal{U}$ and $\mathcal{V}_{\pp} \subseteq \mathcal{V}$.

		It is clear that for any $U \in \mathcal{U}_{\pp}$ and any $V \in \mathcal{V}_{\pp}$, 
		$$\Hom{}_{\Der(R{}_{\pp})}(U,V) \simeq \Hom{}_{\Der(R)}(U,V) = 0,$$ 
		and that $\mathcal{U}_{\pp}[1] \subseteq \mathcal{U}_{\pp}$. Let $X$ be an object of $\Der(R_{\pp})$ and consider the approximation triangle of $X$ with respect to $(\mathcal{U},\mathcal{V})$ in $\Der(R)$:
		$$U \rightarrow X \rightarrow V \rightarrow U[1].$$
		Localizing this triangle at $\pp$, we see by the uniqueness of approximation triangles that $U \in \mathcal{U}_{\pp}$ and $V \in \mathcal{V}_{\pp}$. This shows that $(\mathcal{U}_{\pp},\mathcal{V}_{\pp})$ is a t-structure in $\Der(R_{\pp})$.

		Finally, let $\Phi \subseteq \Der^c(R)$ be a set witnessing the definability of $\mathcal{V}$, that is, 
		$$\mathcal{V} = \{X \in \Der(R) \mid \Hom{}_{\Der(R)}(f,X) \text{ is surjective for all $f \in \Phi$}\}.$$ 
		Then for any $Y \in \Der(R_{\pp})$ we have by the $\otimes^\mathbf{L}_R$-$\RHom_R$ adjunction that there is a natural isomorphism
		$$\Hom{}_{\Der(R_{\pp})}(f \otimes_R R_{\pp},Y) \simeq \Hom{}_{\Der(R)}(f,Y),$$
		which means that $Y \in \mathcal{V}_{\pp}$ if and only if $\Hom_{\Der(R_{\pp})}(f \otimes_R R_{\pp},Y)$ is surjective for all $f \in \Phi$. Since $f \otimes_R R_{\pp}$ is a map in $\Der^c(R_{\pp})$ for any $f \in \Der^c(R)$, this establishes the definability of $\mathcal{V}_{\pp}$ in $\Der(R_{\pp})$.
\end{proof}
\begin{thm}\label{T:gtc}
	Let $R$ be a commutative ring of weak global dimension at most one. Then the following conditions are equivalent:
\begin{enumerate}
		\item[(i)] there is no $\pp \in \Spec(R)$ such that $\pp R_{\pp}$ is a non-zero idempotent ideal in $R_{\pp}$,
	\item[(ii)] any t-structure on $\Der(R)$ with a definable coaisle is compactly generated.
\end{enumerate}
\end{thm}
\begin{proof}
		Assume $(i)$, and let $(\mathcal{U},\mathcal{V})$ be a t-structure with $\mathcal{V}$ definable. By $(i)$, $R_{\pp}$ is a strongly discrete valuation domain for each $\pp \in \Spec(R)$, and therefore, using Lemma~\ref{L:local}, $(\mathcal{U}_{\pp},\mathcal{V}_{\pp})$ is compactly generated for each $\pp \in \Spec(R)$ by Corollary~\ref{C:tcvd}. By Proposition~\ref{P:compinj}, the subcategories $(\mathcal{V}_{\pp})_n$ of $\ModRpp$ are closed under injective envelopes for any $\pp \in \Spec(R)$ and $n \in \mathbb{Z}$. By the same Proposition, it is enough to show that $\mathcal{V}_n$ is closed under injective envelopes for each $n \in \mathbb{Z}$. Let $M \in \mathcal{V}_n$, and let $E$ be the injective envelope of $M$. For any $\pp \in \Spec(R)$, the module $M_{\pp}$ belongs to $(\mathcal{V}_{\pp})_n \subseteq \mathcal{V}_n$, again using Lemma~\ref{L:local} for the last inclusion. The natural map $\iota: M \rightarrow \prod_{\pp \in \Spec(R)}M_{\pp}$ is a monomorphism. Since $R \rightarrow R_{\pp}$ is a flat ring epimorphism, the injective envelope $E(M_{\pp})$ in $\ModRpp$ is an injective $R$-module. Therefore, we can use the injectivity to extend $\iota$ to a map $\varphi: E \rightarrow \prod_{\pp \in \Spec(R)} E(M_{\pp})$. As $\varphi$ extends $\iota$, and $M$ is essential in $E$, it follows that $\varphi$ is a monomorphism. Therefore, $E$ is a direct summand in $\prod_{\pp \in \Spec(R)} E(M_{\pp})$. But $\prod_{\pp \in \Spec(R)} E(M_{\pp}) \in \mathcal{V}_n$, and thus $E \in \mathcal{V}_n$.

		The converse implication follows again from \cite[Theorem 7.2]{BS}.
\end{proof}
\section{Homological ring epimorphisms versus density}\label{S:nondense}
	Let $R$ be a valuation domain and $\mathbb{X}=(\mathcal{X}_n \mid n \in \mathbb{Z})$ an admissible filtration in $\Spec(R)$. We call such a sequence \EMP{nowhere dense} if the admissible system $\mathcal{X}_n$ is nowhere dense for all $n \in \mathbb{Z}$. Note that $\mathbb{X}$ is a nowhere dense admissible filtration if and only if it is just a nested sequence of nowhere dense admissible systems.
	
	The aim is to show that the coaisles corresponding to nowhere dense admissible filtrations via Theorem~\ref{T02} are precisely those arising from a chain of homological epimorphisms via Proposition~\ref{P:construction}. The starting point is the following classification of homological ring epimorphism from \cite{BS}:
	\begin{thm}\emph{(\cite[Theorem 5.23]{BS})}\label{T:BSepi}
		Let $R$ be a valuation domain. Then there is a bijection between:
			\begin{enumerate}
					\item[(i)] nowhere dense admissible systems $\mathcal{X}$ in $\Spec(R)$, and
				\item[(ii)] epiclasses of homological ring epimorphisms $\lambda: R \rightarrow S$.
			\end{enumerate}
			The bijection $(ii) \rightarrow (i)$ assigns to $\lambda$ the set of all intervals obtained as follows: For each maximal ideal $\nn \in \mSpec(S)$, the composition map $R \xrightarrow{\lambda} S \xrightarrow{\text{can}} S_{\nn}$ is equivalent to the natural map $R \rightarrow R_{\qq}/\pp$ for some interval $[\pp,\qq]$ in $\Spec(R)$ with $\pp$ idempotent. Then $\mathcal{X}$ is the collection of all intervals obtained by going through all maximal ideals of the commutative ring $S$ (\cite[Proposition 5.5(2)]{BS}).
	\end{thm}
	Combining Theorem~\ref{T:BSepi} and Theorem~\ref{T:epiclass}, we see that nowhere dense admissible systems in $\Spec(R)$ correspond to extension-closed bireflective subcategories of $\ModR$. The next step is to compute these subcategories. 
	\begin{lem}\label{L:birefl}
			Let $R$ be a valuation domain and $\mathcal{X}$ be a nowhere dense admissible system in $\Spec(R)$. Then the extension-closed bireflective subcategory $\ModS \simeq \mathcal{B} \subseteq \ModR$ corresponding to the homological epimorphism $\lambda: R \rightarrow S$ via Theorem~\ref{T:epiclass}, which in turn corresponds to $\mathcal{X}$ via Theorem~\ref{T:BSepi}, can be written as follows:
			$$\mathcal{B} = \{M \in \ModR \mid K(\qq,\pp) \otimes_R M \text{ is exact for all }(\qq,\pp) \in \mathcal{G}(\mathcal{X})\}.$$

			Furthermore, we can write $\mathcal{B} = \mathcal{C} \cap \mathcal{D}$, where $\mathcal{C}$ is the cosilting class corresponding to $\mathcal{X}$ via Theorem~\ref{T01}, and $\mathcal{D}$ is the class of those $R$-modules $M$ such that $F_{\qq}(M) \in \ModRqq$ for each gap $(\qq,\pp) \in \mathcal{G}(\mathcal{X})$ (cf. \S 7).
	\end{lem}
	\begin{proof}
	Put	$\mathcal{B}' = \{M \in \ModR \mid K(\qq,\pp) \otimes_R M \text{ is exact for all }(\qq,\pp) \in \mathcal{G}(\mathcal{X})\}$ and let us start by showing that $\mathcal{B}'$ is an extension-closed bireflective subcategory in $\ModR$. Since $\mathcal{X}$ is nowhere dense, the constant sequence $\mathbb{X} = (\mathcal{X}_n \mid n \in \mathbb{Z})$ defined by $\mathcal{X}_n = \mathcal{X}$ for all $n \in \mathbb{Z}$ is an admissible filtration, as the degreewise non-density condition is satisfied trivially. Let $\mathcal{V}$ be the definable coaisle corresponding to $\mathbb{X}$ via Proposition~\ref{P:intervalstocoaisles}. Since $\mathbb{X}$ is constant, $\mathcal{V}_n = \mathcal{V}_{n+1}$ for all $n \in \mathbb{Z}$, and by the definition of $\mathcal{V}$, $\mathcal{V}_n = \mathcal{C} \cap \mathcal{D}$. Furthermore, Corollary~\ref{C:tensorvanish} together with $\mathcal{X}$ being nowhere dense implies that $\mathcal{V}_n = \mathcal{B}'$. By Proposition~\ref{P:moduletheoretic}, $\mathcal{V}_n$ is closed under products, coproducts, extensions, kernels, and since $\mathcal{V}_n = \mathcal{V}_{n-1}$, also under cokernels. Equivalently, $\mathcal{B}'$ is an extension-closed bireflective subcategory of $\ModR$.

	It remains to show that $\mathcal{B}' = \mathcal{B}$, that is, that $\mathcal{B}'$ equals the image of the fully faithful forgetful functor $\ModS \rightarrow \ModR$. Let $\tau: R \rightarrow T$ be a homological epimorphism corresponding to $\mathcal{B}'$ via Theorem~\ref{T:epiclass}, and let $\mathcal{Y}$ be the nowhere dense admissible system corresponding to $\tau$ via Theorem~\ref{T:BSepi}. Since $\mathcal{B}' = \mathcal{B}$ if and only if the ring epimorphisms $\lambda$ and $\tau$ inhabit the same epiclass (Theorem~\ref{T:epiclass}), by Theorem~\ref{T:BSepi} it is enough to show that $\mathcal{Y} = \mathcal{X}$.  For each $[\pp,\qq] \in \mathcal{Y}$ there is a maximal ideal $\mm$ of $T$ such that $T_{\mm} \simeq R_{\qq}/\pp$. Since $T \in \mathcal{B}'$, also $T_{\mm} \in \mathcal{B}'$, and thus $R_{\qq}/\pp \in \mathcal{B}' = \mathcal{C} \cap \mathcal{D}$. Since $\mathcal{B}'$ is bireflective, this implies $\ModRqqpp \subseteq \mathcal{B}'$. Thus by Lemma~\ref{L:mc} $[\pp,\qq]$ has to be contained in some interval from $\mathcal{X}$, and so $\mathcal{Y}$ is a nested subsystem of $\mathcal{X}$. Now let $[\pp,\qq] \in \mathcal{X}$. Since $R_{\qq}/\pp \in \mathcal{B}'$, there is a homological ring epimorphism $\gamma: T \rightarrow R_{\qq}/\pp$. As $R_{\qq}/\pp$ is a local ring, there is a maximal ideal $\mm$ of $T$ such that $\gamma$ factorizes as $T \xrightarrow{\text{can}} T_{\mm} \rightarrow R_{\qq}/\pp$. Then $T_{\mm} \simeq R_{\qq'}/\pp'$ for some $[\pp',\qq'] \in \mathcal{Y}$. Because the latter factorization produces a ring epimorphism $R_{\qq'}/\pp' \rightarrow R_{\qq}/\pp$, the interval $[\pp,\qq] \in \mathcal{X}$ has to be included in $[\pp',\qq]' \in \mathcal{Y}$. But $\mathcal{Y}$ is a nested subsystem of $\mathcal{X}$, and thus $[\pp,\qq] = [\pp',\qq']$ and consequently, $\mathcal{X} = \mathcal{Y}$.
	\end{proof}
	\begin{lem}\label{L:bireflnested}
			Let $R$ be a valuation domain, and $\mathcal{X}_0, \mathcal{X}_1$ be two nowhere dense admissible systems in $\Spec(R)$ such that $\mathcal{X}_0$ is a nested subsystem of $\mathcal{X}_1$. Let $\mathcal{B}_0$ and $\mathcal{B}_1$ be the extension-closed bireflective subcategories corresponding to $\mathcal{X}_0$ and $\mathcal{X}_1$, respectively. Then $\mathcal{B}_0 \subseteq \mathcal{B}_1$.
	\end{lem}
	\begin{proof}
			By Lemma~\ref{L:birefl}, we have for all $i=0,1$ that
			$$\mathcal{B}_i = \mathcal{C}_i \cap \mathcal{D}_i,$$
			where $\mathcal{C}_i$ is the cosilting class corresponding to $\mathcal{X}_i$, and $\mathcal{D}_i$ consists of those modules $M$ such that $F_{\qq}(M) \in \ModRqq$ for all gaps $(\qq,\pp) \in \mathcal{G}(\mathcal{X}_i)$ with $\qq \in \Spec(R)$. Since $\mathcal{X}_i$ is nowhere dense for $i=0,1$, the sequence $(\mathcal{X}_0,\mathcal{X}_1)$ can extended to an admissible filtration by setting $\mathcal{X}_n = \mathcal{X}_0$ for $n<0$ and $\mathcal{X}_n = \mathcal{X}_1$ for $n>1$. Let $\mathcal{V}$ be the definable coaisle corresponding to this admissible filtration by Theorem~\ref{T02}. Then we simply observe that 
$$\mathcal{B}_0 =\mathcal{C}_0 \cap \mathcal{D}_0 = \mathcal{V}_{-1} \subseteq \mathcal{V}_0 = \mathcal{C}_0 \cap \mathcal{D}_1 \subseteq \mathcal{V}_1 = \mathcal{C}_1 \cap \mathcal{D}_1 =\mathcal{B}_1,$$
which establishes the proof.
	\end{proof}
	\begin{thm}\label{T:homepi}
		Let $R$ be a valuation domain and $\mathcal{V}$ be a definable coaisle in $\Der(R)$. Then the two following conditions are equivalent:
			\begin{enumerate}
				\item[(i)] the admissible filtration $\mathbb{X}$ corresponding to $\mathcal{V}$ via Theorem~\ref{T02} is nowhere dense,
				\item[(ii)] $\mathcal{V}$ arises from a chain of homological ring epimorphisms as in Proposition~\ref{P:construction}.
			\end{enumerate}
	\end{thm}
	\begin{proof}
			Let us start with a definable coaisle $\mathcal{V}$ corresponding in the sense of Theorem~\ref{T02} to a nowhere dense admissible filtration $\mathbb{X}$. By Lemma~\ref{L:bireflnested}, the admissible filtration $\mathbb{X}$ induces a sequence of extension-closed bireflective subcategories
			\begin{equation}\label{E:chainbirefl}\cdots \subseteq \mathcal{B}_{n-1} \subseteq  \mathcal{B}_n \subseteq \mathcal{B}_{n+1} \subseteq \cdots,\end{equation}
			and thus, by the discussion in (\ref{SS:homological}), a chain of homological epimorphisms. Let
			$$\mathcal{V'} = \{X \in \Der(R) \mid H^n(X) \in \Cogen(\mathcal{B}_n) \cap \mathcal{B}_{n+1} ~\forall n \in \mathbb{Z}\}$$
			be the definable coaisle induced by this chain via Proposition~\ref{P:construction}. Let $\mathbb{X}'$ be the admissible filtration corresponding to $\mathcal{V}'$ by Theorem~\ref{T02}. Fix an integer $n \in \mathbb{Z}$. If $[\pp,\qq] \in \mathcal{X}_n$, then $R_{\qq}/\pp \in \mathcal{B}_n$, which implies that $R_{\pp}/\qq \in \mathcal{B}_n$, and thus $R_{\pp}/\qq \in \Cogen(\mathcal{B}_n) \cap \mathcal{B}_{n+1}$. This means that $[\pp,\qq]$ is contained in some interval from $\mathcal{X}'_{n}$ by Lemma~\ref{L:mc}(ii). On the other hand, let $[\pp,\qq] \in \mathcal{X}'_n$. Then $R_{\pp}/\qq \in \Cogen(\mathcal{B}_n)$. But by Lemma~\ref{L:birefl} and Proposition~\ref{P:tor}, $\mathcal{B}_n$ is contained in the cosilting class $\mathcal{C}_n$ corresponding to $\mathcal{X}_n$ via Theorem~\ref{T01}, and thus $R_{\pp}/\qq$ belongs to $\mathcal{C}_n$. But as $R_{\pp}/\qq$ is an $R_{\qq}/\pp$-module, it also belongs to $\mathcal{D}_n$ (as defined in \S 7, also see Proposition~\ref{P:sameint}), and whence to $\mathcal{V}_n$ by Lemma~\ref{L:contain}. This implies that $[\pp,\qq]$ is contained in some interval from $\mathcal{X}_n$ again by Lemma~\ref{L:mc}(ii). Using the disjoint property of admissible systems, we showed that $\mathbb{X} = \mathbb{X}'$, and thus $\mathcal{V} = \mathcal{V}'$ by Theorem~\ref{T02}. In particular, $\mathcal{V}$ is induced by a chain of homological epimorphisms.

			For the converse, let $\mathcal{V}$ be a definable coaisle arising from a sequence (\ref{E:chainbirefl}) of extension-closed bireflective subcategories. For each $n \in \mathbb{Z}$, let $\mathcal{X}_n$ be a nowhere dense admissible system corresponding to $\mathcal{B}_n$ via Theorem~\ref{T:BSepi}. First, we claim that the sequence $\mathbb{X} = (\mathcal{X}_n \mid n \in \mathbb{Z})$ is an admissible filtration. Since the admissible systems $\mathcal{X}_n$ are nowhere dense, it is enough to show that $\mathcal{X}_n$ is a nested subsystem of $\mathcal{X}_{n+1}$ for each $n \in \mathbb{Z}$. Let $\mu_n: S_{n+1} \rightarrow S_n$ be a homological epimorphism induced by the inclusion $\mathcal{B}_n \subseteq \mathcal{B}_{n+1}$. If $[\pp,\qq] \in \mathcal{X}_n$ then Theorem~\ref{T:BSepi} implies that there is a ring epimorphism $S_n \rightarrow R_{\qq}/\pp$, and therefore we have a ring epimorphism $\nu: S_{n+1} \rightarrow R_{\qq}/\pp$. Let $\oo = \nu^{-1}[\qq/\pp] \in \Spec(S_{n+1})$, and let $\nn \in \mSpec(S_{n+1})$ be any maximal ideal of $S_{n+1}$ containing $\oo$. Then we have a chain of ring epimorphism as follows:
			$$S_{n+1} \xrightarrow{\text{can}} (S_{n+1})_{\nn} \xrightarrow{\nu \otimes_R (S_{n+1})_{\nn}} R_{\qq}/\pp.$$
			By Theorem~\ref{T:BSepi}, $(S_{n+1})_{\nn}$ is isomorphic to $R_{\qq'}/\pp'$ for some interval $[\pp',\qq'] \in \mathcal{X}_{n+1}$. Then we have a ring epimorphism $R_{\qq'}/\pp' \rightarrow R_{\qq}/\pp$, which implies that $[\pp',\qq']$ contains $[\pp,\qq]$. We showed that $\mathbb{X}$ is an admissible filtration.

			Let $\mathcal{V}'$ be a definable coaisle corresponding to $\mathbb{X}$ via Theorem~\ref{T02}. Then $\mathcal{V}=\mathcal{V}'$ by the first part of the proof.
	\end{proof}
	\begin{remark}\label{R:dense}
		\begin{itemize}
				\item If $R$ is a valuation domain such that $\Spec(R)$ is countable, then each admissible system is nowhere dense. Indeed, if $\mathcal{X}$ is an admissible system on $\Spec(R)$, then $\mathcal{X}$ is countable as well. If there was a dense interval $\xi < \chi$ in $(\mathcal{X},\leq)$, then the restriction of the order to this interval would yield a non-trivial countable totally ordered set which is order-complete and dense. By a classical result of Cantor, any countable dense totally ordered set embeds into $(\mathbb{Q},\leq)$, a contradiction with the order-completeness. Therefore, for any valuation domain with a countable Zariski spectrum, any admissible filtration is nowhere dense, and thus by Theorem~\ref{T:homepi} any definable coaisle is induced by a chain of homological ring epimorphisms.
				\item On the other hand, \cite[Example 5.1 and Remark 5.2]{B2} and Example~\ref{EX0} provide examples of definable coaisles over a valuation domain \EMP{not} induced by a chain of homological ring epimorphisms.
		\end{itemize}
	\end{remark}
	\subsection{Compactly generated t-structures revisited}
	Let $R$ be a valuation domain. Proposition~\ref{P:cg} shows that the admissible filtrations $\mathbb{X}=(\mathcal{X}_n \mid n \in \mathbb{Z})$ corresponding to compactly generated t-structures are precisely those such that $\mathcal{X}_n$ is either empty or a singleton consisting of an interval of the form $[0,\qq_n]$ for some $\qq_n \in \Spec(R)$. On the other hand, such admissible systems are precisely those corresponding to flat ring epimorphisms via Theorem~\ref{T:BSepi}, since flat ring epimorphisms over $R$ coincide with the classical localizations by \cite[Proposition 5.4]{BS}. In this way, we obtain the following result:
	\begin{thm}\label{T:homepicomp}
		Let $R$ be a valuation domain and $\mathcal{V}$ be a definable coaisle in $\Der(R)$. Then the two following conditions are equivalent:
			\begin{enumerate}
				\item[(i)] the t-structure $(\mathcal{U},\mathcal{V})$ is compactly generated,
				\item[(ii)] $\mathcal{V}$ arises from a chain of flat ring epimorphisms as in Proposition~\ref{P:construction}.
			\end{enumerate}
	\end{thm}
	\subsection{Non-degeneracy and unbounded cosilting objects}
	The final goal is to restrict Theorem~\ref{T02} to those definable coaisles, which belong to non-degenerate t-structures. In other words, to identify those definable coaisles, which are induced by a pure-injective cosilting object of $\Der(R)$ (see Corollary~\ref{C:rosie3}). We will show that the right part of the non-degeneracy condition can only be achieved if the coaisle is cohomologically bounded below. On the other hand, we will exhibit in Example~\ref{EX1} a definable coaisle for which the left part of the non-degeneracy condition is achieved non-trivially. In other words, any coaisle induced by a pure-injective cosilting complex over a valuation domain is cohomologically bounded from below, but there are such which are not co-intermediate. In particular, any pure-injective cosilting complex over a valuation domain is cohomologically bounded below, but there are pure-injective cosilting complexes which are \EMP{not} bounded cosilting complexes.
	\begin{lem}\label{L:BW}
		Let $R$ be a valuation domain and $\mathcal{V}$ be a definable coaisle Then $\bigcap_{n \in \mathbb{Z}} \mathcal{V}[n] = 0$ if and only if there is $l \in \mathbb{Z}$ such that $\mathcal{V} \subseteq \Der^{>l}$.
	\end{lem}
	\begin{proof}
			The ``if'' statement is trivial, thus we need to show just the ``only if'' implication. Let $\Theta(\mathcal{V}) = (\mathcal{X}_n \mid n \in \mathbb{Z})$ be the admissible filtration corresponding to $\mathcal{V}$. Since $\bigcap_{n \in \mathbb{Z}} \mathcal{V}[n] = 0$, for each prime $\pp \in \Spec(R)$ there is an integer $m_{\pp}$ such that $\kappa(\pp)[n] \not\in \mathcal{V}$ for any $n>m_{\pp}$. Therefore, if we let $\mathcal{K}_n = \{\pp \in \Spec(R) \mid \kappa(\pp)[-n] \in \mathcal{V}\}$ be the subset of $\Spec(R)$ used in Definition~\ref{D:phipsi} for each $n \in \mathbb{Z}$, we see that $\bigcap_{n \in \mathbb{Z}}\mathcal{K}_n = \emptyset$.

		It is enough to show that there is $l \in \mathbb{Z}$ such that $\mathcal{K}_l = \emptyset$. Indeed, then necessarily $\mathcal{X}_l = \emptyset$, and thus $\mathcal{V} \subseteq \Der^{>l}$ by Theorem~\ref{T02}. Towards contradiction, suppose that $\mathcal{K}_n \neq \emptyset$ for all $n \in \mathbb{Z}$, and choose for each $n > 0$ a prime ideal $\pp_n \in \mathcal{K}_{-n}$. 

		We claim that the sequence $(\pp_n \mid n > 0)$ contains an infinite monotone subsequence. This follows by adapting the classical Bolzano-Weierstrass Theorem from the theory of metric spaces to our situation. Indeed, let $B \subseteq \mathbb{Z}_{>0}$ be the set of all those positive integers $b$ such that $\pp_n \subsetneq \pp_b$ for all $n>b$. If $B$ is an infinite set, then the subsequence $(\pp_b \mid b \in B)$ is clearly strictly decreasing, and we are done. If otherwise $B$ is finite, let $b \in B$ be its maximal element. Define an increasing sequence $k_1,k_2,k_3,\ldots$ of positive integers by the following induction. Set $k_1 = b+1$. For each $m > 1$, we have by induction that $k_{m-1} \geq k_1 > b$, and thus $k_{m-1} \not\in B$. Therefore, there is $k_m > k_{m-1}$ such that $\pp_{k_{m-1}} \subseteq \pp_{k_m}$. In this way, we have defined an increasing subsequence $(\pp_{k_n} \mid n>0)$, establishing the claim.

			Let $\pp$ be the limit of the monotone subsequence of $(\pp_{k_n} \mid n > 0)$ obtained in the previous paragraph, that is, $\pp$ is either the union or the intersection of such sequence, depending on whether the subsequence is increasing or decreasing. By the proof of Lemma~\ref{L:nogaps}, and since $\{\pp_{k_m} \mid m \geq -n\} \subseteq \mathcal{K}_n$, we have $\pp \in \mathcal{K}_n$ for all $n<0$. Therefore, $\pp \in \bigcap_{n \in \mathbb{Z}} \mathcal{K}_n$, which is a contradiction.
	\end{proof}
	\begin{cor}\label{C:bbelow}
		Let $R$ be a valuation domain. Then any pure-injective cosilting object in $\Der(R)$ is cohomologically bounded below.
	\end{cor}
	\begin{proof}
		Let $(\mathcal{U},\mathcal{V}) = (\Perp{\leq 0}C,\Perp{>0}C)$ be the t-structure in $\Der(R)$ induced by $C$. By Theorem~\ref{T:rosie2} and Corollary~\ref{C:rosie3}, $(\mathcal{U},\mathcal{V})$ is a non-degenerate t-structure such that $\mathcal{V}$ is definable. Therefore, there is $l \in \mathbb{Z}$ such that $\mathcal{V} \subseteq \Der^{>l}$ by Lemma~\ref{L:BW}. Since $C \in \mathcal{V}$, the claim follows.
	\end{proof}
	We conclude the paper with several examples of definable coaisles. Examples \ref{EX3} and \ref{EX2} illustrate that for the second part of the non-degeneracy condition, it is not enough to consider the ``support'' sets $\mathcal{K}_n$, and that it is also not enough to assume that the smallest admissible system containing $\mathcal{X}_n$ for all $n \in \mathbb{Z}$ is the maximal one, that is $\{[0,\mm]\}$. The promised Example~\ref{EX1} exhibits a t-structure induced by a non-bounded cosilting complex.
\begin{example}\label{EX3}
		Let $R$ be a valuation domain such that $\Spec(R) = \{0,\mm\}$ and such that $\mm = \mm^2$. Such a valuation domain can be constructed by the means of \cite[\S II, Theorem 3.8]{FS} with the value group chosen for example as $\mathbb{R}$. Consider the admissible filtration $\mathbb{X}$ defined as follows:
		$$
		\mathcal{X}_n = \begin{cases}
			\emptyset, & n<0 \\
			\{[0,0],[\mm,\mm]\}, & n \geq 0.
		\end{cases}
		$$
			Then $\mathbb{X}$ is nowhere dense and corresponds via Theorem~\ref{T:homepi} to a chain of homological epimorphism of the following form
		$$\cdots \leftarrow 0 \leftarrow 0 \leftarrow Q \times R/\mm \leftarrow Q \times R/\mm  \leftarrow Q \times R/\mm  \leftarrow \cdots$$
		We claim that the corresponding t-structure $(\mathcal{U},\mathcal{V})$ is \textit{not} non-degenerate. Note that clearly $\bigcap_{n \in \mathbb{Z}}\mathcal{V}[n]=0$. Set $\mathcal{L} = \bigcap_{n \in \mathbb{Z}}\mathcal{U}[n]$. Since 
		$$\mathcal{L} = \Perp{\mathbb{Z}}\mathcal{V} = \{X \in \Der(R) \mid \Hom{}_{\Der(R)}(X,V[i])=0 ~\forall V \in \mathcal{V}, i \in \mathbb{Z}\},$$ 
		we have that $\mathcal{L} = 0$ if and only if $(\Perp{\mathbb{Z}}\mathcal{V})\Perp{0} =  \Der(R)$. But clearly $\mathcal{V} \subseteq \Der(Q \times R/\mm)$, where the derived category of the homological epimorphism $R \rightarrow Q \times R/\mm$ is viewed as a full subcategory of $\Der(R)$. Since $\Der(Q \times R/\mm)$ is a coaisle of a stable t-structure in $\Der(R)$ (see e.g. \cite[5.9]{KrL}), we have $(\Perp{\mathbb{Z}}\mathcal{V})\Perp{0} \subseteq \Der(Q \times R/\mm)$, and hence $\mathcal{L} \neq 0$. Therefore, we have that $\bigcap_{n \in \mathbb{Z}}\mathcal{U}[n] \neq 0$, while $\bigcup_{n \in \mathbb{Z}} \mathcal{K}_n = \mathcal{K}_0 = \Spec(R)$.
\end{example}

	\begin{example}\label{EX1}
		Let $R$ be a valuation domain such that 
		$$\Spec(R) = \{0 = \qq{}_0 \subsetneq \qq{}_1 \subsetneq \qq{}_2 \subsetneq \cdots \subsetneq \qq{}_n \subsetneq \cdots \subseteq \mm\}.$$
			Such a valuation domain can be constructed again with the use of \cite[\S II, Theorem 3.8]{FS}, the value group can be chosen as $\mathbb{Z}^{(\omega)}$ with the lexicographic order, and since the maximal ideal $\mm$ is the union of a strictly increasing sequence of primes, it is necessarily idempotent, see Lemma~\ref{L:VD}(iv). Consider the admissible filtration $\mathbb{X}$ defined as follows:
		$$
		\mathcal{X}_n = \begin{cases}
			\emptyset, & n<0 \\
			\{[0,\qq_n],[\mm,\mm]\}, & n \geq 0.
		\end{cases}
		$$
			Then $\mathbb{X}$ is nowhere dense and corresponds via Theorem~\ref{T:homepi} to a chain of homological epimorphism of the following form
		$$\cdots \leftarrow 0 \leftarrow 0 \leftarrow Q \times R/\mm \leftarrow R_{\qq_1} \times R/\mm \leftarrow R_{\qq_2} \times R/\mm \leftarrow \cdots$$
			We claim that  the corresponding t-structure $(\mathcal{U},\mathcal{V})$ \textit{is} non-degenerate. Indeed, clearly $\bigcap_{n \in \mathbb{Z}}\mathcal{V}[n]=0$. Set $\mathcal{L}=\bigcap_{n \in \mathbb{Z}} \mathcal{U}[n]$ and let us show that $\mathcal{L} = 0$. Fix $L \in \mathcal{L}$ Since $\Der(R_{\qq_n})^{\geq n} \subseteq \mathcal{V}$, then $L \otimes_R R_{\qq_n} = 0$ in $\Der(R)$ for any $n \in \mathbb{Z}$. Therefore, $H^n(L)$ is annihilated by $\mm$ for all $n \in \mathbb{Z}$, and since $\mm$ is flat, this means that $L \otimes_R \mm = 0$ in $\Der(R)$. By \cite[5.9]{KrL}, there is a triangle
			$$L \otimes_R \mm \rightarrow L \rightarrow L \otimes_R^{\mathbf{L}} R/\mm \rightarrow L \otimes_R \mm[1],$$
			and therefore $L \simeq L \otimes_R^{\mathbf{L}} R/\mm  \in \Der(R/\mm)$. But since $R/\mm \in \mathcal{V}$, this implies that $L = 0$.

		We showed that $(\mathcal{U},\mathcal{V})$ is non-degenerate, but since $\mathcal{V}_n \subseteq \operatorname{Mod-(R_{\qq_{n+1}}\times R/\mm)}$ for all $n \geq 0$, $(\mathcal{U},\mathcal{V})$ is not co-intermediate. Therefore, $(\mathcal{U},\mathcal{V})$ is induced by a pure-injective cosilting complex which is not bounded.
	\end{example}
\begin{example}\label{EX2}
		Let $R$ be the same valuation domain as in Example~\ref{EX1}, and consider the admissible filtration $\mathbb{X}$ defined as follows:
		$$
		\mathcal{X}_n = \begin{cases}
			\emptyset, & n<0 \\
			\{[0,\qq_n]\}, & n \geq 0.
		\end{cases}
		$$
			Then $\mathbb{X}$ is nowhere dense and corresponds via Theorem~\ref{T:homepi} to a chain of homological epimorphism of the following form
		$$\cdots \leftarrow 0 \leftarrow 0 \leftarrow Q \leftarrow R_{\qq_1} \leftarrow R_{\qq_2} \leftarrow \cdots$$
		We claim that  the corresponding t-structure $(\mathcal{U},\mathcal{V})$ is \textit{not} non-degenerate. Set $\mathcal{L}=\bigcap_{n \in \mathbb{Z}} \mathcal{U}[n]$ and let us show that $\mathcal{L} = \Der(R/\mm) \neq 0$. By an argument similar to Example~\ref{EX1}, we see that $\mathcal{L} \subseteq \Der(R/\mm)$. Let $X \in \Der(R/\mm)$, and let us show that $\Hom_{\Der(R)}(X,\mathcal{V}) = 0$. Then $X$ is quasi-isomorphic to a split complex of $R/\mm$-modules, and therefore we can without loss of generality assume that $X$ is a stalk complex, say $X = R/\mm^{(\varkappa)}[-n]$ for some cardinal $\varkappa$. For any $V \in \mathcal{V}$ we have $\Hom_{\Der(R)}(X,V) \simeq \Hom_{\Der(R)}(X,\tau^{\leq n}V)$, where $\tau^{\leq n}V$ is the soft truncation of $V$ to degrees $\leq n$. Since $V \in \mathcal{V}$, we have that $\tau^{\leq n}V \in \Der(R_{\qq_{n+1}})$. But $X \otimes_R R_{\qq_{n+1}} = 0$, and thus $\Hom_{\Der(R)}(X,\tau^{\leq n}V) = 0$.

		Then $\mathbb{X}$ corresponds to a t-structure which is not non-degenerate, even though the smallest admissible system containing $\mathcal{X}_n$ as a nested subsystem for all $n \in \mathbb{Z}$ is clearly the maximal admissible system $\{[0,\mm]\}$.
\end{example}
\bibliographystyle{abbrv}

\end{document}